\documentclass[10pt,letterpaper]{article}
\usepackage[margin=1in]{geometry}
\makeatletter
\def\thm@space@setup{%
\thm@preskip=1em \thm@postskip=0pt
}
\makeatother
\usepackage{import}
\usepackage{amsmath, amsthm, amssymb, bbold, mathtools}
\usepackage{footnote}
\usepackage{graphicx}
\usepackage{acronym}
\usepackage{rotating}

\usepackage{subfig}
\usepackage{import}
\usepackage{cite}
\usepackage{enumerate}
\usepackage{hyperref}
\usepackage[shortlabels]{enumitem}

\numberwithin{equation}{section}
\newtheorem{assumption}{Assumption}[section]

\newtheorem{proposition}{Proposition}[section]
\newtheorem{lemma}{Lemma}[section]

\newtheorem{definition}{Definition}[section]
\newtheorem{theorem}{Theorem}[section]

\newtheorem{example}{Example}[section]
\newtheorem{remark}{Remark}[section]

\acrodef{ouq}[OUQ]{optimal uncertainty quantification}
\acrodef{dro}[DRO]{distributionally robust optimization}
\acrodef{saa}[SAA]{sample average approximation}
\acrodef{ldt}[LDT]{large deviation theory}
\acrodef{ldp}[LDP]{large deviation principle}
\acrodef{lln}[LLN]{law of large numbers}
\acrodef{kl}[KL]{Kullback-Leibler}
\acrodef{iid}[{i.i.d.\ \!\!}]{independent and identically distributed}
\acrodef{qp}[QP]{quadratic program}
\acrodef{qcqp}[QCQP]{quadratically constrained quadratic program}
\acrodef{vod}[VoD]{value of data}
\acrodef{fl}[FL]{Fenchel-Legendre}

\usepackage{tikz}
\usepackage{pgfplots}
\usetikzlibrary{plotmarks,arrows,shapes,backgrounds}
\usetikzlibrary{automata}
\usetikzlibrary{matrix,positioning}
\usetikzlibrary{decorations.pathmorphing,calc}
\usepackage{tikz-3dplot}
\usetikzlibrary{patterns}
\usepackage{framed}
\usepackage{bm}

\usepackage{xcolor,calc}
\newcommand{\noopsort}[1]{}

\newcommand{\drop}[1]{}

\usepackage{cleveref}
\newcommand{\Indic}[1]{\mathsf{1}_{#1}}
\newcommand{\inprod}[2]{\ensuremath{\left\langle{#1}\vphantom{\big|},\vphantom{\big|}{#2}\right\rangle}}

\newcommand{\tnorm}[1]{\Vert#1\Vert}

\newcommand{\D}[2]{D\!\left(#1 \Vert #2 \right)}
\newcommand{\Dc}[2]{D_c\!\left(#1 \Vert #2 \right)}

\newcommand{\tpose}{^\top}
\newcommand{\iprod}[2]{\langle #1, #2 \rangle}

\newcommand{\A}{A}

\newcommand{\xivec}{{\xi}_{[T]}}

\renewcommand{\tfrac}[2]{{#1}/{#2}}

\newcommand{\mc}{\mathcal}
\newcommand{\mb}{\mathbb}
\renewcommand{\emph}{\textbf}
\def\d{\mathrm{d}}
\def\N{\mathbb{N}}
\def\Re{\mathbb{R}}
\def\SS{\mathbb{S}}

\DeclareMathOperator{\cl}{cl}
\DeclareMathOperator{\interior}{int}

\DeclareMathOperator{\dom}{dom}

\providecommand{\keywords}[1]{\textbf{\textit{Keywords---}} #1}
\usepackage[bottom]{footmisc} %footnote
\allowdisplaybreaks[1]

\title{\bf A Pareto Dominance Principle\\ for Data-Driven Optimization}
\author{Tobias Sutter$^1$ \and Bart P.G.\ \mbox{Van Parys}$^2$ \and Daniel Kuhn$^3$}
\date{\small{
    $^1$Department of Computer Science, University of Konstanz, {tobias.sutter@uni-konstanz.de}\\
    $^2$Sloan School of Management, Massachusetts Institute of Technology, {vanparys@mit.edu}\\
    $^3$Risk Analytics and Optimization Chair, Ecole Polytechnique F\'ed\'erale de Lausanne,{ daniel.kuhn@epfl.ch}\\%
    [2ex]}
  \today
}
\pgfplotsset{compat=1.16}
\usepgfplotslibrary{fillbetween}
\pgfdeclarepatternformonly{north east lines wide}%
{\pgfqpoint{-1pt}{-1pt}}%
{\pgfqpoint{10pt}{10pt}}%
{\pgfqpoint{9pt}{9pt}}%
{
  \pgfsetlinewidth{0.4pt}
  \pgfpathmoveto{\pgfqpoint{0pt}{0pt}}
  \pgfpathlineto{\pgfqpoint{9.1pt}{9.1pt}}
  \pgfusepath{stroke}
}

\begin{document}
\setlength{\baselineskip}{1.5em}

\maketitle

\begin{abstract}
We propose a statistically optimal approach to construct data-driven decisions for stochastic optimization problems. Fundamentally, a data-driven decision is simply a function that maps the available training data to a feasible action. It can always be expressed as the minimizer of a surrogate optimization model constructed from the data. The quality of a data-driven decision is measured by its out-of-sample risk. An additional quality measure is its out-of-sample disappointment, which we define as the probability that the out-of-sample risk exceeds the optimal value of the surrogate optimization model. The crux of data-driven optimization is that the data-generating probability measure is unknown. An ideal data-driven decision should therefore minimize the out-of-sample risk simultaneously with respect to {\em every} conceivable probability measure (and thus in particular with respect to the unknown true measure). Unfortunately, such ideal data-driven decisions are generally unavailable. 
This prompts us to seek data-driven decisions that minimize the in-sample risk subject to an upper bound on the out-of-sample disappointment-again simultaneously with respect to every conceivable probability measure.
We prove that such Pareto dominant data-driven decisions exist under conditions that allow for interesting applications: the unknown data-generating probability measure must belong to a parametric ambiguity set, and the corresponding parameters must admit a sufficient statistic that satisfies a large deviation principle. If these conditions hold, we can further prove that the surrogate optimization model generating the optimal data-driven decision must be a distributionally robust optimization problem constructed from the sufficient statistic and the rate function of its large deviation principle. This shows that the optimal method for mapping data to decisions is, in a rigorous statistical sense, to solve a distributionally robust optimization model. Maybe surprisingly, this result holds irrespective of whether the original stochastic optimization problem is convex or not and holds even when the training data is non-i.i.d. As a byproduct, our analysis reveals how the structural properties of the data-generating stochastic process impact the shape of the ambiguity set underlying the optimal distributionally robust optimization model.
\end{abstract}


%%%%%%%%%%%%%%%%%
%% SEC. Introduction
%%%%%%%%%%%%%%%%%

\section{Introduction}

A fundamental challenge in data-driven decision-making is to construct estimators for the optimal solutions of stochastic optimization problems based on limited training data. We address this challenge within a well-defined framework that is sufficiently general to support a broad spectrum of applications. The primitives of this framework are a stochastic optimization problem representing the ground truth against which the estimators will be assessed, a family of probability measures that capture prior structural knowledge and a stochastic process that generates training samples. The stochastic optimization problem minimizes a generic objective function that depends on the probability measure governing the uncertain problem parameters. Examples of such objective functions include the expected value or some risk measure of an uncertain loss function, the conditional expectation of an uncertain loss function given contextual covariates or the long-run average expected cost of a parametric control policy etc. The crux of data-driven decision-making is that the probability measure underlying the stochastic optimization problem is unknown. Throughout this paper we assume, however, that this probability measure is known to belong to a parametric family of the form~$\{\mathbb P_\theta:\theta\in\Theta\}$. In addition, we assume that we have access to a finite trajectory of an exogenous stochastic process, which generates training samples that provide statistical information about~$\theta$. Examples of stochastic processes to be studied in this paper include processes of \ac{iid} random variables on a finite state space, finite-state Markov chains, different classes of vector autoregressive processes or \ac{iid} processes with parametric distribution functions, but many other examples are conceivable. These examples highlight that we actually allow the training samples to display serial dependence.

It is convenient to embed the original stochastic optimization problem into a parametric family of problems that are obtained by replacing the unknown true probability measure with any~$\mathbb P_\theta$, $\theta\in\Theta$. The resulting stochastic optimization problems can be concisely represented as $\min_{x\in X} c(x,\theta)$, $\theta\in\Theta$, where $X$ denotes the feasible set and~$c(x,\theta)$ stands for the risk or cost of the decision~$x$ under the probability measure~$\mathbb P_\theta$. As the parameter~$\theta$ corresponding to the true probability measure is unknown, however, it is unclear which problem instance should be solved. We thus have no choice but to solve a data-driven surrogate optimization problem~$\min_{x\in X} \widehat c_T(x)$, whose objective function~$\widehat c_T$ is constructed independently of~$\theta$ from~$T$ training samples. In the following, we denote by~$\widehat x_T$ an optimal solution of the surrogate optimization problem, which is necessarily a function of the~$T$ training samples, too. For the sake of a succinct terminology, we henceforth refer to~$\widehat c_T$ as a {\em data-driven predictor} because it predicts the risk of any decision~$x$ in view of the available data. Similarly, we refer to~$\widehat x_T$ as a {\em data-driven prescriptor} because it prescribes a feasible decision in view of the available data. We emphasize that a data-driven prescriptor could be essentially {\em any} function that maps the available training data to a feasible decision. Indeed, it is easy to convince oneself that any such function can be expressed as the minimizer of a carefully constructed surrogate optimization problem. The main goal of this paper is to design---in a rigorous statistical sense---an `optimal' surrogate optimization problem, which is equivalent to finding `optimal' data-driven predictors and prescriptors.

The quality of a data-driven prescriptor~$\widehat x_T$ under~$\mathbb P_\theta$ is unequivocally measured by its out-of-sample risk~$c(\widehat x_T,\theta)$. As the true~$\theta$ is unknown, an ideal prescriptor would have to minimize the out-of-sample risk simultaneously for all~$\theta\in\Theta$ and thus necessarily also for the unknown true~$\theta$. Unfortunately, such ideal data-driven prescriptors are unavailable for non-trivial stochastic optimization problems. To circumvent this difficulty, we recall that any data-driven prescriptor~$\widehat x_T$ is induced by some data-driven predictor~$\widehat c_T$, and we define~$\widehat c_T(\widehat x_T)$ as the in-sample risk of~$\widehat x_T$, which is a function of the training samples alone and therefore accessible to the decision-maker. Note, however, that~$\widehat x_T$ may be induced by many different data-driven predictors~$\widehat c_T$ and that our definition of the in-sample risk depends on the particular choice of~$\widehat c_T$. In particular, $\widehat c_T$ could be shifted by a constant without affecting~$\widehat x_T$. Minimizing the in-sample risk instead of the out-of-sample risk is therefore nonsensical unless we restrict the choice of~$\widehat c_T$. To this end, we define the {\em out-of-sample disappointment} of~$\widehat x_T$ under~$\mathbb P_\theta$ as the probability that the out-of-sample risk strictly exceeds the in-sample risk of~$\widehat x_T$. This means that if the out-of-sample disappointment is high, then the predicted risk of~$\widehat x_T$ is likely to underestimate its true risk, which lulls the decision-maker into a false sense of security and invariably leads to disappointment in out-of-sample tests. Note that the true out-of-sample disappointment is again {\em in}accessible to the decision-maker because it depends on the unknown parameter~$\theta$. By construction, however, the out-of-sample disappointment decreases as $\widehat c_T$ increases. This reasoning motivates us to formulate an optimization problem that finds data-driven predictor-prescriptor pairs with an optimal trade-off between in-sample risk and out-of-sample disappointment. As each data-driven predictor encodes itself a surrogate optimization problem, any optimization problem over~$\widehat c_T$ and $\widehat x_T$ constitutes indeed a {\em meta-optimization problem}, that is, an optimization problem over surrogate optimization problems.

To describe the envisaged meta-optimization problem more precisely, we define the asymptotic in-sample risk of a data-driven predictor~$\widehat c_T$ and the corresponding data-driven prescriptor~$\widehat x_T$ under~$\mathbb P_\theta$ as
\[
    \lim_{T\to\infty} \mb E_\theta \left[\widehat c_T(\widehat x_T) \right],
\]
and we define the asymptotic decay rate of the out-of-sample disappointment of~$\widehat c_T$ and~$\widehat x_T$ under~$\mathbb P_\theta$ as
\[
    \limsup_{T\to\infty} \frac{1}{T} \log \mb P_\theta [ c(\widehat x_T, \theta) > \widehat c_T(\widehat x_T)].
\]
Both of these statistical performance indicators, which are well-defined under mild regularity conditions, depend on the unknown parameter~$\theta$. We therefore intend to optimize them simultaneously for all~$\theta\in\Theta$, which leads to a multi-objective optimization problem. This problem minimizes the asymptotic in-sample risk simultaneously for all~$\theta\in\Theta$ under the condition that the asymptotic decay rate of out-of-sample disappointment is smaller than $r\ge 0$ for every~$\theta\in\Theta$. The risk-aversion parameter $r$ is chosen by the decision-maker. Even though it plays the role of a hyperparameter, it is directly interpretable thanks to its link to the out-of-sample disappointment.
Multi-objective optimization problems typically only admit Pareto optimal solutions, that is, feasible solutions that are not Pareto dominated by any other feasible solution. Maybe surprisingly, however, we will see that the proposed meta-optimization problem sometimes admits Pareto dominant solutions, that is, feasible solutions that Pareto dominate all other feasible solutions. Thus, such Pareto dominant solutions simultaneously minimize all objective functions of the meta-optimization problem. Moreover, if they exist, these solutions are available in closed-form and admit an intuitive interpretation.

Data-driven predictors and prescriptors are essentially arbitrary functions of the available~$T$ training samples. Processing or even storing such functions might easily become impractical for large~$T$. A natural approach to simplify the proposed meta-optimization problem is to compress the observation history of the training samples into a statistic~$\widehat S_T$ of constant dimension and to restrict attention to {\em compressed} data-driven predictors and prescriptors that depend on the training samples only indirectly through~$\widehat S_T$. The resulting restricted meta-optimization problem is often easier to handle than the original meta-optimization problem.

We are now ready to summarize the main contributions of this work.

\begin{enumerate}
    \setlength\itemsep{0em}
\item We prove that if the statistic~$\widehat S_T$ satisfies a large deviation principle, then the restricted meta-optimiza\-tion problem over all compressed data-driven predictors and prescriptors admits a Pareto dominant solution. Moreover, the optimal data-driven predictor evaluates, for every fixed decision~$x$, the worst case of the risk~$c(x,\theta)$ across all~$\theta$ in a ball of radius~$r$ around~$\widehat S_T$, where the discrepancy between~$\widehat S_T$ and~$\theta$ is measured via the rate function of the large deviation principle at hand. The surrogate optimization problem induced by this optimal predictor thus represents a distributionally robust optimization problem, and the radius~$r$ of the underlying ambiguity set coincides with the upper bound on the decay rate of the out-of-sample disappointment enforced by the restricted meta-optimization problem. 

\item We demonstrate that the restricted meta-optimization problem and its Pareto dominant solution are invariant under homeomorphic coordinate transformations of the statistic~$\widehat S_T$ as well as the distribution family~$\{\mathbb P_\theta:\theta\in\Theta\}$. This implies that the chosen parametrizations, which are invariably somewhat arbitrary, have no impact on how the optimal data-driven prescriptor maps the raw data to decisions.

\item We prove that if the set~$\{\mathbb P_\theta:\theta\in\Theta\}$ represents an exponential family with sufficient statistic~$\widehat S_T$ and if~$\widehat S_T$ satisfies a large deviation principle, then compressing the training samples into~$\widehat S_T$ destroys no useful statistical information, and the original meta-optimization problem is indeed equivalent to the restricted meta-optimization problem. Thus, the original meta-optimization problem also admits a Pareto dominant solution that has a distributionally robust interpretation. This result establishes a separation principle that enables a decoupling of estimation and optimization, and it can be viewed as a non-trivial extension of the celebrated Rao-Blackwell theorem~\cite{ref:Rao-45, ref:Blackwell-47} to data-driven decision problems.

\item We explicitly derive the optimal data-driven predictors corresponding to different data-generating stochastic processes including finite-state \ac{iid} processes, finite-state Markov chains, two different classes of autoregressive processes as well as \ac{iid} processes with parametric distribution functions.
\end{enumerate}

Our results suggest that the optimal method for mapping data to decisions is, in a rigorous statistical sense, to solve a distributionally robust optimization model. As we will see, this conclusion persists irrespective of whether the original stochastic optimization problem is convex or not, and it persists even when the training data is non-i.i.d. As a byproduct, our analysis reveals how the structural properties of the data-generating stochastic process impact the shape of the ambiguity set underlying the optimal (distributionally robust) surrogate optimization problem. This paper therefore generalizes the preliminary results for \ac{iid} training samples on a finite state space reported in~\cite{ref:vanParys:fromdata-17}. In fact, we will demonstrate through a running example that these results emerge as a special case of a significantly more general theory of data-driven decision-making.

%*********************************
The existing literature on data-driven stochastic optimization is vast. Arguably the most popular approach is the \ac{saa}, which replaces the unknown true probability distribution of the uncertain parameters in the problem's objective function with the empirical distribution corresponding to the training samples. The asymptotic properties of the resulting \ac{saa} problem are well understood if the training samples are \ac{iid}; see, {\em e.g.}, \cite{ref:shapiro-89,ref:shapiro-90,ref:Shapiro-91,ref:Shaprio-93,shapiro2014lectures,ref:King-91,ref:King-93}. In particular, the optimal value of the \ac{saa} problem is known to be strongly consistent and asymptotically normal \cite[Sections~5.1.1--5.1.2]{shapiro2014lectures}, which facilitates a rigorous probabilistic error analysis that yields increasingly accurate confidence bounds as the sample size grows. If the sample size is small relative to the number of uncertain problem parameters, however, then the optimal solution of the \ac{saa} problem tends to display an excellent in-sample performance alongside a poor out-of-sample performance. This phenomenon can be interpreted as an overfitting effect, which is sometimes referred to as the optimization bias \cite{shapiro2003} or the optimizer's curse \cite{smith2006optimizer}. Data-driven \ac{dro} has been widely championed as an effective means to combat this phenomenon. It seeks a decision that minimizes the worst-case risk with respect to all probability distributions in an ambiguity set constructed from the training samples. If one can guarantee that the unknown true distribution belongs to the ambiguity set with high probability, then the optimal value of the \ac{dro} problem provides an upper confidence bound on the out-of-sample performance of its optimal solution. Out-of-sample guarantees of this kind were first obtained for a Chebyshev ambiguity set that contains all probability distributions whose mean vectors and covariance matrices are close to the empirical mean and the empirical covariance matrix \cite{delage2010distributionally}. As the sample size grows, the moment estimates become increasingly accurate, in which case this Chebyshev ambiguity set reduces to the family of all probability distributions that share the same first- and second-order moments as the unknown true distribution. Since this family contains distributions with strikingly different shapes (and not only the true distribution), the optimal value of the corresponding \ac{dro} problem fails to be asymptotically consistent. Pertinent out-of-sample guarantees have also been established for ambiguity sets containing all probability distributions that are close to the empirical distribution with respect to some information divergence \cite{bental2013uncertain}, for ambiguity sets containing all distributions that pass a statistical goodness-of-fit test against the observed training data \cite{ref:Bertsimas-2018} or for ambiguity sets containing all distributions that are close to the empirical distribution with respect to some Wasserstein distance \cite{ref:Peyman-18,ref:DROtutorial-19}. If these ambiguity sets are scaled sufficiently slowly, then the corresponding \ac{dro} problems can be rendered asymptotically consistent without compromising their out-of-sample guarantees. By leveraging ideas from empirical likelihood theory, it has recently been shown that significantly tighter out-of-sample bounds can be obtained by relaxing the requirement that the ambiguity set must contain the unknown true distribution with high probability \cite{ref:Lam-16,ref:Duchi-21}. 

In view of the many ambiguity sets permeating the extant literature, it is natural to wonder which ones of them offer optimal statistical guarantees. For example, given an ambiguity set with a prescribed `shape' determined by the choice of a specific information divergence or probability metric, it is natural to seek the smallest radius for which the corresponding \ac{dro} problem offers an upper confidence bound on the original stochastic optimization problem with a desired significance level. The scaling of the optimal radius with respect to the sample size~$T$ is indeed known both for divergence ambiguity sets \cite{ref:Lam-16,ref:Duchi-21} as well as for Wasserstein ambiguity sets \cite{ref:Blanchet_2019,gao2020finitesample}. A more challenging task than merely tuning the size would be to tune the size and the shape of the ambiguity set simultaneously. The study of optimal ambiguity sets was pioneered in~\cite{ref:Gupta-19}, where the smallest convex ambiguity sets that satisfy a Bayesian robustness guarantee are identified under certain convexity assumptions about the stochastic optimization problem.

In addition, ambiguity sets that offer optimal statistical guarantees in view of the central limit theorem are investigated in \cite{ref:Lam-16}. In this case the optimal ambiguity sets constitute carefully scaled Burg-entropy divergence balls centered at the empirical distribution. Recently it has been shown that if the training samples are i.i.d., then among {\em all} data-driven decisions whose out-of-sample risk is dominated by their in-sample risk with high confidence, the decision with the lowest in-sample risk can be computed by solving a \ac{dro} problem with a relative entropy ambiguity set centered at the empirical distribution \cite{ref:vanParys:fromdata-17}. This result indicates that, at least in simple stylized settings, data-driven \ac{dro} provides an optimal approach for mapping data to decisions. In this paper we extend the main results of~\cite{ref:vanParys:fromdata-17} to more general (not necessarily risk-neutral) stochastic optimization problems, more general (not necessarily finitely supported) parametric distribution families and more general (not necessarily i.i.d.) data-generating stochastic processes. As a byproduct of our general theory of data-driven decision-making, we discover several new \ac{dro} schemes that are statistically optimal for different structures of the data-generating stochastic process. %For example, we present the first data-driven \ac{dro} scheme for Markovian data, which offers (optimal) statistical guarantees based on a single trajectory of serially dependent training samples.  
Our theory thus provides practical guidance for choosing the best decision model in different data-driven decision situations.
We also stress that \cite{ref:vanParys:fromdata-17} assumes the predictors and prescriptors to depend on the training data only indirectly through the empirical distribution. Here, we do not impose such an implicit structure. Instead, we consider a much larger class of prescriptors that essentially depend on the training data in an arbitrary manner.

All statistical guarantees reviewed so far rely indeed on the assumption that the training samples are~i.i.d. Moreover, the literature on data-driven \ac{dro} with non-\ac{iid} data is remarkably scarce. We are only aware of three recent papers addressing this topic. First, if the training samples are generated by a fast mixing process, then asymptotic confidence intervals for the optimal value of a stochastic optimization problem can be obtained by solving \ac{dro} problems with divergence ambiguity sets \cite{ref:Duchi-21}. However, the resulting confidence bounds depend on the unknown probability distribution and are therefore primarily of theoretical interest. In addition, data-driven \ac{dro} models with Wasserstein ambiguity sets constructed from training samples following an autoregressive process are proposed in~\cite{ref:Dou-19}. While these ambiguity sets offer rigorous out-of-sample guarantees, their shapes are chosen ad hoc. Finally, distributionally robust Markov decision processes with Wasserstein ambiguity sets for the uncertain transition kernel are developed in~\cite{ref:derman2020distributional}. In this case the training dataset consists of multiple \ac{iid} trajectories of serially correlated states, which may be difficult to acquire in practice. In contrast to all of these approaches, we devise here a principled approach to generate statistically optimal data-driven decisions based on a {\em single} trajectory of the data-generating process.

While this paper was under review, its main results were extended along several dimensions. For example, if the training data is generated by a Markov chain, then the statistically optimal DRO models derived in Section~\ref{sssec:Markovchain} of this paper give rise to high-dimensional non-convex optimization problems. An efficient Frank-Wolfe algorithm to solve these problems is developed in \cite{ICML-Li-21}. In addition, a critical assumption of this paper is that the training and the test data are generated by the same stochastic process. This assumption is relaxed in \cite{ref:Sutter:NeurIPS-21}, where the large deviation-type results of this paper for i.i.d.~data are combined with the principle of minimum discriminating information to address data-driven decision problems suffering from a distribution shift. Another basic assumption of this paper is that the decision-maker requires the out-of-sample disappointment to decay at a fixed exponential rate. This assumption can be relaxed using ideas from moderate deviations theory if the training samples are generated by a finite state i.i.d.\ process \cite{ref:Amine-21}. Specifically, it is shown that if the out-of-sample disappointment must decay superexponentially, then the Pareto dominant data-driven prescriptor is obtained by solving a classical robust optimization model that minimizes the worst-case risk with respect to all possible uncertainty realizations. On the other hand, if the out-of-sample disappointment must decay subexponentially, then the Pareto dominant data-driven prescriptor is obtained by solving an empirical risk minimization problem with a variance penalty. Finally, we assume in this paper that the decision-maker has access to noise-free training samples. This assumption is relaxed in \cite{vanparys2021optimal}, where a DRO model based on an entropic optimal tranport distance is shown to provide Pareto dominant data-driven prescriptors when the training samples are corrupted by noise.

 The out-of-sample disappointment and the in-sample risk are by no means the only performance criteria for which the best representatives within a certain class of prescriptors are accessible. Another performance criterion of interest is the regret convergence rate. For example, in the context of data-driven linear optimixation with side information, it has recently been shown that the na\"ive ``estimate and then optimize'' approach is markedly superior to the ``induced empirical risk method'' with respect to this criterion \cite{ref:hu2020fast}.

The paper develops as follows. \Cref{sec:stochastic:programming} formalizes our approach to data-driven decision-making and constructs the meta-optimization problems that will be used to find optimal data-driven predictors and prescriptors. Sections~\ref{sec:optimal:dd:prescriptors} and~\ref{sec:equivalence} establish sufficient conditions under which the restricted and original meta-optimization problems have Pareto dominant solutions, respectively, and \Cref{sec:models} showcases practically relevant examples in which these conditions hold. All proofs, along with several auxiliary results, are provided in the appendix.

\paragraph{Notation.} A multi-objective optimization problem $\min_{x\in\mc X}  \{f_\alpha(x)\}_{\alpha\in\mc A}$ is determined by its feasible set~$\mc X$ and its objective functions $f_\alpha:\mc X\to \Re$ indexed by $\alpha \in \mc A$.
A strong solution is a feasible solution~$x^\star\in\mc X$ that Pareto dominates every other feasible solution in the sense that $f_\alpha(x^\star)\leq f_\alpha(x)$ for all $x\in\mc X$ and $\alpha \in\mc A$. A weak solution is a feasible solution $x^\star\in\mc X$ that is not Pareto dominated by any other feasible solution in the sense that there exists no $x\in \mc X$ such that $f_\alpha(x) \leq f_\alpha(x^\star)$ for all $\alpha \in\mc A$. A function $f:\mc X \to \mc Y$ from $\mc X \subseteq \Re^n$ to $\mc Y\subseteq \Re^m$ is called quasi-continuous at $x\in\mc X$ if for every neighbourhood $\mc U\subseteq \mc X$ of $x$ and every neighborhood $\mc V\subseteq \mc Y$ of $f(x)$ there exists a non-empty open set $\mc W\subseteq \mc U$ with $f(x')\in\mc V$ for all $x'\in\mc W$. Note that $\mc W$ may not contain $x$. The $n$-dimensional probability simplex is denoted by $\Delta_n = \{x \in \mb R^n_+ : \sum_{i=1}^n x_i = 1\}$. For any logical expression $\mathcal E$, the indicator function $\Indic{\mathcal E}$ evaluates to~1 if $\mathcal E$ is true and to~0 otherwise, and for any $A,B\in\Re^{n\times m}$ the trace inner
product is denoted by $\inprod{A}{B}=\mathsf{tr}(A^\top B)$.

%%%%%%%%%%%%%%%%%%%%%%%
%%    SEC. Stochastic optimization
%%%%%%%%%%%%%%%%%%%%%%%
\section{Data-driven optimization} \label{sec:stochastic:programming}
Throughout this paper we assume that all random objects are defined on the same abstract probability space~$(\Omega, \mc F, \mb P_\star)$, and we study a general stochastic optimization problem of the form
\begin{equation}
\label{eq:initial:problem-f}
    \min_{x\in X} \; c(x,\mb P_\star),
\end{equation}
where the goal is to find a decision~$x \in X\subseteq \Re^n$ that minimizes a real-valued objective or `cost' function~$c(x,\mb P_\star)$ depending on the probability measure~$\mb P_\star$. We assume throughout the paper that~$X$ is compact and that~$c(x,\mb P_\star)$ is continuous in~$x$. These assumptions guarantee that the minimum in~\eqref{eq:initial:problem-f} is attained.

\begin{example}[Objective functions]
\label{ex:objective-functions}
A popular objective arising in operations research and statistics is to minimize the expected value of a loss function~$\ell(x,\xi)$ that depends both on the decision~$x$ and an exogenous random vector~$\xi\in\Re^m$. Denoting the expectation operator with respect to~$\mb P_\star$ by~$\mb E_{\mb P_\star}[\cdot]$, we thus set
\begin{subequations}\label{eq:examples:c}
\begin{equation}\label{eq:prescription_problem}
    c(x,\mb P_\star)= \mb E_{\mb P_\star}[\ell(x, \xi)].
\end{equation}
In risk averse optimization~\cite[Chapter~6]{shapiro2014lectures} the expectation is replaced with a risk measure~$\varrho_{\mb P_\star}[\cdot]$. We thus~set
\begin{equation}\label{eq:prescription_problem:2}
    c(x,\mb P_\star)= \varrho_{\mb P_\star}[\ell(x,\xi)]. 
\end{equation}
Examples of risk measures include the variance, the value-at-risk or the conditional value-at-risk of the loss as well as their convex combinations with the expected loss. Decision makers sometimes have access to contextual covariates, that is,  observable random variables that are correlated with the unobservable random variables impacting the loss function. In such situations it is beneficial to solve a conditional stochastic optimization problem that minimizes the conditional expectation of the loss given the contextual covariates~\cite{ref:Cynthia-19, ref:Kallus-2020,ref:hu2020fast}. If the matrix~$C\in\Re^{m\times m_C}$ filters out~$m_C$ observable covariates from~$\xi$ and if these covariates are known to fall within a Borel set~$B\subseteq \Re^{m_C}$ (note that $B$ could represent a singleton), then we set the objective function to
\begin{equation}
\label{eq:SP:covariate}
    c(x,\mb P_\star)= \mb E_{\mb P_\star}[\ell(x, \xi)|C\xi\in B].
\end{equation}
Contextual information may include weather forecasts, Twitter feeds or Google Trends data. Stochastic control, as a last example, aims to guide a dynamical system to a desirable state, assuming that the system's state obeys a recursion~$s_{t+1}=f(s_t,u_t,\xi_t)$ that depends on some control inputs~$u_t$ and exogenous random disturbances~$\xi_t$ at time~$t\in\N$. If the inputs are set to~$u_t=\pi_x(s_t)$ for some control policy~$\pi_x$ parametrized by~$x\in X$ and if~$\ell(u_t,s_t)$ represents the cost at time~$t$, then one may minimize the long-run average cost
\begin{equation}
\label{eq:ex:dynamical:system}
    c(x,\mb P_\star)=\lim_{T\to\infty}\frac{1}{T}\sum_{t=1}^T \mb E_{\mb P_\star}[\ell(\pi_x(s_t),s_t)].
\end{equation}
\end{subequations}
Note that~$x$ impacts the objective function~\eqref{eq:ex:dynamical:system} both directly through the policy~$\pi_x$ as well as indirectly through the states~$s_t$, $t\in\N$, which are defined by a recursion that depends on~$x$. We also emphasize that some mild technical assumptions are needed for the objective functions~\eqref{eq:prescription_problem}--\eqref{eq:ex:dynamical:system} to be well-defined. But the above examples show that the abstract stochastic optimization problem~\eqref{eq:initial:problem-f} is remarkably general.
\end{example}

%The cost functions \eqref{eq:prescription_problem}, \eqref{eq:prescription_problem:2}, \eqref{eq:SP:covariate} and \eqref{eq:ex:dynamical:system} are prominent examples in decision theory, nevertheless the results developed in this paper do not rely on the specific structure of any of those problems. We therefore mainly consider these examples as motivation for our general framework.

When reasoning about the stochastic optimization problem \eqref{eq:initial:problem-f}, it is expedient to distinguish the {\em prediction problem}, which aims to evaluate the cost $c(x,\mb P_\star)$ associated with a fixed decision $x$, from the {\em prescription problem}, which aims to find a decision $x^\star$ that minimizes the cost $c(x,\mb P_\star)$ across all $x\in X$. We emphasize that any procedure for solving the prescription problem invariably necessitates a procedure for solving the corresponding prediction problem. As the prediction problem is reminiscent of an uncertainty quantification problem \cite{lemaitre2010introduction}, however, it is of interest in its own right. Unfortunately, already the prediction problem poses two formidable challenges. On the one hand, the probability measure~$\mb P_\star$, which is needed to evaluate the objective function, is usually unobservable. On the other hand, even if one had access to~$\mb P_\star$, computing the objective function $c(x,\mb P_\star)$ for a fixed decision~$x$ might be difficult. For example, evaluating the expectation in~\eqref{eq:prescription_problem} is \#P-hard even if~$\ell(x,\xi)$ is defined as the non-negative part of an affine function of~$\xi$ and if~$\xi$ is uniformly distributed on the standard hypercube in~$\Re^m$ under the probability measure~$\mb P_\star$ \cite[Corollary~1]{Hanasusanto2016complexity}.

In the following we develop a systematic approach for addressing the prediction and prescription problems when~$\mb P_\star$ is only indirectly observable through finitely many training samples. We endeavor to keep the proposed framework as general as possible. In particular, we will forgo any restrictive independence assumptions and explicitly account for the possibility that the training data are serially dependent.

\subsection{The data-driven newsvendor problem} \label{sec:newsvendor}
%%%%%%%%%%%%%%%%%%%%%%%%%%%%%%%%%%%%%%%%

We first exemplify several popular approaches to data-driven decision-making in the context of the classical newsvendor problem, which %is arguably one of the most fundamental decision problems in operations management and economics. The newsvendor model 
captures the fundamental dilemma faced by the seller of a perishable good. The textbook example of such a seller is a newsvendor who sells a daily newspaper that becomes worthless at the end of the day. At the beginning of each day, the newsvendor orders~$x\in X$ newspapers from the publisher at the wholesale price~$k\ge 0$, where $X=\{1,\hdots, d\}$. Then, the uncertain demand~$\xi\in\Xi$ is revealed, where $\Xi=X$, and the newsvendor sells newspapers at the retail price $p>k$ until either the inventory or the demand is exhausted. The number of newspapers sold is thus given by~$\min\{x,\xi \}$, and the total cost amounts to~$\ell(x,\xi)=kx -p \min\{x,\xi\}$. If the probability measure $\mb P_\star$ governing the demand is known, then the problem of minimizing the expected cost can be formulated as a stochastic optimization problem of the form~\eqref{eq:initial:problem-f} with objective function $\mb E_{\mb P_\star}[\ell(x,\xi)]= \sum_{i\in \Xi} \ell(x, i)\, (\theta_\star)_i$, where the probability vector $\theta_\star\in\Delta_{d}$ is defined through $(\theta_\star)_i=\mb P_\star[\xi=i]$ for all $i\in \Xi$. Note that~$\theta_\star$ captures all information about~$\mb P_\star$ that is needed to solve the newsvendor's decision problem. By slight abuse of notation, we may thus identify~$\mb P_\star$ with~$\theta_\star$ and use~$c(x,\theta_\star)$ as a shorthand for the expected cost of any fixed order quantity $x\in X$. % and probability vector~$\theta_\star\in\Delta_{d}$. 
If the demands on different days are \ac{iid}, then the law of large numbers guarantees that $\min_{x\in X} c(x,\theta_\star)$ represents the minimum cost attainable by the newsvendor on average in the long run.

In reality, the probability measure $\mb P_\star$ and the probability vector $\theta_\star$ are unobservable and must be estimated from historical demand realizations $\xi_t\in\Xi$, $t=1,\ldots,T$, which we refer to as training samples. We assume here for simplicity that the training samples are mutually independent, but the general methods developed in this paper do not rely on this assumption. Given a batch of only~$T$ training samples, the newsvendor now seeks to answer three intertwined questions: (i)~What is the expected cost of a given ordering decision? (ii)~How many newspapers should be ordered so as to minimize the expected cost? (iii)~What is the probability that the unknown true expected cost of the chosen ordering decision exceeds the estimated cost? 

In the following we designate all estimators ({\em i.e.}, all functions of the data) with a superscript `\;$\widehat{}$\;' as well as a subscript~`$T$' indicating the size of the underlying dataset. For example, we use $\widehat c_T(x)$ to denote an estimator of the expected cost~$c(x,\theta_\star)$  constructed from~$T$ demand samples, where~$x$ is any feasible ordering decision. Similarly, we use~$\widehat x_T$ to denote an estimator for the optimal ordering decision constructed from~$T$ demand samples. Below we assume that $\widehat x_T\in\arg\min_{x\in X}\widehat c_T(x)$, that is, we assume that any estimator for the optimal ordering decision is induced by some estimator for the expected cost function. Note that this assumption can be imposed without loss of generality. Indeed, any estimator~$\widehat x_T$ for the optimal decision can be expressed as a minimizer of a cost function estimator; for example, we may set~$\widehat c_T(x)=(x-\widehat x_T)^2$. 

Questions~(i) and~(ii) above address the construction of the estimators~$\widehat c(x)$ and~$\widehat x_T$, respectively, while question~(iii) asks for the probability of the event~$c(\widehat x_T, {\theta_\star}) > \widehat c_T(\widehat x_T)$. In this event the true (out-of-sample) expected cost of the chosen decision $\widehat x_T$ exceeds the estimated (in-sample) expected cost, which might lead to a budget overrun and force the newsvendor into financial distress. In the following we refer to the probability of this event (with respect to the sampling of the training dataset) as the \textit{out-of-sample disappointment}. %, which is evaluated under the joint distribution of the demand samples, 
Note that in the event $c(\widehat x_T, {\theta_\star}) < \widehat c_T(\widehat x_T)$ there is also a discrepancy between the estimated budget and the true expected cost. However, in this event the newsvendor faces no severe financial repercussions.

There are countless possibilities to construct cost estimators~$\widehat c_T(x)$ and the corresponding decision estimators~$\widehat x_T$ from the training data. Different estimators may offer different statistical guarantees and display different computational properties. However, the existing literature offers little guidance on how to choose among these many estimators. Moreover, there could exist yet undiscovered estimators that dominate all known estimators in terms of some meaningful statistical criteria. In the following we will compare different estimators in terms of the exponential decay rate of their out-of-sample disappointment, which is defined as
\[
  \lim_{T\to\infty} - \frac 1T \log \mb P_\star[c(\widehat x_T, {\theta_\star}) > \widehat c_T(\widehat x_T)],
\]
and in terms of their asymptotic in-sample cost, which is defined as~$\lim_{T\to\infty} \mb E_{\mb P_\star}[\widehat c_T(\widehat x_T)]$. We will see later that these quantities are well-defined for a wide range of estimators. In the remainder we thus view a pair of cost and decision estimators as `desirable' if the asymptotic in-sample cost is low ({\em i.e.}, the expected cost of~$\widehat x_T$ is predicted to be low) and if the decay rate of the out-of-sample disappointment is high ({\em i.e.}, the probability that the true expected cost of~$\widehat x_T$ exceeds the predicted cost decays quickly as~$T$ grows).

%(the probability of a budget overrun drops quickly as~$T$ grows) and 

Arguably one of the simplest conceivable cost estimators is the empirical cost~$\widehat c_T(x)=\frac{1}{T}\sum_{t=1}^T \ell(x,\xi_t)$. Thus, we have~$\widehat c_T(x) = c(x, \widehat S_T)$, where~$c(x,\theta)=\sum_{i\in\Xi} \ell(x,i) \theta_i$ represents the expected cost of the decision~$x$ when the demand uncertainty is described by the probability vector~$\theta\in\Delta_d$, and the statistic~$\widehat S_T\in\Delta_d$ stands for the empirical probability vector, whose $i^{\rm th}$ component $(\widehat S_T)_{i} = \frac1T \sum_{t=1}^{T} \Indic{ \xi_t=i}$ records the empirical frequency of the~$i^{\rm th}$ demand scenario for each~$i\in\Xi$. Using the central limit theorem, one can show that the expected in-sample cost of the empirical cost estimator and its induced decision estimator converges to the true optimum~$\min_{x\in X} c(x,\theta_\star)$ and that the out-of-sample disappointment converges to~$50\%$ as~$T$ grows. Thus, the decay rate of the out-of-sample disappointment vanishes completely; see also~\cite[Example~2]{ref:vanParys:fromdata-17}. 

A na\"ive approach to force the out-of-sample disappointment to decay would be to add a constant positive penalty~$\varepsilon$ to the empirical cost estimator, thereby increasing its asymptotic in-sample cost and thus introducing a conservative bias. This reasoning suggests that the in-sample cost and the out-of-sample disappointment stand in direct competition. 
%We may for instance consider $\widehat c^{\hspace{0.25em}{\rm{saa}}}_T+\varepsilon$ instead for some $\varepsilon>0$ and buy better statistical performance at the expense of conservatism in the sense of over-budgeting the expected cost of the newsvendor.
%The level of conservatism of an estimator $(\widehat x_T, \widehat c_T)$ grows with $\mb E_{\theta_\star}[\widehat c_T(\widehat x_T)]$, which can be interpreted as the estimated long-run expected cost of the newsvendor under policy $\widehat x_T$. It is clear that the performance of any data-driven formulation must be interpreted as how well its makes this balance between (i) out-of-sample disappointment and (ii) budget conservatism.
In order to provide a better intuition for the trade-off between these statistical performance criteria, we further investigate three distributionally robust cost estimators of the form~$\widehat c_T(x)=\max_{\theta\in \widehat \Theta_T} c(x,\theta)$, which evaluate the worst-case expected cost of the decision~$x$ with respect to all probability vectors from within some ambiguity set~$\widehat\Theta_T \subseteq \Delta_d$ constructed from the training data. 

Traditionally, distributionally robust optimization has mostly studied moment ambiguity sets such as
\(
    \textstyle \widehat\Theta_T = \{\theta\in\Delta_{d} : |\sum_{i\in \Xi} i^j \theta_{i}-\sum_{i\in \Xi} i^j (\widehat S_T)_{i} |\leq \varepsilon ~ \forall j=1,\hdots,J \}.
\)
All probability vectors in this ambiguity set share, to within an absolute error tolerance~$\varepsilon\ge 0$, the same moments of all orders up to~$J$ as the empirical probability vector~$\widehat S_T$. In the subsequent numerical experiments we set~$J=4$. The tolerance~$\varepsilon$ is usually tuned to ensure that~$\widehat\Theta_T$ contains the unknown data-generating probability vector~$\theta_\star$ with a prescribed high confidence; see~\cite[Section~3]{delage2010distributionally} for the first results of this kind.
%three different sizes of ambiguity set (i.e., $\varepsilon$ can be `small', `medium' or `large').
The recent literature has witnessed an increasing interest in Wasserstein ambiguity sets of the form
\(
    \widehat \Theta_T = \{\theta\in\Delta_{d} : \mathsf{d_W}(\theta,\widehat S_T)\leq \varepsilon \},
\)
where $\mathsf{d_W}(\theta,\widehat S_T)$ denotes the first Wasserstein distance between~$\theta$ and~$\widehat S_T$ \cite{kantorovich1958space}. This ambiguity set can be viewed as a Wasserstein ball of radius~$\varepsilon\ge 0$ around~$\widehat S_T$ in~$\Delta_d$. Unlike the moment ambiguity set, the Wasserstein ambiguity set shrinks to the singleton that contains merely the empirical probability vector if we set~$\varepsilon=0$. In general, $\varepsilon$ can again be tuned to ensure that~$\theta_\star\in\widehat \Theta_T$ with any prescribed high confidence \cite[Section~3]{ref:Peyman-18}. 
%where $\mb W_\varepsilon(\widehat\theta_T)\defn\{\theta\in\Delta_{d} : \mathsf{d_W}(\theta,\widehat\theta_T)\leq \varepsilon\}$, 
Finally, we also study relative entropy ambiguity sets of the form
\(
    \widehat \Theta_T = \{\theta\in\Delta_{d} : D(\widehat S_T\|\theta)\leq \varepsilon \},
\)
where~$D(\widehat S_T\|\theta)$ stands for the relative entropy (or Kullback-Leibler divergence) of~$\widehat S_T$ with respect to~$\theta$. This ambiguity set also shrinks to a singleton for~$\varepsilon=0$, and~$\varepsilon$ can again be tuned to guarantee that~$\widehat \Theta_T$ covers~$\theta_\star$ with a prescribed probability \cite[Section~3]{bental2013uncertain}. In contrast to most of the existing literature on distributionally robust optimization, here we are {\em not} concerned about whether or not the ambiguity set covers~$\theta_\star$. Instead, we view any distributionally robust optimization model simply as a vehicle for transforming data to decisions, and we are merely interested in the statistical properties of the resulting cost and decision estimators. %, and we assess its quality only in terms of the 

%\cite{hu2012kullback}

%%%%%%%% 
\begin{figure}[t] 
\centering
\subfloat[Out-of-sample disappointment versus training sample size~$T$ %\newline \ \qquad $\mb P_{\theta_\star}{[}c(\widehat x_T,\theta_\star)>\widehat c_T(\widehat x_T){]}$
]{% This file was created by matlab2tikz.
% Minimal pgfplots version: 1.3
%  http://www.mathworks.com/matlabcentral/fileexchange/22022-matlab2tikz
\begin{tikzpicture}

\begin{axis}[%
width=1.4in,
height=1.4in,
at={(1.011111in,0.641667in)},
scale only axis,
%xmode=log,
ymode=log,
log ticks with fixed point,
xmin=10,
xmax=200,
xlabel={{\color{white}{$\lim_{T\to\infty}\frac{\widehat 1}{T}$}}$T${\color{white}{$\lim_{T\to\infty}\frac{1}{T}$}}},
ymin=0.0001,
ymax=1,
ylabel={\footnotesize{$\mb P_{\star} [ c(\widehat x_T, {\theta_\star}) > \widehat c_T(\widehat x_T) ]$}},
y label style={yshift=0.0em},
ytick = {0.0001,0.001,0.01,0.1,0.5},
yticklabels ={$10^{-4}$,$10^{-3}$,$10^{-2}$,$10^{-1}$,$0.5$},
title style={font=\bfseries},
legend to name=named,
legend style={legend cell align=left,align=left,draw=white!15!black,legend columns=2}
]

%SAA method
%%%%%%%%%%%%%%%%%%%%%%%%%%%%%%%%%
\addplot [color=brown,solid,line width = 1.2pt, smooth]
  table[row sep=crcr]{%
10	0.679633333333333\\
25	0.5275\\
41	0.521566666666667\\
56	0.5491\\
71	0.498866666666667\\
86	0.5273\\
102	0.536033333333333\\
117	0.497933333333333\\
132	0.523633333333333\\
147	0.5428\\
163	0.498333333333333\\
178	0.520166666666667\\
193	0.490733333333333\\
208	0.5126\\
224	0.5183\\
239	0.494233333333333\\
254	0.512033333333333\\
269	0.529666666666667\\
285	0.495433333333333\\
300	0.513833333333333\\
};
\addlegendentry{Empirical cost};

\addplot [color=green,solid,line width = 1.2pt]
  table[row sep=crcr]{%
10	0.31719\\
14	0.27314\\
18	0.23969\\
22	0.21361\\
26	0.19231\\
29	0.17963\\
33	0.16369\\
37	0.14951\\
41	0.13792\\
45	0.12643\\
49	0.11746\\
53	0.10964\\
57	0.10236\\
60	0.09772\\
64	0.0905\\
68	0.08476\\
72	0.07955\\
76	0.07394\\
80	0.06836\\
84	0.06402\\
88	0.05931\\
91	0.0565\\
95	0.05271\\
99	0.04885\\
103	0.04617\\
107	0.04264\\
111	0.04055\\
115	0.03814\\
119	0.03599\\
122	0.03381\\
126	0.03212\\
130	0.0296\\
134	0.02812\\
138	0.02648\\
142	0.02442\\
146	0.02315\\
150	0.02196\\
153	0.02122\\
157	0.01977\\
161	0.01869\\
165	0.01771\\
169	0.0166\\
173	0.01577\\
177	0.01455\\
181	0.01378\\
184	0.01315\\
188	0.01262\\
192	0.01165\\
196	0.01068\\
200	0.01037\\
};
\addlegendentry{Moment-based ambiguity set  ($\varepsilon=0.05$)};

\addplot [color=cyan, solid,line width = 1.2pt]
  table[row sep=crcr]{%
10	0.25757\\
14	0.19503\\
18	0.15218\\
22	0.12195\\
26	0.09899\\
29	0.08498\\
33	0.06981\\
37	0.05799\\
41	0.04793\\
45	0.03972\\
49	0.03327\\
53	0.02786\\
57	0.02299\\
60	0.02002\\
64	0.01718\\
68	0.01434\\
72	0.01223\\
76	0.01022\\
80	0.00812\\
84	0.00688\\
88	0.00595\\
91	0.00507\\
95	0.0043\\
99	0.00373\\
103	0.00327\\
107	0.00279\\
111	0.00234\\
115	0.00195\\
119	0.00157\\
122	0.00132\\
126	0.00129\\
130	0.00109\\
134	0.00097\\
138	0.00076\\
142	0.0006\\
146	0.0006\\
150	0.0005\\
153	0.00044\\
157	0.00037\\
161	0.00036\\
165	0.00032\\
169	0.00023\\
173	0.0002\\
177	0.00012\\
181	0.00013\\
184	0.0001\\
188	9e-05\\
192	6e-05\\
196	4e-05\\
200	4e-05\\
};
\addlegendentry{Moment-based ambiguity set ($\varepsilon=0.13$) $\quad$};

\addplot [color=blue,solid,line width = 1.2pt]
  table[row sep=crcr]{%
10	0.21461\\
14	0.1518\\
18	0.10972\\
22	0.08186\\
26	0.06124\\
29	0.05061\\
33	0.03902\\
37	0.0307\\
41	0.02392\\
45	0.01823\\
49	0.01413\\
53	0.01086\\
57	0.00865\\
60	0.00723\\
64	0.0057\\
68	0.00437\\
72	0.00348\\
76	0.00278\\
80	0.00211\\
84	0.0018\\
88	0.00146\\
91	0.00112\\
95	0.00095\\
99	0.00072\\
103	0.00052\\
107	0.00046\\
111	0.00039\\
115	0.00028\\
119	0.00026\\
122	0.00023\\
126	0.00012\\
130	9e-05\\
134	6e-05\\
138	7e-05\\
142	6e-05\\
146	8e-05\\
150	4e-05\\
153	1e-05\\
157	1e-05\\
161	0\\
165	0\\
169	0\\
173	1e-05\\
177	1e-05\\
181	1e-05\\
184	1e-05\\
188	0\\
192	0\\
196	0\\
200	0\\
};

\addlegendentry{Moment-based ambiguity set ($\varepsilon=0.2$)};

\addplot [color=orange,solid,line width = 1.2pt]
  table[row sep=crcr]{%
10	0.1609\\
20	0.0716\\
30	0.0483\\
40	0.023\\
50	0.0149\\
60	0.0071\\
70	0.0035\\
80	0.0024\\
90	0.0014\\
100	0.0007\\
110	0.0004\\
120	0.0003\\
130	0.0001\\
140	0.00002\\
150	0.00002\\
160	0.00001\\
170	0.00001\\
180	0\\
190	0\\
200	0\\
};
\addlegendentry{Wasserstein ambiguity set ($\varepsilon=0.28$)};

%KL - Method
%%%%%%%%%%%%%%%%%%%%%%%
\addplot [color=magenta,solid, line width = 1.2pt]
  table[row sep=crcr]{%
10	0.21869\\
14	0.12618\\
18	0.06063\\
22	0.04038\\
26	0.05102\\
29	0.03583\\
33	0.02079\\
37	0.01353\\
41	0.00815\\
45	0.00602\\
49	0.00422\\
53	0.00396\\
57	0.00256\\
60	0.00203\\
64	0.00136\\
68	0.00093\\
72	0.00071\\
76	0.00065\\
80	0.00044\\
84	0.0003\\
88	0.00016\\
91	0.00016\\
95	9e-05\\
99	3e-05\\
103	6e-05\\
107	2e-05\\
111	1e-05\\
115	2e-05\\
119	0\\
122	0\\
126	2e-05\\
130	1e-05\\
134	1e-05\\
138	1e-05\\
142	0\\
146	0\\
150	0\\
153	0\\
157	0\\
161	0\\
165	0\\
169	0\\
173	0\\
177	0\\
181	0\\
184	0\\
188	0\\
192	0\\
196	0\\
200	0\\
};
\addlegendentry{Relative entropy ambiguity set ($\varepsilon=0.12$)};

\end{axis}
\end{tikzpicture}%
 \label{fig:newsvendor:disappointment} }   
\hspace{1mm}
\subfloat[In-sample cost versus training sample size~$T$ %\newline  $\lim_{T\to\infty}\mb E_{\theta_\star}{[}\widehat c_T(\widehat x_T){]}$ 
]{\begin{tikzpicture}

\begin{axis}[%
width=1.4in,
height=1.4in,
at={(1.011111in,0.641667in)},
scale only axis,
area legend,
xmin=10,
xmax=200,
ymin=-8.6,
ymax=-5.8,
xlabel={{\color{white}{$\lim_{T\to\infty}\frac{\widehat 1}{T}$}}$T${\color{white}{$\lim_{T\to\infty}\frac{1}{T}$}}},
legend style={legend cell align=left,align=left,draw=white!15!black},
ylabel={\footnotesize{$\mb E_{\mathbb{P}_\star}[\widehat c_T(\widehat x_T)]$}},
y label style={yshift=-0.65em}
]

\addplot [color=magenta,solid,line width = 1.2pt]
  table[row sep=crcr]{%
10	-7.01863929524303\\
14	-6.90027394621443\\
18	-6.83523203072383\\
22	-6.79194211874691\\
26	-6.76247493534892\\
29	-6.74339996775228\\
33	-6.72516671016689\\
37	-6.71153583823882\\
41	-6.69790088911979\\
45	-6.68888939032243\\
49	-6.68028311284254\\
53	-6.67262355322833\\
57	-6.66792425670287\\
60	-6.66229739103664\\
64	-6.65755604436462\\
68	-6.65165354147988\\
72	-6.64761961492752\\
76	-6.64335526560934\\
80	-6.63848125410692\\
84	-6.63582544829923\\
88	-6.63212180998676\\
91	-6.63075208696928\\
95	-6.62798938118148\\
99	-6.62541179695912\\
103	-6.62369614471228\\
107	-6.62113419532301\\
111	-6.61894269286582\\
115	-6.61706578935708\\
119	-6.6159502396686\\
122	-6.61422161970633\\
126	-6.61242192038233\\
130	-6.61076299645775\\
134	-6.60883654912161\\
138	-6.60744437780694\\
142	-6.60544842319226\\
146	-6.60419451010456\\
150	-6.60354396701989\\
153	-6.60227579648789\\
157	-6.60080681755029\\
161	-6.59934029866129\\
165	-6.59837346167927\\
169	-6.5973143953698\\
173	-6.59646283072847\\
177	-6.59541273376944\\
181	-6.59432614854369\\
184	-6.5940253379534\\
188	-6.59313805193528\\
192	-6.59203604071451\\
196	-6.59087427895434\\
200	-6.59025949075887\\
};

\addplot [color=green,solid,line width = 1.2pt]
  table[row sep=crcr]{%
10	-7.36862612817014\\
14	-7.26212839876729\\
18	-7.21062762920112\\
22	-7.17771135985715\\
26	-7.15670825804205\\
29	-7.14316632813764\\
33	-7.1303421244018\\
37	-7.12162931759408\\
41	-7.11125651467403\\
45	-7.10555605456781\\
49	-7.10006255263995\\
53	-7.09559226799943\\
57	-7.09350788679994\\
60	-7.0897724019163\\
64	-7.08729566585501\\
68	-7.08403313768041\\
72	-7.08183027312909\\
76	-7.07923732461732\\
80	-7.07650482892414\\
84	-7.0746480290928\\
88	-7.07222642752249\\
91	-7.07186629598221\\
95	-7.07017089822603\\
99	-7.06919359854525\\
103	-7.06836578101265\\
107	-7.06692013651942\\
111	-7.06606657083382\\
115	-7.06545269956678\\
119	-7.06520993780891\\
122	-7.06449717320911\\
126	-7.06364058246025\\
130	-7.06272539424893\\
134	-7.06203307355256\\
138	-7.06143505517369\\
142	-7.06028080425262\\
146	-7.05985333010664\\
150	-7.05980922964139\\
153	-7.05925093301287\\
157	-7.05861020702926\\
161	-7.05798578363577\\
165	-7.05777902496632\\
169	-7.0573792424501\\
173	-7.05727731965586\\
177	-7.05685493602322\\
181	-7.05636498692783\\
184	-7.0564262150854\\
188	-7.05620713745863\\
192	-7.05582366626725\\
196	-7.05536228721087\\
200	-7.05537462478842\\
};

\addplot [color=cyan,solid,line width = 1.2pt]
  table[row sep=crcr]{%
10	-6.69320502720096\\
14	-6.5756586834422\\
18	-6.51798203906592\\
22	-6.48157237959006\\
26	-6.45804012579001\\
29	-6.44282942708689\\
33	-6.42865515234826\\
37	-6.41874169344973\\
41	-6.40681095825039\\
45	-6.3996670266256\\
49	-6.3931298065299\\
53	-6.38759532968607\\
57	-6.38447758220575\\
60	-6.37985216175065\\
64	-6.37674071299751\\
68	-6.37261525635777\\
72	-6.3695466945371\\
76	-6.36626678483915\\
80	-6.36278170513888\\
84	-6.36022386369045\\
88	-6.3570245046212\\
91	-6.35626743851034\\
95	-6.35394952011694\\
99	-6.35237248824914\\
103	-6.35101178663701\\
107	-6.34892165989356\\
111	-6.34757146124627\\
115	-6.34644149315219\\
119	-6.34562572592672\\
122	-6.34448315099729\\
126	-6.3430841786155\\
130	-6.34163529049144\\
134	-6.34053534041042\\
138	-6.33954435772143\\
142	-6.33791669170572\\
146	-6.33707451019405\\
150	-6.33659080246299\\
153	-6.33579736589872\\
157	-6.33483232074662\\
161	-6.33384754512071\\
165	-6.33333121876265\\
169	-6.33256098349599\\
173	-6.33216905996465\\
177	-6.3314386194209\\
181	-6.33052431783457\\
184	-6.33037798036836\\
188	-6.32990242338654\\
192	-6.32922705267361\\
196	-6.32839631217685\\
200	-6.3281884240556\\
};

\addplot [color=blue, solid,line width = 1.2pt]
  table[row sep=crcr]{%
10	-6.38125543788758\\
14	-6.25282753547663\\
18	-6.18846449243821\\
22	-6.14791219279667\\
26	-6.12190400884754\\
29	-6.10517059146406\\
33	-6.08943185320195\\
37	-6.07820060937101\\
41	-6.06512797587913\\
45	-6.05692418512957\\
49	-6.04961005423323\\
53	-6.04336838887214\\
57	-6.03953809682298\\
60	-6.03434702867522\\
64	-6.03066805596749\\
68	-6.02600411402943\\
72	-6.02236263692597\\
76	-6.01866413146194\\
80	-6.01466010618053\\
84	-6.01168065972225\\
88	-6.00805168640155\\
91	-6.00699372472526\\
95	-6.00425963859653\\
99	-6.002361143165\\
103	-6.00071435553399\\
107	-5.99819195633437\\
111	-5.99650552788786\\
115	-5.99508488164248\\
119	-5.99399945016809\\
122	-5.9925933007013\\
126	-5.99080655250441\\
130	-5.9890553157692\\
134	-5.98764790358213\\
138	-5.98634473309981\\
142	-5.98438777275429\\
146	-5.98332164767129\\
150	-5.98260570003707\\
153	-5.98168530392157\\
157	-5.98044509983019\\
161	-5.97919745545893\\
165	-5.97842688885525\\
169	-5.97741836003951\\
173	-5.97688023966608\\
177	-5.97596917951036\\
181	-5.97479810211788\\
184	-5.97455170083938\\
188	-5.97382458063238\\
192	-5.97294001192715\\
196	-5.97193453562356\\
200	-5.97152747044442\\
};

\addplot [color=orange,solid,line width = 1.2pt]
  table[row sep=crcr]{%
10	-6.71161999999998\\
20	-6.53963999999999\\
30	-6.48222333333333\\
40	-6.45378750000003\\
50	-6.43493200000009\\
60	-6.43163500000002\\
70	-6.43070000000002\\
80	-6.42473125\\
90	-6.42588444444451\\
100	-6.42297799999996\\
110	-6.42034727272734\\
120	-6.41896750000003\\
130	-6.41809461538457\\
140	-6.41804500000001\\
150	-6.41717999999999\\
160	-6.416105625\\
170	-6.41548058823528\\
180	-6.41463388888889\\
190	-6.41417368421049\\
200	-6.41370450000001\\
};

%SAA
\addplot [color=brown,solid,smooth,line width = 1.2pt]
  table[row sep=crcr]{%
10	-8.60679333333325\\
25	-8.41675466666896\\
41	-8.3689495934956\\
56	-8.35042083333333\\
71	-8.34377793427228\\
86	-8.34046240310076\\
102	-8.33747287581692\\
117	-8.33561339031361\\
132	-8.33470429292926\\
147	-8.33465873015873\\
163	-8.33457464212675\\
178	-8.33416629213473\\
193	-8.33407979274629\\
208	-8.33363076923076\\
224	-8.33326666666667\\
239	-8.33356708507657\\
254	-8.33361706036747\\
269	-8.33403568773238\\
285	-8.33395976608193\\
300	-8.33381622222225\\
};

% %NN
% \addplot [color=gray,solid,smooth,line width = 1.2pt]
%   table[row sep=crcr]{%
% 1	-6.77892283870968\\
% 11.4736842105263	-6.81341579269056\\
% 21.9473684210526	-6.84543169368242\\
% 32.4210526315789	-6.87497054168528\\
% 42.8947368421053	-6.90203233669913\\
% 53.3684210526316	-6.92661707872398\\
% 63.8421052631579	-6.94872476775981\\
% 74.3157894736842	-6.96835540380663\\
% 84.7894736842105	-6.98550898686444\\
% 95.2631578947368	-7.00018551693325\\
% 105.736842105263	-7.01238499401305\\
% 116.210526315789	-7.02210741810383\\
% 126.684210526316	-7.02935278920561\\
% 137.157894736842	-7.03412110731838\\
% 147.631578947368	-7.03641237244214\\
% 158.105263157895	-7.03622658457689\\
% 168.578947368421	-7.03356374372263\\
% 179.052631578947	-7.02842384987937\\
% 189.526315789474	-7.02080690304709\\
% 200	-7.01071290322581\\
% };

\end{axis}
\end{tikzpicture}%

%
%\begin{tikzpicture}
%\begin{axis}[
%    ybar,
%    width=2.5in,
%height=2.5in,
% %   enlargelimits=0.15,
%    legend style={at={(0.5,-0.15)},
%      anchor=north,legend columns=-1},
%   % ylabel={\#participants},
%    symbolic x coords={d},
%    xtick=data,
%    nodes near coords,
%    nodes near coords align={vertical},
%    ]
%\addplot[blue] coordinates {(d,7) };
%\addplot[green, dashed] coordinates {(d,4) };
%\addplot coordinates {(d,1) };
%\legend{used,understood,not understood}
%\end{axis}
%\end{tikzpicture}
\label{fig:newsvendor:conservatism}}
\hspace{1mm}
\subfloat[Asymptotic in-sample cost versus decay rate of out-of-sample disappointment]{% This file was created by matlab2tikz.
% Minimal pgfplots version: 1.3
%  http://www.mathworks.com/matlabcentral/fileexchange/22022-matlab2tikz
\definecolor{mycolor1}{rgb}{0.00000,0.44700,0.74100}%
\begin{tikzpicture}

\begin{axis}[%
width=1.4in,
height=1.4in,
at={(1.011111in,0.41667in)},
scale only axis,
xlabel={\footnotesize{$\lim\limits_{T\to\infty}-\frac{1}{T}\log \mb P_{\star} [ c(\widehat x_T, {\theta_\star}) > \widehat c_T(\widehat x_T) ]$}},
ylabel={\footnotesize{$\lim\limits_{T\to\infty} \mb E_{\mathbb{P}_\star}[\widehat c_T(\widehat x_T)]$}},
y label style={yshift=-0.75em},
xmin=0.00,
xmax=0.12,
xtick = {0,0.02,0.05,0.1},
xticklabels = {0,2\%,5\%,10\%},
ymin=-8.6,
ymax=-5.8,
scaled ticks=false, tick label style={/pgf/number format/fixed},
legend style={legend cell align=left,align=left,draw=white!15!black}
]

%KL
%%%%%%%%
%large plot
\addplot [color=magenta,only marks,mark=*,
mark options={scale=1.2, fill=magenta},
forget plot]
  table[row sep=crcr]{%
0.1171	-6.60\\
};
% %small plots
% \addplot [color=magenta,only marks,mark=*,
% mark options={scale=0.5, fill=magenta},
% forget plot]
%   table[row sep=crcr]{%
% 0.0110  -5.90\\
% 0.0202  -5.68\\
% 0.044	-5.363\\
% 0.0595  -5.093\\
% 0.0685  -4.827\\
% 0.0777  -4.70\\
% 0.0809  -4.67\\
% 0.0863  -4.66\\
% 0.0903 -4.64\\
% 0.1141  -4.65\\
% 0.1148  -4.644\\
% };

%interpolation line
\addplot [color=magenta,style=densely dotted, line width = 1.2pt, smooth]
  table[row sep=crcr]{%
0	-8.29921852664511\\
0.00789473684210526	-8.057757190922\\
0.0157894736842105	-7.80980474826911\\
0.0236842105263158	-7.63536119868644\\
0.0315789473684211	-7.47442654217398\\
0.0394736842105263	-7.32700077873173\\
0.0473684210526316	-7.1930839083597\\
0.0552631578947368	-7.07267593105789\\
0.0631578947368421	-6.96577684682629\\
0.0710526315789474	-6.87238665566491\\
0.0789473684210526	-6.79250535757374\\
0.0868421052631579	-6.72613295255278\\
0.0947368421052631	-6.67326944060205\\
0.102631578947368	-6.63391482172152\\
0.110526315789474	-6.60806909591122\\
0.118421052631579	-6.59573226317112\\
0.126315789473684	-6.59690432350125\\
0.134210526315789	-6.61158527690159\\
0.142105263157895	-6.63977512337214\\
0.15	-6.68147386291291\\
};

%Moment
%%%%%%%%
%Large plots
\addplot [color=green,only marks,mark=*,
mark options={scale=1.2, fill=green},
forget plot]
  table[row sep=crcr]{%
0.0233	-7.055\\
};

\addplot [color=cyan,only marks,mark=*,
mark options={scale=1.2, fill=cyan},
forget plot]
  table[row sep=crcr]{%
0.048	-6.42 \\
};

\addplot [color=brown,only marks,mark=*,
mark options={scale=1.2, fill=brown},
forget plot]
  table[row sep=crcr]{%
0.0	-8.33 \\
};

\addplot [color=blue,only marks,mark=*,
mark options={scale=1.2, fill=blue},
forget plot]
  table[row sep=crcr]{%
0.0766	-5.97\\
};

% %Neural Network
% \addplot [color=gray,only marks,mark=*,
% mark options={scale=1.2, fill=gray},
% forget plot]
%   table[row sep=crcr]{%
% %0.0075	-8.031\\
% %0.0107 -7.531\\
% 0.0366 -6.949\\
% %0.046 -6.531\\
% };

% %small plots
% \addplot [color=blue,only marks,mark=*,
% mark options={scale=0.5, fill=blue},
% forget plot]
%   table[row sep=crcr]{%
% 0.0283	-4.91\\
% 0.0360	-4.688\\
% 0.0611	-4.183\\
% 0.0258  -4.982\\
% 0.0447  -4.544\\
% 0.0667  -4.114\\
% 0.0521  -4.395\\
% 0.0599  -4.248\\
% 0.0833  -3.895\\
% 0.0306  -4.83\\
% 0.0333  -4.76\\
% 0.0399  -4.61\\
% 0.0485  -4.47\\
% 0.0733  -4.04\\
% 0.0846  -3.84\\
% };

%line plot - using some interpolation
\addplot [color=blue,style=densely dotted, line width = 1.5pt, smooth]
  table[row sep=crcr]{%
0	-7.79002355531436\\
0.00789473684210526	-7.51911178500722\\
0.0157894736842105	-7.26745189998762\\
0.0236842105263158	-7.03504390025555\\
0.0315789473684211	-6.82188778581101\\
0.0394736842105263	-6.62798355665401\\
0.0473684210526316	-6.45333121278454\\
0.0552631578947368	-6.2979307542026\\
0.0631578947368421	-6.1617821809082\\
0.0710526315789474	-6.04488549290133\\
0.0789473684210526	-5.94724069018199\\
0.0868421052631579	-5.86884777275019\\
0.0947368421052631	-5.80970674060591\\
0.102631578947368	-5.76981759374918\\
0.110526315789474	-5.74918033217997\\
0.118421052631579	-5.7477949558983\\
};

%WASSERSTEIN
%%%%%%%%%%%%%

%Large plot
\addplot [color=orange,only marks,mark=*,
mark options={scale=1.2, fill=orange},
forget plot]
  table[row sep=crcr]{%
0.0766	-6.3\\
};

% %small plot
% \addplot [color=orange,only marks,mark=*,
% mark options={scale=0.5, fill=orange},
% forget plot]
%   table[row sep=crcr]{%
% 0.0086  -6.0\\  
% 0.0245  -5.635\\  
% 0.0290  -5.292\\
% 0.0453	-4.94\\
% 0.0556  -4.73\\
% %0.0562  -4.80\\
% 0.0594  -4.52\\
% 0.0614  -4.44\\
% 0.0757  -4.31\\
% 0.0819  -4.234\\
% 0.0863  -4.06\\ 
% 0.1028  -3.894\\
% };

%interpolation
\addplot [color=orange,style=densely dotted, line width = 1.5pt, smooth]
  table[row sep=crcr]{%
0	-8.31126778615304\\
0.00789473684210526	-8.04142042798743\\
0.0157894736842105	-7.78550851544254\\
0.0236842105263158	-7.54353204851838\\
0.0315789473684211	-7.31549102721494\\
0.0394736842105263	-7.10138545153223\\
0.0473684210526316	-6.90121532147025\\
0.0552631578947368	-6.71498063702899\\
0.0631578947368421	-6.54268139820845\\
0.0710526315789474	-6.38431760500864\\
0.0789473684210526	-6.23988925742955\\
0.0868421052631579	-6.10939635547119\\
0.0947368421052631	-5.99283889913356\\
0.102631578947368	-5.89021688841664\\
0.110526315789474	-5.80153032332046\\
0.118421052631579	-5.726779203845\\
0.126315789473684	-5.66596352999026\\
0.134210526315789	-5.61908330175625\\
0.142105263157895	-5.58613851914296\\
0.15	-5.5671291821504\\
};

%SAA plot
\addplot [color=brown,style=densely dotted, line width = 1.5pt, smooth]
  table[row sep=crcr]{%
0	-8.37691106727298\\
0.00789473684210526	-8.01185126941229\\
0.0157894736842105	-7.67208518188246\\
0.0236842105263158	-7.35761280468348\\
0.0315789473684211	-7.06843413781535\\
0.0394736842105263	-6.80454918127807\\
0.0473684210526316	-6.56595793507165\\
0.0552631578947368	-6.35266039919607\\
0.0631578947368421	-6.16465657365135\\
0.0710526315789474	-6.00194645843749\\
0.0789473684210526	-5.86453005355447\\
0.0868421052631579	-5.75240735900231\\
0.0947368421052631	-5.665578374781\\
0.102631578947368	-5.60404310089055\\
0.110526315789474	-5.56780153733095\\
};

% %NN plot
% \addplot [color=gray,style=densely dotted, line width = 1.5pt, smooth]
%   table[row sep=crcr]{%
% 0	-8.47130363453205\\
% 0.00526315789473684	-8.22389596476318\\
% 0.0105263157894737	-7.99051289180804\\
% 0.0157894736842105	-7.77115441566662\\
% 0.0210526315789474	-7.56582053633893\\
% 0.0263157894736842	-7.37451125382496\\
% 0.0315789473684211	-7.19722656812473\\
% 0.0368421052631579	-7.03396647923822\\
% 0.0421052631578947	-6.88473098716543\\
% 0.0473684210526316	-6.74952009190638\\
% 0.0526315789473684	-6.62833379346105\\
% 0.0578947368421053	-6.52117209182945\\
% 0.0631578947368421	-6.42803498701158\\
% 0.068421052631579	-6.34892247900743\\
% 0.0736842105263158	-6.28383456781701\\
% 0.0789473684210526	-6.23277125344032\\
% 0.0842105263157895	-6.19573253587736\\
% 0.0894736842105263	-6.17271841512812\\
% 0.0947368421052632	-6.16372889119261\\
% 0.1	-6.16876396407083\\
% };

%FEASIBILITY LINE
%large plot
\addplot [color=black,dashed]
  table[row sep=crcr]{%
0.05	-20\\
0.05	-5\\
};

% %shade infeasible area
%     \addplot[color=black,fill=black, pattern=north east lines,  domain=0:0.05,samples=100] {x-8.5} \closedcycle;

\end{axis}
\end{tikzpicture}%
\label{fig:newsvendor:comparison}} \newline
\ref{named}
\caption[]{Statistical properties of different cost and decision estimators for a data-driven newsvendor problem with ordering cost~$k=5$ and retail price~$p=7$, where the demand~$\xi$ follows a shifted binomial distribution with~$10$ trials, success probability~$0.5$ and shift~$1$. All probabilities and expectations involving random training data are evaluated empirically using~$10^4$ independent training sets. %The curves in Figures~\ref{fig:newsvendor:disappointment} and~\ref{fig:newsvendor:conservatism} correspond to the same estimators as the dots in Figure~\ref{fig:newsvendor:comparison} with the corresponding colors. %Figure~\ref{fig:newsvendor:disappointment} shows the out-of-sample disappointment $\mb P_{\star} [ c(\widehat x_T, {\theta_\star}) > \widehat c_T(\widehat x_T) ]$ and their corresponding conservatism $\lim_{T\to\infty} \mb E_{\mathbb{P}_\star}[\widehat c_T(\widehat x_T)]$ is shown in Figure~\ref{fig:newsvendor:conservatism} for five different data-driven predictors $\widehat c_T$ and prescriptors $\widehat x_T$. Figure~\ref{fig:newsvendor:comparison} then compares the four methods regarding both out-of-sample disappointment decay rate and conservatism.
%The shaded area in Figure~\ref{fig:newsvendor:comparison} characterizes all instances whose corresponding out-of-sample disappointment probabilities decay slower than the desired rate of $5\%$.
}
\label{fig:newsvendor}
\end{figure}
%%%%%%%%%%%%

% \begin{figure}
%   \input{newsvendor_comparison.tex}\label{fig:newsvendor:comparison}
%   \caption{Trade-off}
% \end{figure}

Figure~\ref{fig:newsvendor} visualizes the out-of-sample disappointment and the expected in-sample cost as well as the trade-off between the asymptotic in-sample cost and the decay rate of the out-of-sample disappointment for different estimators. Figure~\ref{fig:newsvendor:disappointment} shows that, as a function of~$T$, the out-of-sample disappointment always traces out an almost perfect straight line on a logarithmic scale. This observation suggests that the out-of-sample disappointment decays exponentially and is therefore faithfully represented by its decay rate. 

The solid lines in Figures~\ref{fig:newsvendor:disappointment} and~\ref{fig:newsvendor:conservatism} correspond to the empirical cost estimator (light brown) and to distributionally robust cost estimators with a moment ambiguity set
(green: $\varepsilon=0.05$, light blue: $\varepsilon=0.13$, dark blue: $\varepsilon=0.2$), a Wasserstein ambiguity set (orange: $\varepsilon=0.28$) and a relative entropy ambiguity set (magenta: $\varepsilon=0.12$). The $\varepsilon$ hyperparameters are chosen to ensure efficient use of the available plotting area. As expected, the empirical cost estimator is the most optimistic one in the sense that it displays the lowest in-sample cost, but its out-of-sample disappointment fails to decay. Any distributionally robust cost estimator becomes increasingly pessimistic as the size parameter~$\varepsilon$ of the underlying ambiguity set increases. The dashed lines in Figure~\ref{fig:newsvendor:comparison} visualize the trade-off between the asymptotic in-sample cost and the decay rate of the out-of-sample disappointment for the na\"ive penalized empirical cost estimator~$\widehat c(x)=c(x,\widehat S_T)+\varepsilon$ (light brown) and for the distributionally robust cost estimators with a moment ambiguity set (dark blue), a Wasserstein ambiguity set (orange) and a relative entropy ambiguity set (magenta) as~$\varepsilon$ is swept. The six dots in Figure~\ref{fig:newsvendor:comparison} correspond to the six estimators investigated in Figures~\ref{fig:newsvendor:disappointment} and~\ref{fig:newsvendor:conservatism}. As expected, the dashed lines corresponding to the distributionally robust cost estimators with a Wasserstein and a relative entropy ambiguity set intersect because both of these estimators reduce to the empirical cost estimator for~$\varepsilon=0$. Maybe surprisingly, the distributionally robust cost estimators associated with the relative entropy ambiguity set dominate those associated with the Wasserstein ambiguity set and even more so those associated with the moment ambiguity set, that is, their asymptotic in-sample cost is lowest for any fixed decay rate of the out-of-sample disappointment. They also dominate the penalized empirical cost estimators. It is now natural to ask whether there exists a globally {\em least conservative} cost estimator whose asymptotic in-sample risk is minimal across {\em all} conceivable cost estimators (not necessarily only distributionally robust ones) with a prescribed decay rate of the out-of-sample disappointment. 
For example, if we require a decay rate of at least $5\%$, all cost estimators on the right hand side of the vertical dashed line in Figure~\ref{fig:newsvendor:comparison} are feasible. A simple line search reveals that this includes all penalized empirical cost estimators with penalty $\varepsilon\geq 1.9$, all distributionally robust cost estimators with a moment ambiguity set of size~$\varepsilon\geq 0.14$, all distributionally robust cost estimators with a Wasserstein ambiguity set of radius $\varepsilon\geq 0.23$ and all distributionally robust cost estimators with a relative entropy ambiguity set of radius $\varepsilon\geq 0.05$. But many other estimators not considered in this experiment are feasible, too. We endeavor to identify the least conservative of {\em all} such feasible estimators. In the remainder we address this fundamental challenge under significantly more~general~conditions.

%%%%%%%%%%%%%%%%%%%%%%%%%%%%%%
%    SSEC. Data-driven prediction & prescription
%%%%%%%%%%%%%%%%%%%%%%%%%%%%%%

\subsection{Data-driven predictors and prescriptors} \label{ssec:dd_predictions_prescriptions}

%equip it with the subspace topology induced by the Euclidean topology on~$\mb R^n$. Recall that a set~$A\subseteq X$ is open in the subspace topology on~$X$ if~$A=X\cap O$ for some set~$O\subseteq \mb R^n$ that is open in the Euclidean topology on~$\mb R^n$.

We now return to the general stochastic optimization problem~\eqref{eq:initial:problem-f}, and we assume that the unknown probability measure~$\mb P_\star$ must be learned from a finite sample path of a stochastic process $\{\xi_t\}_{t\in\mb N}$ with state space~$\Xi\subseteq \Re^m$. Like any random object, this data-generating stochastic process is defined on the measurable space~$(\Omega, \mc F)$. From now on we assume that even though the probability measure~$\mb P_\star$ is unknown, it belongs to a known finitely parametrized ambiguity set. This premise is formalized in the following assumption.

\begin{assumption}[Finitely parametrized ambiguity set]
\label{ass:parametrization}
The probability measure $\mb P_\star$ belongs to a finitely parametrized ambiguity set $\mc P=\{\mb P_{\theta}:\theta \in \Theta\}$, where $\Theta$ is the relative interior of a convex subset of the finite-dimensional parameter space~$\Re^d$, and~$\mb P_\theta$ is a probability measure on~$(\Omega,\mc F)$ for every~$\theta\in\Theta$. %{\color{red} We equip~$\Theta$ with the subspace topology induced by the Euclidean topology on~$\mb R^d$.}
%\item \label{item2:ass:model:class} {\color{red} The stochastic process $\{\xi_t\}_{t\in\mb N}$ converges weakly to a stationary distribution $F_\theta$. }
% \item \label{item2:ass:model:class}{\color{red}The stochastic process $\{\xi_t\}_{t\in\mb N}$ is stationary under any model, that is, for any given~$\theta\in\Theta$ there exists a cumulative distribution function $F_\theta$ with $\mb P_\theta[\xi_t\leq z]=F_\theta(z)$ for all $z\in\mb R^{m}$ and $t\in\mb N$.}
\end{assumption}
As each~$\theta\in\Theta$ encodes a unique probabilistic model~$\mb P_\theta$, for ease of terminology, we will henceforth refer to~$\theta$ as a {\em model} and to~$\Theta$ as the {\em model space}. The ambiguity set~$\mc P$ is meant to capture all structural information on~$\mb P_\star$ that is available before observing any statistical data. This justifies our assumption that~$\mc P$ is known to contain the probability measure~$\mb P_\star$ with certainty (and not only with high confidence). Assumption~\ref{ass:parametrization} thus implies that there exists a model~$\theta_\star\in\Theta$ with~$\mb P_{\theta_\star}=\mb P_\star$.

To provide some intuition for the abstract concepts introduced in this paper, we use the class of \ac{iid}~stochastic processes with a finite state space as a running example. This example will further show that the approach to data-driven decision-making developed in~\cite{ref:vanParys:fromdata-17} emerges as a simple special case of a considerably more general framework. Several alternative data generation processes will be discussed in Section~\ref{sec:models}. % in order to showcase the generality of the proposed framework.

\begin{example}[Ambiguity set for finite state i.i.d.\ processes] \label{ex:finitestate:iid:part1}
%As in the newsvendor problem discussed in Section~\ref{sec:newsvendor},
Assume that $\Xi=\{1,\ldots,d\}$, the random variables $\xi_t$ are serially independent under $\mb P_\star$ and $\mb P_\star[\xi_t=i]=(\theta_\star)_{i}>0$ for all $i\in\Xi$ and $t\in\mb N$. The vector~$\theta_\star$ thus encodes the unknown probability mass function of $\xi_t$, which is independent of~$t$. These assumptions imply that $\mb P_\star$ belongs to an ambiguity set of the form $\mc P=\{\mb P_\theta:\theta\in\Theta \}$, where $\Theta =\{\theta\in\mb R^d_{++}: \sum_{i=1}^d \theta_i=1\}$ is the positive probability simplex, and each $\theta$ encodes a probability measure $\mb P_\theta$ on $(\Omega, \mc F)$ satisfying
\(
	\mb P_\theta [ \xi_t=i_t~\forall t=1,\ldots,T] = \prod_{t=1}^T \theta_{i_t}~ \forall i_t\in \Xi , \, t=1,\ldots,T,\,T\in\mb N.
\)
\end{example}

We now embed the original stochastic optimization problem~\eqref{eq:initial:problem-f} into a family of problems corresponding to the probability measures~$\mb P_\theta$, $\theta\in\Theta$. Therefore, by slightly abusing notation with the goal to avoid clutter, we henceforth parametrize the objective function of problem~\eqref{eq:initial:problem-f} by~$\theta$ instead of~$\mb P_\theta$.

%To this end, we define model-based predictors and prescriptors.

\begin{definition}[Model-based predictors and prescriptors]
\label{def:parametric-predictor}
For any fixed model $\theta\in \Theta$, we define the model-based predictor~$c(x,\theta)$ as the objective function of problem~\eqref{eq:initial:problem-f} when~$\mb P_\star$ is replaced with~$\mb P_\theta$ and the corresponding model-based prescriptor $x^\star(\theta)\in\arg\min_{x\in X}c(x,\theta)$ as a decision that minimizes $c(x,\theta)$ over $x\in X$. 
\end{definition}

The stochastic program~\eqref{eq:initial:problem-f} can now be identified with the {\em prescription problem} of computing $x^\star(\theta_\star)$. Similarly, the evaluation of the objective function of a given decision~$x\in X$ in~\eqref{eq:initial:problem-f} can be identified with the {\em prediction problem} of computing $c(x,\theta_\star)$. In the remainder we impose the following regularity condition.

\begin{assumption}[Uniform continuity and boundedness of the model-based predictor]
\label{ass:continuity}
The model-based predictor~$c(x,\theta)$ is uniformly continuous and bounded on $X\times \Theta$.
\end{assumption}

Note that if~$c(x,\theta)$ is uniformly continuous and bounded on $X\times \Theta$, then it admits a unique uniformly continuous and bounded extension to~$X\times\cl\Theta$ \cite[Theorem~5.15]{ref:aliprantis-07}. By slight abuse of notation, we will denote this extension by~$c(x,\theta)$, too. Assumption~\ref{ass:continuity} is trivially satisfied by the newsvendor problem of Section~\ref{sec:newsvendor}.
%For stochastic programming problems~\eqref{eq:prescription_problem}, if the cost function~$\gamma(x,\xi)$ is uniformly continuous and bounded on~$X\times\Xi$, and if~$F_\theta$ represents a uniformly continuous function from~$\Theta$ to the space of distribution functions equipped with the weak topology, then Assumption~\ref{ass:continuity} is automatically satisfied. 
As neither the model-based predictor~$c(x, \theta_\star)$ nor the model-based prescriptor~$x^\star(\theta_\star)$ can be evaluated for the unknown true model~$\theta_\star$, we will now approximate them by functions of the available data. In order to formally define data-driven predictors and prescriptors, we denote by $\xivec=(\xi_1,\hdots,\xi_T)$ the history of the data-generating process up to time~$T$, and we let $\mc F_T\subseteq \mc F$ be the $\sigma$-algebra generated by~$\xivec$ for any $T\in\mb N$. We also use $\mb E_\theta[\cdot]$ to denote the expectation operator with respect to $\mb P_\theta$ for any model~$\theta\in\Theta$. 

%denote by $(\xi_1,\xi_2,\hdots)\in\Xi^\infty$ a sample path drawn at random under the probability measure~$\mb P^\star=\mb P_{\theta^\star}$. Moreover, for any $T\in\N$, we denote by $\xivec=(\xi_1,\hdots,\xi_T)$ the subpath of the first $T$ samples in $\xi$. Finally, we let $\mc F_T\subseteq \mc F$ be the $\sigma$-algebra generated by $\xivec$.

%We can now at last present our formal definition of what constitutes data-driven decision-making.

\begin{definition}[Data-driven predictors]
\label{def:dd_prediction}
A decision-dependent stochastic process $\widehat c=\{\widehat c_T(x)\}_{T\in\N,\,x\in X}$ valued in $\mb R$ is called a data-driven predictor if it satisfies the following conditions.
\begin{enumerate}[(i)]
  \setlength\itemsep{0em}
\item \label{def:dd_prediction:0} \textbf{Continuity in the decisions.} The random variable $\widehat c_T(x)$ is continuous in~$x\in X$ for all $T\in\mb N$.
\item \label{def:dd_prediction:i} \textbf{Non-anticipativity.} The process $\{\widehat c_T(x)\}_{T\in\N}$ is adapted to the filtration $\{\mc F_T\}_{T\in\N}$ for every $x\in X$.
%\item \label{def:dd_prediction:i:nonnegativity} \textbf{Non-negativity:} the process $\{\widehat c_T(x)\}_{T\in\N}$ is non-negative for every $x\in X$;
\item \label{def:dd_prediction:ii:boundedness} \textbf{Uniform integrability.} There exists a non-negative random variable $\overline c$ such that $\mb E_\theta[\bar c]<\infty$ for all $\theta\in\Theta$ and $|\widehat c_T(x)|\leq \overline c$ $\mb P_\theta$-almost surely for all $T\in\N$, $x\in X$ and $\theta\in\Theta$.
\item \label{def:dd_prediction:iii:convergence} \textbf{Convergence of objective.} There exists a deterministic Borel-measurable function $c_\infty:X\times\Theta\to\Re$ such that, as $T$ grows, $\widehat c_T(x)$ converges in probability under $\mb P_\theta$ to $c_\infty(x,\theta)$ for every $x\in X$ and $\theta\in\Theta$.
\item \label{def:dd_pres:iii:convergence} \textbf{Convergence of optimal value.} There exists a deterministic Borel-measurable function $v_\infty:\Theta\to\Re$ such that, as $T$ grows, $\min_{x\in X}\widehat c_T(x)$ converges in probability under $\mathbb P_\theta$ to $v_\infty(\theta)$ for every $\theta\in\Theta$.
% \begin{equation*}
%   \forall \varepsilon>0: \quad \lim_{T\to\infty}\mb P_{\theta}[|\widehat c_T(x) - c_\infty(x,\theta)|<\varepsilon ]=1.
% \end{equation*}
\end{enumerate}
\end{definition}

If we use data-driven predictors as accessible proxies for inaccessible model-based predictors, then it is reasonable to assume that they share all known properties of the model-based predictors. The continuity condition~\ref{def:dd_prediction:0} in Definition~\ref{def:dd_prediction} is thus a natural consequence of Assumption~\ref{ass:continuity}. In addition, as~$X$ is compact, this condition guarantees that the data-driven decision problem~$\min_{x\in X} \widehat c_T(x)$ is sovlable for every~$T\in\mathbb N$. The non-anticipativity condition~\ref{def:dd_prediction:i} implies via \cite[Theorem~5.4.2]{ref:Ash-00} that for any $T\in\N$ there exists a Borel-measurable function~$f_T:X\times\Xi^T\to\Re$ with $\widehat c_T(x) = f_T(x,\xivec)$. This means that~$\widehat c_T(x)$ may depend only on the history~$\xivec$ of the data-generating process observed up to time $T$. The uniform integrability condition~\ref{def:dd_prediction:ii:boundedness} is of technical nature and non-restrictive in all examples studied in this paper. % and is needed {\color{red}to control the limit of the data-driven predictor's expected value.} 
The convergence condition~\ref{def:dd_prediction:iii:convergence} implies that for any fixed $\theta\in\Theta$, the predictor~$\widehat c_T(x)$ represents a consistent estimator for $c_\infty(x,\theta)$ if the data is generated under $\mb P_\theta$. Note that we explicitly allow for the possibility that $c_\infty(x,\theta)\neq c(x,\theta)$, that is, $\widehat c_T(x)$ may in fact be a {\em biased} estimator for the model-based predictor $c(x,\theta)$. Similarly, the convergence condition~\ref{def:dd_pres:iii:convergence} implies that for any fixed $\theta\in\Theta$ the optimal value $\widehat c_T(\widehat x_T)$ of the data-driven optimization problem $\min_{x\in X} \widehat c_T(x)$ represents a consistent estimator for~$v_\infty(\theta)$ if the data is generated under $\mb P_\theta$. %Thus, it is asymptotically deterministic. 
Thus, it may be a {\em biased} estimator for the optimal value of the stochastic optimization problem $\min_{x\in X}c(x,\theta)$.
% {\color{magenta}We emphasize that condition~\ref{def:dd_pres:iii:convergence} is not implied by conditions~\ref{def:dd_prediction:0}--\ref{def:dd_prediction:iii:convergence}. To see this, set~$X=[0,1]$, and define a degenerate decision-dependent process~$\widehat c$ through $\widehat c_T(x)=\min\{1,T|x-1/T|\}$ if $T$ is odd and $\widehat c_T(x)=1$ if~$T$ is even. Thus, there is no randomness in~$\widehat c$, and conditions~\ref{def:dd_prediction:0}--\ref{def:dd_prediction:ii:boundedness} are trivially satisfied. Condition~\ref{def:dd_prediction:iii:convergence} is also satisfied if we set~$c_\infty(x,\theta)=1$. As $\min_{x\in X} \widehat c_T(x)$ evaluates deterministically to~$0$ if~$T$ is odd and to $1$ if~$T$ is even, however, condition~\ref{def:dd_pres:iii:convergence} is violated. This example proves that condition~\ref{def:dd_pres:iii:convergence} is not redundant.\footnote{Similarly, one can show that if all conditions of Definition~\ref{def:dd_prediction} hold, then $v_\infty(\theta)$ may differ from $\min_{x\in X}c_\infty(x,\theta)$.}}
  From now on we denote the set of all data-driven predictors in the sense of Definition~\ref{def:dd_prediction} by~$\widehat{\mathcal{C}}$.

\begin{definition}[Data-driven prescriptors]
\label{def:dd_prescription}
A stochastic process $\widehat x=\{\widehat x_T\}_{T\in\N}$  valued in $X$ is called a data-driven prescriptor if it satisfies the following conditions.
\begin{enumerate}[(i)]
  \setlength\itemsep{0em}
\item \label{def:dd_prescription:i} \textbf{Non-anticipativity.} The process $\{\widehat x_T\}_{T\in\N}$ is adapted to the filtration $\{\mc F_T\}_{T\in\N}$.
\item \label{def:dd_prescription:ii}
\textbf{Compatibility with a data-driven predictor.} There exists a data-driven predictor $\widehat c$ that induces the data-driven prescriptor $\widehat x$ in the sense that $\widehat x_T \in  \arg\min_{x\in X}
\widehat c_T(x)$ for all $T\in\N$.
%\begin{subequations}
% \begin{align}
%   \label{eq:def:quasi:continuous:prescriptor:data:driven}
%   \widehat x_T \in  \arg&\min_{x\in X}\, \widehat c_T(x) \quad \forall  T\in\N. \\[0.5em]
%   \widehat v_T \defn \widehat c_T(\widehat x_T) =  &\min_{x\in X}\, \widehat c_T(x) \quad \forall  T\in\N
%\end{align}
%\end{subequations}
% \begin{equation*}
%   \forall \varepsilon>0: \quad \lim_{T\to\infty}\mb P_{\theta}[|\widehat v_T - v_\infty(\theta)|<\varepsilon ]=1.
% \end{equation*}
\end{enumerate}
\end{definition}

The non-anticipativity condition~\ref{def:dd_prescription:i} implies via \cite[Theorem~5.4.2]{ref:Ash-00} that for any $T\in\N$ there exists a Borel-measurable function $g_T:\Xi^T\to\Re$ with $\widehat x_T = g_T(\xivec)$. The compatibility condition~\ref{def:dd_prescription:ii} requires that any data-driven prescriptor is a minimizer of some data-driven predictor. %We will argue below that this requirement is non-restrictive.
From now on we use~$\widehat{\mathcal{X}}$ to denote the set of all data-driven predictor-prescriptor-pairs of the form~$(\widehat c,\widehat x)$, where $\widehat x$ is induced by $\widehat c$.

One can show that {\em any} data-driven predictor~$\widehat c$ induces a (not necessarily unique) data-driven prescriptor~$\widehat x$. The reason for this is that since $X$ is compact and since $\widehat c_T(x)$ depends continuously on~$x\in X$ and represents an $\mc F_T$-measurable random variable for every fixed~$x$, there exists an $\mc F_T$-measurable random vector $\widehat x_T\in \arg\min_{x\in X} \widehat c_T(x)$ thanks to \cite[Theorem~14.37]{rockafellar1998variational}. Combining~$\widehat x_T$ for all~$T\in\mb N$ yields the desired prescriptor.

We emphasize that essentially any procedure for mapping the available data to an `asymptotically deterministic' feasible decision defines a data-driven prescriptor. Indeed, if a stochastic process $\widehat x=\{\widehat x_T\}_{T\in\N}$ with state space~$X$ is adapted to the filtration~$\{\mc F_T\}_{T\in\N}$ and converges in probability under $\mb P_\theta$ to some deterministic Borel-measurable function $x_\infty(\theta)$ for every $\theta\in\Theta$, then one readily verifies that the decision-dependent stochastic process $\widehat c=\{\widehat c_T(x)\}_{T\in\N}$ defined through $\widehat c_T(x)=\min\{1,\|x-\widehat x_T\|_2\}$ for all~$T\in\mb N$ is a data-driven predictor in the sense of Definition~\ref{def:dd_prediction} that induces~$\widehat x$. %In addition, in this case $\widehat c_T(x)$ converges in probability under $\mb P_\theta$ to the deterministic function $c_\infty(x,\theta)=\min\{1,\| x-x_\infty(\theta)\|_2\}$ for every $x\in X$ and $\theta\in \Theta$.
This example shows that the notion of a data-driven prescriptor is very general. Moreover, Definition~\ref{def:dd_prescription} does not even require $\widehat x_T$ to converge.

%While conditions~\ref{def:dd_prediction:i}, \ref{def:dd_prediction:ii:boundedness} and \ref{def:dd_prediction:iii:convergence} are relatively easy to verify, the ultimate condition~\ref{def:dd_pres:iii:convergence} can be non-trivial to establish given an arbitrary data-driven predictor $\widehat c$. However, in Section \ref{sec:compression:scheme} we will identify a large subclass of prescriptors for which this assumption is verified immediately.

%Any data-driven prescriptor $\widehat x_T$ could be interpreted as an estimator for $x^\star(\theta^\star)$. However, this is not strictly necessary and we emphasize that the class of data-driven predictors and prescriptors is extremely rich. Indeed, for any fixed $T\in\N$, $\widehat c_T(x)$ may be any arbitrary measurable function $f_T$ of $x$ and $\xivec$, while $\widehat x_T$ may be essentially represented by any function $g_T$ of $\xivec$ that can be expressed as a parametric minimizer of some measurable function $f_T(x,\xivec)$. For example, it coincides with the unique parametric minimizer of the function $f_T(x,\xivec)=|x-g_T(\xivec)|$. 

%The boundedness condition \ref{def:dd_prediction:ii:boundedness} is natural since the model based predictor $c(x,\theta)$ is assumed to be bounded, i.e., see Assumption~\ref{ass:continuity}. The convergence in probability condition \ref{def:dd_prediction:iii:convergence} guarantees that as more data is revealed, our prediction ultimately converges to a deterministic cost. 

\begin{example}[Empirical predictor for finite state i.i.d.\ processes]
\label{eq:naive-predictor}
In the context of Example~\ref{ex:finitestate:iid:part1}, assume that the model-based predictor represents an expected loss, that is,  set~$c(x,\theta)=\mb E_\theta[\ell(x,\xi)]$ %$\sum_{i\in\Xi} \ell(x,\xi)\,\theta_i$
for some loss function~$\ell(x,\xi)$ that is continuous in~$x$ and bounded on~$X\times\Xi$, and assume that the random variable~$\xi$ has the same distribution as the i.i.d.\ training samples~$\{\xi_t\}_{t\in\N}$. The newsvendor problem of Section~\ref{sec:newsvendor} satisfies all of these assumptions. %Thus, $c(x,\theta)$ represents the expected loss of~$x$ under~$\mb P_\theta$. 
We now define the empirical predictor~$\widehat c$ through 
\(
	\widehat c_T(x) = \frac{1}{T} \sum_{t=1}^T \ell(x,\xi_t)~ \forall T\in\mb N.
\)
Note that the empirical predictor simply evaluates the sample average of the loss across the observed dataset and represents a data-driven predictor in the sense of Definition~\ref{def:dd_prediction}. While the continuity condition~\ref{def:dd_prediction:0} and the non-anticipativity assumption~\ref{def:dd_prediction:i} hold by construction, the boundedness condition~\ref{def:dd_prediction:ii:boundedness} holds because~$\widehat c_T(x)$ is trivially bounded by the finite constant $\overline c=\max_{x\in X,\, \xi\in\Xi} |\ell(x,\xi)|$ for every $T\in\mb N$. The convergence condition~\ref{def:dd_prediction:iii:convergence} follows by setting $c_\infty(x,\theta) = c(x,\theta)$ and observing that
$\lim_{T\to\infty}|\widehat c_T(x) - c(x,\theta)| =0$ $\mb P_\theta$-almost surely for all $\theta\in\Theta$ thanks to the strong law of large numbers. Note also that $c_\infty(x,\theta)$ is continuous by Assumption~\ref{ass:continuity}.
As $\ell$ is continuous and $X$ is compact, the uniform law of large numbers \cite[Lemma~2.4]{ref:Newey-94} further guarantees that $\lim_{T\rightarrow \infty}\sup_{x\in X} \|\frac{1}{T} \sum_{t=1}^T \ell(x,\xi_t) - c(x,\theta)\| = 0$ $\mb P_\theta$-almost surely for all~$\theta\in\Theta$. Therefore, the convergence condition~\ref{def:dd_pres:iii:convergence} is satisfied if we set $v_\infty(\theta)=\min_{x\in X} c(x,\theta)$.
%$\lim_{T\to\infty}\min_{x\in X} \frac{1}{T} \sum_{t=1}^T \gamma(x,\xi_t) = c(x^\star, \theta):=v_\infty(\theta)$ $\mb P_\theta$-almost surely, where $x^\star \in X$, which implies condition.
\end{example}

We will now investigate sequences of surrogate decision problems of the form~$\min_{x\in X} \widehat c_T(x)$ indexed by~$T\in\mb N$, where $\widehat c$ is a data-driven predictor. As the set~$\widehat{\mc C}$ of admissible predictors is vast, there are endless possibilities to design such surrogate decision problems. An ideal design should have the following property for {\em any} model~$\theta\in \Theta$: If the observable data is generated by~$\mb P_\theta$, then the surrogate decision problem~$\min_{x\in X} \widehat c_T(x)$, which must be constructed without knowledge of~$\theta$, should provide a `good' approximation for the stochastic optimization problem~$\min_{x\in X} c(x,\theta)$ corresponding to model~$\theta$. If such an ideal design can be found, it will provide---in particular---a `good' approximation for the actual decision problem corresponding to the unknown true model~$\theta_\star$. Intuitively, a data-driven predictor~$\widehat c$ and the corresponding predictor~$\widehat x$ provide a `good' design if the data-driven objective function~$\widehat c_T(x)$ is close to the function~$c(x,\theta)$ for large~$T$ and if the data-driven decision~$\widehat x_T$ is near-optimal in the decision problem~$\min_{x\in X} c(x,\theta)$ for large~$T$ whenever the data is generated by~$\mb P_\theta$. In the following we will formalize these intuitions.

The key idea is to find the best possible data-driven predictor~$\widehat c$ by solving an optimization problem over~$\widehat{\mc C}$ and to find the best possible data-driven prescriptor~$\widehat x$ by solving an optimization problem over~$\widehat{\mc X}$. As any predictor $\widehat c$ encodes a procedure for transforming data to surrogate optimization problems, an optimization problem over~$\widehat{\mc C}$ can be viewed as an optimization problem over optimization problems. We will therefore refer to it as a {\em meta-optimization model}. As any data-driven prescriptor is induced by a data-driven predictor, an optimization problem over the set~$\widehat{\mc X}$ of predictor-prescriptor pairs can also be viewed as a meta-optimization problem. In the special case when the data is generated by a simple \ac{iid} processes, such meta-optimization problems were already studied in~\cite{ref:vanParys:fromdata-17}. Here, we will show that these ideas have a much wider scope.

To formulate the desired meta-optimization problems, we first need to introduce some terminology. For any fixed model~$\theta\in\Theta$ and data-driven predictor~$\widehat c$, we will henceforth refer to~$\widehat c_T(x)$ as the {\em in-sample risk} and to $c(x,\theta)$ as the {\em out-of-sample risk} of the decision~$x\in X$. Specifically, if~$\widehat x$ is a data-driven prescriptor induced by~$\widehat c$, then $\widehat c_T(\widehat x_T)$ and $c(\widehat x_T,\theta)$ represent the in-sample and out-of-sample risk of~$\widehat x_T$, respectively. We emphasize that the out-of-sample risk under the true model~$\theta_\star$ is the actual quantity of interest as it represents the objective function value of a given candidate decision in the true stochastic optimization problem~\eqref{eq:prescription_problem}. If the data-generating process is ergodic (which is the case for all examples studied Section~\ref{sec:models}), then the out-of-sample risk also coincides almost surely with the average cost incurred of the given candidate decision along an infinitely long sample path. Unfortunately, only the in-sample risk is observable at the time when the decision problem needs to be solved. Of course, the out-of-sample risk can in principle be computed for any model~$\theta\in\Theta$. But the benefits of this capability remain limited as long as~$\theta_\star$ is unknown.

The ideal meta-optimization problem over all data-driven prescriptors would be tailored to the length~$T$ of the available observation history and would minimize the out-of-sample risk $c(\widehat x_T,\theta_\star)$ of $\widehat x_T$ over all~$(\widehat c,\widehat x)\in \widehat {\mc X}$. As the true model~$\theta_\star$ is unknown, however, such an approach would only be successful if there existed a Pareto
dominant prescriptor that minimizes the out-of-sample risk of~$\widehat x_T$ simultaneously for all models~$\theta\in\Theta$ (and thus in particular for~$\theta_\star$). Unfortunately, finding such a Pareto dominant prescriptor seems too ambitious and is probably impossible. This prompts us to work with an alternative notion of optimality. The key idea is to minimize the in-sample risk subject to a constraint that forces the out-of-sample risk to be smaller than or equal to the in-sample risk. As both the in-sample and the out-of-sample risk are random objects, we impose this constraint probabilistically. To this end, we define a notion of {\em out-of-sample disappointment}. 

%Now that we have defined data-driven decision-making as a predictor-prescriptor pair in Definition \ref{def:dd_prediction} we are in a position to axiomatize its performance. The axiomatization framework presented here is an extension of the framework pioneered by \cite{ref:vanParys:fromdata-17} to our more general setup of what constitutes data-driven decision-making and which we motivated in Section \ref{sec:newsvendor} in the context of the newsvendor problem. We characterize performance in terms of a balance between two desirable properties: out-of-sample disappointment and conservatism. 

%In a decision-making context where the goal is to minimize costs, disappointments (underestimated costs) are more harmful than positive surprises (overestimated costs). While statisticians strive for accuracy by minimizing a symmetric estimation error, decision-makers endeavor to limit the one-sided prediction disappointment. 

\begin{definition}[Out-of-sample disappointment]
\label{def:oos-disappointment}
For any data-driven predictor $\widehat c$ the probability 
\begin{subequations}
\label{eq:oos-disappointment}
\(	
	\mb P_\theta [ c(x, \theta) > \widehat c_T(x)]
\)
is referred to as the out-of-sample prediction disappointment of~$x\in X$ at time $T$ under model $\theta\in\Theta$. Similarly, for any data-driven prescriptor $\widehat x$ induced by a data-driven predictor $\widehat c$ the probability 
\(	
	\mb P_\theta [ c(\widehat x_T, \theta) > \widehat c_T(\widehat x_T)]
\)
\end{subequations}
is termed the out-of-sample prescription disappointment at time $T$ under model $\theta\in\Theta$.
\end{definition}

Note that the out-of-sample disappointment represents the probability that the out-of-sample risk strictly exceeds the in-sample risk. Intuitively, a smaller out-of-sample disappointment should be preferred over a large out-of-sample disappointment. For example, in the context of the newsvendor problem studied in Section~\ref{sec:newsvendor}, a high out-of-sample disappointment entailed a high probability of budget overruns. 

The meta-optimization problem to be developed below aims to minimize the in-sample risk. As $\widehat c_T(x)$ for~$x\in X$ as well as~$\widehat c_T(\widehat x_T)$ are random variables, however, this informal objective is not well-defined. The properties of a data-driven predictor laid out in Definition~\ref{def:dd_prediction} further imply that even the {\em expected} in-sample risk is not well-defined. Indeed, if the data is generated under~$\mb P_\theta$, then~$\mb E_\theta [ \widehat c_T(x)]$ converges to~$c_\infty(x,\theta)$ as $T$ grows, where~$c_\infty$ is the Borel-measurable function whose existence is postulated in Definition~\ref{def:dd_prediction}\ref{def:dd_prediction:iii:convergence}. This follows directly from Lemma~\ref{lem:aux:prob}, which applies because of conditions~\ref{def:dd_prediction:ii:boundedness} and~\ref{def:dd_prediction:iii:convergence} in Definition~\ref{def:dd_prediction}. The same lemma implies that~$\mb E_\theta [ \widehat c_T(\widehat x_T)]$ converges to~$v_\infty(\theta)$ as $T$ grows, where~$v_\infty$ is the Borel-measurable function whose existence is postulated in Definition~\ref{def:dd_prediction}\ref{def:dd_pres:iii:convergence}. 

%~\cite[Theorems~2.3.2 and~2.3.3]{durrett_book} and conditions~\ref{def:dd_prediction:ii:boundedness}--\ref{def:dd_pres:iii:convergence} in Definition~\ref{def:dd_prediction}. }

The above reasoning indicates that both the out-of-sample disappointment as well as the expected in-sample risk depend on the data-generating model~$\theta$ and the length~$T$ of the available observation history. As~$\theta$ is unobservable, however, the meta-optimization problem to be developed may {\em not} depend on~$\theta$ for otherwise its solution would not be implementable. Even though~$T$ is known to the decision-maker, we did not manage to construct a meta-optimization problem that adapts to~$T$ and can still be solved. To eliminate the dependence on both~$\theta$ as well as~$T$, we thus propose to minimize the {\em asymptotic} expected in-sample performance of {\em every}~$\theta\in\Theta$ subject to an upper bound on the {\em asymptotic} exponential decay rate of the out-of-sample disappointment for {\em every}~$\theta\in\Theta$. Note that since the proposed meta-optimization problem accommodates multiple objective functions (one for each~$\theta\in \Theta$) and a constraint that must hold for all realizations of the uncertain parameter~$\theta\in\Theta$, it constitutes a {\em robust multi-objective optimization problem}.

The meta-optimization problem for finding the best data-driven predictor can thus be formulated as
\begin{subequations}\label{eq:optimal}
\begin{equation}
\label{eq:optimal-predictor}
	\begin{array}{ll}
	\displaystyle\mathop{\text{minimize}}_{\widehat c\in\widehat{\mathcal{C}}}{} & \left\{ \displaystyle \lim_{T\to\infty} \mb E_\theta \left[\widehat c_T(x) \right] \right\}_{x\in X,\, \theta\in\Theta} \\
	\text{subject to} & \displaystyle \limsup_{T\to\infty} \frac{1}{T} \log \mb P_\theta [ c(x, \theta) > \widehat c_T(x)]\leq -r \quad \forall x\in X,\;\theta\in\Theta.
	\end{array}
\end{equation}
Recall that the asymptotic expected in-sample risk~$\lim_{T\to\infty} \mb E_\theta [\widehat c_T(x)]$ under model~$\theta$ is well-defined and coincides with the limit function~$c_\infty(x,\theta)$ of Definition~\ref{def:dd_prediction}\ref{def:dd_prediction:iii:convergence} for every~$x\in X$ and~$\theta\in\Theta$. The constraint requires that the out-of-sample disappointment under model~$\theta$ satisfies $\mb P_\theta [ c(x, \theta) > \widehat c_T(x)]\leq e^{-rT+o(T)}$ for every~$x\in X$ and~$\theta\in\Theta$, where~$r> 0$ is a risk-aversion parameter chosen by the decision-maker.

Similarly, the meta-optimization problem for finding the best predictor-prescriptor-pair can be formulated~as
\begin{equation}
\label{eq:optimal-prescriptor}
	\begin{array}{ll}
	\displaystyle\mathop{\text{minimize}}_{(\widehat c,\widehat x)\in{\widehat{\mathcal{X}}}}{}_{} &  \left\{ \displaystyle \lim_{T\to\infty} \mb E_\theta \left[\widehat c_T(\widehat x_T) \right] \right\}_{\theta\in\Theta}  \\
	\text{subject to} & \displaystyle \limsup_{T\to\infty} \frac{1}{T} \log \mb P_\theta [ c(\widehat x_T, \theta) > \widehat c_T(\widehat x_T)] \leq -r \quad \forall \theta\in\Theta.
	\end{array}
\end{equation}
\end{subequations}
As above, $\lim_{T\to\infty} \mb E_\theta [\widehat c_T(\widehat x_T)]$ is well-defined and coincides with the limit function~$v_\infty(\theta)$ of Definition~\ref{def:dd_prediction}\ref{def:dd_pres:iii:convergence} for every~$\theta\in\Theta$, and the constraint requires that $\mb P_\theta [ c(\widehat x_T, \theta) > \widehat c_T(x)]\leq e^{-rT+o(T)}$ for every~$\theta\in\Theta$.

To gain some intuition for the rate constraint in~\eqref{eq:optimal-predictor}, recall from Definition~\ref{def:dd_prediction}\ref{def:dd_prediction:iii:convergence} that $\widehat c_T(x)$ converges in probability to~$c_\infty(x,\theta)$. As convergence in probability implies convergence in distribution, we thus have
%since $c_\infty(x,\theta)$ is a constant \cite[Theorem~2.7]{vaart_1998}. Therefore, the equality above follows directly from the portmanteau Theorem \cite[Theorem~2.1]{billingsley2009convergence}.
% \begin{align*}
% \limsup_{T\to\infty} \mb P_\theta [ c(x, \theta) \geq \widehat c_T(x)]
% = \Indic{c(x, \theta) \geq  c_\infty(x,\theta)} 
% = 0 \quad \forall x\in X, \ \theta\in\Theta.
% \end{align*}
\(
\lim_{T\to\infty} \mb P_\theta [ c(x, \theta) > \widehat c_T(x)]
= \Indic{c(x, \theta) >  c_\infty(x,\theta)}
\)
for all $x\in X, \ \theta\in\Theta\text{ with } c(x, \theta) \neq  c_\infty(x,\theta)$.
The rate constraint in~\eqref{eq:optimal-predictor} requires the out-of-sample disappointment to converge to~$0$ as~$T$ grows. The above reasoning thus implies that the rate constraint is {\em not} satisfiable if there exists a decision~$x\in X$ and a model~$\theta\in\Theta$ with~$c(x, \theta) >  c_\infty(x,\theta)$. In other words, any feasible data-driven predictor~$\widehat c$ must asymptotically exceed (or match) the model-based predictor~$c(x,\theta)$ for all~$x\in X$ and~$\theta\in\Theta$. This conclusion is consistent with the reasoning that led to the meta-optimization problem~\eqref{eq:optimal-predictor}. In the remainder of the paper we will show that an exponentially decaying out-of-sample disappointment necessitates indeed a biased data-driven predictor that {\em strictly} overestimates~$c(x,\theta)$. In addition, the bias increases with the desired decay rate~$r$.

Multi-objective optimization problems such as~\eqref{eq:optimal-predictor} and~\eqref{eq:optimal-prescriptor} typically only admit Pareto optimal solutions, {\em i.e.}, feasible solutions that are not Pareto dominated by any other feasible solution. Perhaps surprisingly, in the remainder of this paper we will show that under some regularity conditions both~\eqref{eq:optimal-predictor} and~\eqref{eq:optimal-prescriptor} admit Pareto dominant solutions, {\em i.e.}, feasible solutions that Pareto dominate all other feasible solutions. Moreover, these solutions admit intuitive closed-form expressions.

%********************************
\subsection{Data compression} \label{sec:compression:scheme}

A defining property of data-driven predictors and prescriptors is that they are adapted to the filtration generated by the data. Thus, they can be seen as sequences of functions that map the increasingly high-dimensional observation history $\xivec\in\mb R^{dT}$ to a cost estimate or a decision, respectively. Processing or even storing such functions might easily become impractical for large~$T$. As a remedy, we will try to compress the observation history~$\xivec$ into a statistic~$\widehat S_T$ of constant dimension~$d$ without sacrificing useful information.

\begin{definition}[Statistic] \label{def:estimator}
  A stochastic process $\widehat S=\{\widehat S_T\}_{T\in\N}$ with a closed state space $\SS\subseteq \Re^d$ is called a statistic if it is adapted to the filtration $\{\mc F_T\}_{T\in\N}$ and if there exists a local homeomorphism~$S_\infty:\Theta\to \SS$ such that, as $T$ grows, $\widehat S_T$ converges in probability under $\mb P_\theta$ to $S_\infty(\theta)$ for every $\theta\in\Theta$. If $S_\infty(\theta)=\theta$ for all~$\theta\in\Theta$, then the statistic~$\widehat S$ is called a consistent model estimator.
\end{definition}

As~$\widehat S$ is adapted to~$\{\mc F_T\}_{T\in\N}$, we know from \cite[Theorem~5.4.2]{ref:Ash-00} that for any $T\in\N$ there exists a Borel-measurable function $h_T:\Xi^T\to\Re^d$ with $\widehat S_T = h_T(\xivec)$.
In the following we will always assume that the state space~$\SS\subseteq \Re^{d}$ is defined as the smallest closed set that satisfies~$\mb P_\theta[\widehat S_T\in \SS]=1$ for all~$\theta\in \Theta$ and~$T\in \N$. It is also useful to define~$\SS_\infty= \{S_\infty(\theta):\theta\in \Theta\}\subseteq \SS$ as the set of all asymptotic realizations of the statistic~$\widehat S$. %By construction, we have~$\lim_{T\to\infty} \mb P_\theta[\widehat S_T\in \SS_\infty]=1$ for all $\theta\in \Theta$. 
As~$\Theta$ is open with respect to the subspace topology on~$\Theta$ and as~$S_\infty$ constitutes a local homeomorphism, $\SS_\infty$ is an open subset of~$\SS$ with respect to the subspace topology on~$\SS$. %The requirement that~$S_\infty$ be one-to-one further implies that the statistic~$\widehat S$ uniquely identifies the data-generating model~$\theta\in\Theta$ when~$T$ tends to infinity.

\begin{example}[Empirical distribution for finite state i.i.d.\ processes] \label{ex:finitestate:iid:part2}
In the context of the finite state i.i.d.\ processes described in Example~\ref{ex:finitestate:iid:part1}, we define the empirical distribution~$\widehat S_T\in\mb R^d$ through
%For the class of finite state i.i.d.~processes described in \Cref{ex:finitestate:iid:part1} the empirical distribution serves as a natural statistic, that is we define $\widehat S_T$ through 
\begin{equation} \label{estimator:iid:finite:state}
	(\widehat S_T)_{i} = \frac1T \sum_{t=1}^{T} \Indic{\xi_t=i} \quad \forall i\in\Xi,~T\in\mb N.
\end{equation}
Thus, the $i^{\rm th}$ component of $\widehat S_T$ records the empirical frequency of observing state~$i$ over the first~$T$ time periods. By construction, $\widehat S=\{\widehat S_T\}_{T\in\N}$ constitutes a consistent model estimator in the sense of Definition~\ref{def:estimator}. Indeed, the strong law of large numbers guarantees that, under~$\mb P_\theta$, the empirical distribution~$\widehat S_T$ converges almost surely (and thus in probability) to~$S_\infty(\theta)=\theta$ for every~$\theta\in\Theta$. Hence, the set~$\SS_\infty$ coincides with the open probability simplex~$
\Theta$. As the support of~$\widehat S_T$ is given by $\Delta_d\cap (\mb Z^d/T)$ for each $T\in \N$, we also have
  \(
    \SS = \cl \left(\cup_{T\in \N} \Delta_d\cap (\mb Z^d/T)\right) =  \cl \left( \Delta_d\cap \mb Q^d\right)  =\Delta_d=\cl \Theta.
  \)
\end{example}

%Even though most statistics to be studied below constitute consistent model estimators, we do not require consistency in our general framework.

We are now ready to introduce families of data-driven predictors and prescriptors that depend on the data only indirectly through a statistic, which may or may not be a consistent model estimator. To our best knowledge, all predictors and prescriptors studied in the existing literature can be represented in this form.

%With the notion of a statistic at hand, given the data-driven predictors and prescriptors introduced in Definition~\ref{def:dd_prediction}, we define its compressed counterparts as functions satisfying certain regularity properties on the previously discussed set $\SS_\infty$ of asymptotically relevant statistic realizations.

\begin{definition}[Compressed data-driven predictors and prescriptors]
\label{def:compressed:dd_prediction}
If~$\SS$ and~$\SS_\infty$ represent the state space and the set of asymptotic realizations of a statistic~$\widehat S$, then~$\tilde c:X\times\SS\to\Re$ is called a compressed data-driven predictor if it is bounded and continuous in~$x$ on~$X\times\SS$ and continuous in~$(x,s)$ on~$X\times\SS_\infty$. In addition, $\tilde x: \SS\to X$ is called a compressed data-driven prescriptor if it is quasi-continuous on~$\SS_\infty$ and there exists a compressed data-driven predictor~$\tilde c$ that induces~$\tilde x$ in the sense that $\tilde x(s) \in \arg\min_{x\in X} \tilde c(x, s)$ for all $s\in\SS$.
\end{definition}

One can show that {\em any} compressed data-driven predictor~$\tilde c$ induces a (not necessarily unique) compressed data-driven prescriptor~$\tilde x$. To see this, note first that the multifunction~$\arg\min_{x\in X} \tilde c(x, s)$ is non-empty-valued because~$X$ is compact and~$\tilde c(x,s)$ is continuous in~$x\in X$ for every fixed $x\in\SS$. Moreover, the restriction of this multifunction to~$\SS_\infty$ admits a quasi-continuous selector. This follows from the reasoning after Definition~3 in~\cite{ref:vanParys:fromdata-17}, which applies here because~$\tilde c$ is continuous on~$X\times \SS_\infty$ and~$X$ is compact.
%Note that by following the justification of \cite[Definition~3]{ref:vanParys:fromdata-17}, the minimum in \eqref{eq:def:quasi:continuous:prescriptor:data:driven:compressed} indeed always admits a quasi-continuous selection on $\SS_\infty$. 
Note also that any compressed data-driven predictor~$\tilde c$ and the underlying statistic~$\widehat S$ induce a data-driven predictor~$\widehat c$ defined through~$\widehat c_T(x)=\tilde c(x,\widehat S_T)$ for all $x\in X$ and~$T\in\N$. One readily verifies that $\widehat c$ satisfies all conditions of Definition~\ref{def:dd_prediction}. Indeed, conditions~\ref{def:dd_prediction:0}--\ref{def:dd_prediction:ii:boundedness} follow directly from the definitions of the statistic~$\widehat S$ and the compressed data-driven predictor~$\tilde c$. To check condition~\ref{def:dd_prediction:iii:convergence}, fix a probability measure~$\mb P_\theta$, and recall that~$\widehat S_T$ converges in probability to~$S_\infty(\theta)$. By the continuous mapping theorem~\cite[Theorem~3.2.4]{durrett_book}, which applies because~$S_\infty(\theta)\in\SS_\infty$ and because~$\tilde c(x,s)$ is continuous in~$s\in\SS_\infty$ for every fixed~$x\in X$, we may then conclude that $\widehat c_T(x)$ converges in probability to~$\tilde c(x,S_\infty(\theta))$. As this reasoning applies to every model~$\theta\in\Theta$, condition~\ref{def:dd_prediction:iii:convergence} holds with $c_\infty(x,\theta)=\tilde c(x,S_\infty(\theta))$. To check
condition~\ref{def:dd_pres:iii:convergence}, fix again a probability measure~$\mb P_\theta$, and introduce a real-valued function $\tilde v(s) = \min_{x\in X} \tilde c(x,s)$, which is continuous in $s\in\SS$ by Berge's maximum theorem \cite[pp.~115--116]{berge1997topological}. Invoking the continuous mapping theorem as above, it then follows that $\tilde v(\widehat S_T)=\min_{x\in X} \widehat c_T(x)$ converges in probability to $\tilde v(S_\infty(\theta))$. As this reasoning applies to every model~$\theta\in\Theta$, condition~\ref{def:dd_pres:iii:convergence} holds with $v_\infty(\theta)=\tilde v(S_\infty(\theta))$, which is a continuous function by construction. Finally, any compressed data-driven prescriptor~$\tilde x$ and the corresponding statistic~$\widehat S$ induce a data-driven prescriptor $\widehat x$ defined trough $\widehat x_T=\tilde x(\widehat S_T)$ for all~$T\in\mb N$. One readily verifies that~$\widehat x$  satisfies all conditions of Definition~\ref{def:dd_prescription}. Indeed, condition~\ref{def:dd_prescription:i} follows directly from the defining properties of a statistic. To check condition~\ref{def:dd_prescription:ii}, recall that any compressed data-driven prescriptor $\tilde x$ is induced by some compressed data-driven predictor~$\tilde c$. Next, define an ordinary data-driven predictor~$\widehat c$ through $\widehat c_T(x) = \tilde c (x, \widehat S_T)$ for all~$x\in X$ and~$T\in\mb N$. By our earlier reasoning, $\widehat c$ satisfies indeed all conditions of Definition~\ref{def:dd_prediction}. Then, we have 
\[
    \widehat x_T = \tilde x(\widehat S_T)\in \arg\min_{x\in X} \tilde c(x,\widehat S_T) =\arg\min_{x\in X} \widehat c_T(x)\quad \forall T\in\mb N,
\]
where the two equalities follow from the definitions of~$\widehat x_T$ and~$\widehat c_T$, respectively, while the membership relation holds by assumption. Hence, $\widehat x$ is induced by the data-driven predictor~$\widehat c$, and thus condition~\ref{def:dd_prescription:ii} holds.

In analogy to our conventions of Section~\ref{ssec:dd_predictions_prescriptions}, from now on we denote the set of all compressed data-driven predictors by~$\tilde{\mc C}$ and the set of all compressed data-driven predictor-prescriptor-pairs by~$\tilde {\mc X}$.

\begin{example}[Empirical predictor for finite state \ac{iid} processes revisited]
\label{eq:compressed:naive-predictor}
If $\widehat S=\{\widehat S_T\}_{T\in\mb N}$ is any statistic whose state space $\SS$ is a subset of~$\cl(\Theta)$, then the model-based predictor~$c$ of Definition~\ref{def:parametric-predictor} constitutes a trivial compressed data-driven predictor with respect to~$\widehat S$. Indeed, recall that $c$ admits a continuous extension to~$X\times\cl(\Theta)$ thanks to Assumption~\ref{ass:continuity}. In the context of the finite state \ac{iid} processes described in Example~\ref{ex:finitestate:iid:part1}, it is natural to set~$\widehat S_T$ to the empirical distribution over the first~$T$ observations as in Example~\ref{ex:finitestate:iid:part2}. In this case, we have~$\SS=\cl(\Theta)$, which ensures that $\widehat c_T(x)=c(x,\widehat S_T)$ is well-defined for every~$x\in X$ and~$T\in\mb N$. In the special case when $c(x,
\theta)=\mb E_\theta[\ell(x,\xi)]$, a direct calculation shows that~$\widehat c_T(x)=\frac{1}{T} \sum_{t=1}^T \ell(x,\xi_t)$. Thus, the data-driven predictor~$\widehat c=\{
\widehat c_T\}_{T\in\mb N}$ induced by~$c$ and~$\widehat S$ coincides with the empirical predictor of Example~\ref{eq:naive-predictor}.
\end{example}

% An estimator-driven predictor $\tilde c$ is called \textit{strongly conservative} if $\tilde c(x,\theta)> c(x,\theta)$ for all $x\in X$ and $\theta\in\Theta$.

We now consider a restriction of the meta-optimization problem~\eqref{eq:optimal-predictor} that optimizes only over {\em compressed} data-driven predictors~$\tilde c\in\tilde{\mc C}$. As in~\eqref{eq:optimal-predictor}, we minimize the asymptotic expected in-sample performance of {\em every}~$\theta\in\Theta$ subject to an upper bound on the asymptotic exponential decay rate of the out-of-sample disappointment for {\em every}~$\theta\in\Theta$. Identifying each compressed data-driven predictor~$\tilde c$ with an ordinary data-driven predictor~$\widehat c$ defined through~$\widehat c_T(x)=\tilde c(x,\widehat S_T)$, $T\in\mb N$, and observing that $\lim_{T\to\infty} \mb E_\theta [\tilde c(x,\widehat S_T)] = \tilde c(x,S_\infty(\theta))$ for all~$x\in X$ and~$\theta\in\Theta$ thanks to the continuous mapping theorem~\cite[Theorem~3.2.4]{durrett_book} and Lemma~\ref{lem:aux:prob}, the restricted meta-optimization problem can be formulated as follows.
\begin{subequations} \label{eq:optimal:compressed}
\begin{equation}
\label{eq:optimal-predictor:compressed}
	\begin{array}{ll}
	\displaystyle\mathop{\text{minimize}}_{\tilde c\in \tilde{\mathcal{C}}}&  \{ \tilde c(x,S_\infty(\theta)) \}_{x\in X,\, \theta\in\Theta} \\
	\text{subject to} & \displaystyle \limsup_{T\to\infty} \frac{1}{T} \log \mb P_\theta [ c(x, \theta) > \tilde c(x, \widehat S_T)]\leq -r \quad \forall x\in X,\;\theta\in\Theta
	\end{array}
      \end{equation}
Likewise, identifying each compressed data-driven predictor-prescriptor pair~$(\tilde c,\tilde x)$ with an ordinary data-driven predictor-prescriptor pair~$(\widehat c,\widehat x)$ defined through~$\widehat c_T(x)=\tilde c(x,\widehat S_T)$ and~$\widehat x_T=\tilde x(\widehat S_T)$, $T\in\mb N$, and observing that $\lim_{T\to\infty} \mb E_\theta [\tilde c(\tilde x(\widehat S_T),\widehat S_T)] = \tilde c(\tilde x(S_\infty(\theta)),S_\infty(\theta))$ for all~$\theta\in\Theta$ thanks to the continuous mapping theorem and Lemma~\ref{lem:aux:prob}, we obtain the following restriction of the meta-optimization problem~\eqref{eq:optimal-prescriptor}.
\begin{equation}
\label{eq:optimal-prescriptor:compressed}
	\begin{array}{ll}
	\displaystyle\mathop{\text{minimize}}_{(\tilde c,\tilde x)\in \tilde{\mathcal{X}}}{}_{} & \{ \tilde c(\tilde x(S_\infty(\theta)),S_\infty(\theta)) \}_{\theta\in\Theta} \\
	\text{subject to} & \displaystyle \limsup_{T\to\infty} \frac{1}{T} \log \mb P_\theta [ c(\tilde x(\widehat S_T), \theta) > \tilde c(\tilde x(\widehat S_T), \widehat S_T)] \leq -r \quad \forall \theta\in\Theta
	\end{array}
\end{equation}
\end{subequations}
Focusing on compressed data-driven predictors and prescriptors %, which depend on the data only indireclty through a statistic of constant dimension, 
seems natural and is indeed the {\em de facto} standard. On the one hand, one would expect that the corresponding restricted meta-optimization problems~\eqref{eq:optimal:compressed} are easier to solve than the original meta-optimization problems~\eqref{eq:optimal}. On the other hand, it is unclear how much performance is sacrificed by this restriction. In the following we will first show that the restricted meta-optimization problems~\eqref{eq:optimal:compressed} admit Pareto dominant solutions whenever the underlying statistic~$\widehat S$ satisfies a large deviation principle. Later we will show that the compressed and original meta-optimization problems are equivalent whenever~$\widehat S$ represents a sufficient statistic.

%%%%%%%%%%%%%%%%%
%% SEC. Optimal dd predictors and prescriptors
%%%%%%%%%%%%%%%%%
\section{Pareto dominant predictors and prescriptors} \label{sec:optimal:dd:prescriptors}

\subsection{Large deviation principles}
\label{ssec:LD}

In order to construct Pareto dominant solutions for the restricted meta-optimization problems~\eqref{eq:optimal:compressed}, if they exist, we first review and extend some fundamental definitions and concepts from large deviations theory.
Large deviations theory provides bounds on the exponential rate at which the probabilities of atypical realizations of a given statistic~$\widehat S$ decay as the length~$T$ of the observation history grows. These bounds are expressed in terms of a rate function, which depends on a realization of~$\widehat S$ and the data-generating model~$\theta$.
%The two key definitions and concepts introduced in this section are that of a rate function and of a large deviations principle (\ac{ldp}).

\begin{definition}[{Rate function \cite[Section~2.1]{dembo2009large}}]
\label{def:rate_function:original}
A function $I:\SS\times \cl\Theta\rightarrow [0,\infty]$ is called a rate function if~$I(s,\theta)$ is lower semi-continuous in $s$ throughout~$\SS\times\cl \Theta$.
\end{definition}
\begin{definition}[Large deviation principle]
\label{def:wldt}
The statistic $\widehat S=\{\widehat S_T\}_{T\in \N}$ with state space $\SS$ satisfies a \ac{ldp} with rate function $I$, if for all $\theta\in\Theta$ and Borel sets $\mc D \subseteq \SS$ we have
\begin{subequations}
\label{eq:ldp_exponential_rates:old}
\begin{align}
	\label{eq:ldp_exponential_rates_lb:old}
	-\inf_{s \in \interior{\mc D}} \, I(s, \theta)~& \leq \liminf_{T\to \infty}~\frac1T \log \mb P_{\theta}[ \widehat S_T \in \mc D ] \\
	& \leq \limsup_{T\to \infty}~\frac1T \log \mb P_\theta[ \widehat S_T \in \mc D ] \leq -\inf_{s \in \cl {\mc D}} \,  I(s, \theta).
	\label{eq:ldp_exponential_rates_ub:old}
\end{align}
\end{subequations}
\end{definition}

As the subspace topology on~$\SS$ induced by the Euclidean topology on~$\mb R^d$ is Hausdorff, 
%In order to construct the interior and the closure of~$\mc D$, we need to specify a topology on~$\SS$. If $\SS$ is equipped with the coarse topology $\{\emptyset, \SS\}$, then any statistic~$\widehat S$ with state space~$\SS$ trivially satisfies an LDP with respect to any rate function that satisfies $\inf_{s \in \SS} I(s, \theta) = 0$. The situation changes completely if~$\SS$ is equipped with a regular Hausdorff topology, {\em e.g.}, a metric topology. In this case 
we know from~\cite[Lemma~4.1.4]{dembo2009large} that if~$\widehat S$ satisfies an LDP, then the inequalities~\eqref{eq:ldp_exponential_rates:old} uniquely determine the rate function on~$\SS\times\Theta$. %Throughout this paper we will equip~$\SS$ with the subspace topology induced by the Euclidean topology on~$\mb R^d$. As the subspace topology on the closed set~$\SS$ is metrizable, the above reasoning implies that all rate functions discussed in this paper are unique on~$\SS\times\Theta$. Nevertheless,
However, Definition~\ref{def:rate_function:original} requires the rate function to be defined on~$\SS\times\cl\Theta$. Even though its values on the boundary of~$\Theta$ are immaterial, extending the rate function  to~$\SS\times\cl\Theta$ will  simplify some of the derivations in Section~\ref{sec:optimal:dd:prescriptors}, provided the extension preserves the regularity conditions of Definition~\ref{def:rate_function} below.

Before defining regular rate functions, we discuss a few immediate consequences of the inequalities~\eqref{eq:ldp_exponential_rates:old}. First, as $\widehat S_T$ converges in
probability to $S_\infty(\theta)$ under~$\mb P_\theta$, the LDP bound \eqref{eq:ldp_exponential_rates_ub:old} implies that
\[
    0 = \lim_{T\to\infty} \frac{1}{T}\log \mb P_{\theta}\left[\| \widehat S_T-S_\infty(\theta)\|_2 \leq \frac{1}{k} \right] \le -\inf_{s\in\SS} \left\{ I(s,\theta) : \| s-S_\infty(\theta)\|_2 \leq \frac{1}{k} \right\}\quad \forall k\in\mb N.
\]
Thus, there is a sequence $\{s_k\}_{k\in\mb N}$ in~$\SS$ that converges to~$S_\infty(\theta)$ and satisfies $\lim\inf_{k\in\mb N} I(s_k,\theta)\leq 0$, which implies via the non-negativity and lower semi-continuity of the rate function that~$I(S_\infty(\theta),\theta)=0$. This in turn implies that the minima of the optimization problems in~\eqref{eq:ldp_exponential_rates_lb:old} and~\eqref{eq:ldp_exponential_rates_ub:old} evaluate to~$0$ whenever~$S_\infty(\theta)$ falls within the interior of~$\mc D$. In this case the LDP inequalities reduce to the trivial statement that $\mb P_{\theta}[ \widehat S_T \in \mc D ]$ converges to~$1$ as~$T$ grows. In general, the LDP inequalities~\eqref{eq:ldp_exponential_rates:old} imply that the probability~$\mb P_{\theta}[ \widehat S_T \in \mc D ]$ is bounded below by $e^{- \overline r\,T +o(T)}$, where $\overline r=\inf_{s \in \interior{\mc D}} I(s, \theta))$ represents the $I$-distance between~$\theta$ and the interior of~$\mc D$, and bounded above by $e^{- \underline r\,T +o(T)}$, where $\underline r=\inf_{s \in \cl{\mc D}} I(s, \theta)$ represents the $I$-distance between~$\theta$ and the closure of~$\mc D$. 

In the following we will study the class of compressed data-driven predictors and prescriptors corresponding to a statistic~$\widehat S$ that satisfies an LDP. We will see that the {\em optimal} predictors and prescriptors within this class (that is, the Pareto dominant solutions of the meta-optimization problems~\eqref{eq:optimal-predictor:compressed} and~\eqref{eq:optimal-prescriptor:compressed}) can be constructed in closed form if this LDP's rate function is regular in the sense of the following definition.

\begin{definition}[Regular rate function]
\label{def:rate_function}
We call a rate function $I$ regular if the following conditions hold. 
\begin{enumerate}[(i)]
    \setlength\itemsep{0em}
\item \label{def:regular:rate:i:radial:monotonicity}{\bf Radial monotonicity in $\theta$.} $\cl \{\theta\in \Theta:I(s, \theta)< r\}=\{\theta\in \cl\Theta:I(s, \theta)\leq r\}$ for all~$s \in \SS_\infty$, $r>0$.
\item \label{item:regular:rate:continuity} {\bf Continuity.} $I(s,\theta)$ is continuous on $\SS \times\Theta$. 
\item \label{item:regular:rate:compactness} {\bf Level-compactness.} $\{(s,\theta)\in  \SS \times \cl\Theta :I(s,\theta)\leq r\}$ is compact for every $r\geq 0$. %{\color{red} compact relative to the subspace topology on $\SS \times \cl\Theta$}
\end{enumerate}
\end{definition}
%Note that the definition of the rate function, \Cref{def:rate_function:original}, is consistent with the large deviation literature \cite{dembo2009large}.  

Definition~\ref{def:rate_function} strengthens the more common notion of a `good' rate function. Recall that a rate function~$I$ is called good if %for any fixed $\theta\in\cl\Theta$ the mapping $s\mapsto I(s,\theta)$ is lower semi-continuous and 
$\{s\in\SS : I(s,\theta)\leq r\}$ is compact for every $r\ge 0$ and~$\theta\in\Theta$ \cite[Section~1.2]{dembo2009large}. As $\SS_\infty$ is a subset of~$\SS$ and as~$\SS$ is closed, condition~\ref{item:regular:rate:compactness} of Definition~\ref{def:rate_function} implies indeed that every regular rate function is good. Note also that the radial monotonicity condition~\ref{def:regular:rate:i:radial:monotonicity} may be difficult to check in practice. A more easily checkable sufficient condition for radial monotonicity is that for any~$s \in \SS_\infty$, $\theta\in\cl\Theta$ and~$\theta_s\in\Theta$ with~$S_\infty(\theta_s)=s$ (note that~$\theta_s$ exists because~$s\in\SS_\infty$) we have
\begin{equation}\label{eq:suff:radial:mon}
	I(s,(1-\lambda) \theta_s+\lambda\theta)\leq I(s,\theta)\quad \forall \lambda\in[0,1),
\end{equation}
where the inequality is strict if~$I(s,\theta)>0$. The inequality~\eqref{eq:suff:radial:mon} actually inspired the name `radial monotonicity' for condition~\ref{def:regular:rate:i:radial:monotonicity}. We will now prove that~\eqref{eq:suff:radial:mon} together with condition~\ref{item:regular:rate:compactness} implies condition~\ref{def:regular:rate:i:radial:monotonicity}. To this end, fix any~$s\in \SS_\infty$, and set $A(s)=\{\theta\in \Theta:I(s, \theta)< r\}$ and $B(s)=\{\theta\in \cl\Theta:I(s, \theta)\leq r\}$. By construction, we have~$A(s)\subseteq B(s)$. As~$B(s)$ is compact thanks to condition~\ref{item:regular:rate:compactness}, this even implies that~$\cl A(s)\subseteq B(s)$. It remains to be shown that~\eqref{eq:suff:radial:mon} implies the converse inclusion~$ B(s)\subseteq \cl A(s)$. To this end, fix any~$\theta\in B(s)$, and choose any~$\theta_s\in\Theta$ with $s=S_\infty(\theta_s)$, which exists because~$s\in\SS_\infty$. Next, define~$\theta(\lambda)=(1-\lambda) \theta_s+\lambda\theta$ for all~$\lambda\in[0,1)$. As~$\theta_s\in\Theta$ and~$\theta\in\cl\Theta$, and as~$\Theta$ is open and convex, the line segment principle~\cite[Proposition~1.3.1]{bertsekas2009convex} implies that~$\theta(\lambda)\in \Theta$ for all~$\lambda\in [0,1)$. By~\eqref{eq:suff:radial:mon}, we also have~$I(s,\theta(\lambda))\leq I(s,
\theta)\leq r$ for all~$\lambda\in [0,1)$, where at least one of the two inequalities is strict. This reasoning shows that~$\theta(\lambda)\in A(s)$ for all~$\lambda\in[0,1)$. As~$\theta(\lambda)$ approaches~$\theta$ arbitrarily closely when $\lambda$ increases towards~$1$, we may finally conclude that~$\theta\in \cl A(s)$. We have therefore shown that condition~\ref{def:regular:rate:i:radial:monotonicity} holds.

\begin{example}[An LDP for finite state \ac{iid} processes] \label{ex:finitestate:iid:LDP}
Consider the class of finite state \ac{iid} processes introduced in Example~\ref{ex:finitestate:iid:part1}, and let~$\widehat S_T$ be the empirical distribution defined in Example~\ref{ex:finitestate:iid:part2}. The classical Sanov theorem \cite[Theorem~2.1.10]{dembo2009large} asserts that~$\widehat S$ satisfies an LPD with good rate function~$I(s,\theta) =  \D{s}{\theta}$, where $\D{s}{\theta}=\sum_{i=1}^d s_i \log (s_i/\theta_i )$ denotes the relative entropy of~$s$ with respect to~$\theta$, and where we use the standard conventions that~$0\log(0/p)=0$ for any~$p\ge 0$ and~$p\log(p/0)=\infty$ for any~$p>0$. The relative entropy is also referred to as Kullback-Leibler divergence. By the information inequality \cite[Theorem~2.6.3]{cover2006elements}, it is non-negative and vanishes if and only if~$\theta=s$. Moreover, the relative entropy is a regular rate function in the sense of Definition~\ref{def:rate_function}. To see this, note that~$\D{s}{\theta}$ is continuous on~$\SS\times \Theta$ and level-compact \cite[pp.~13--18]{dembo2009large}. In addition, $\D{s}{\theta}$ is jointly convex in~$s$ and~$\theta$ thanks to~\cite[Theorem~2.7.2]{cover2006elements}. Recalling from Example~\ref{ex:finitestate:iid:part2} that~$S_\infty$ is the identity function, for any~$s\in\SS_\infty=\Theta$ and~$\theta\in\cl\Theta$ we thus have
\(
	\D{s}{(1-\lambda) s+\lambda\theta}\leq (1-\lambda)\D{s}{s}+ \lambda \D{s}{\theta}\leq \D{s}{\theta}~ \forall \lambda\in [0,1),
\)
where the second inequality holds because~$\D{s}{s}=0$ and~$\D{s}{\theta}\ge 0$. Note that this inequality is strict for~$\D{s}{\theta}>0$, and thus $\D{s}{\theta}$ is radially monotonic in $\theta$ due to~\eqref{eq:suff:radial:mon}. Hence, the relative entropy is indeed a regular rate function. 
%A visualization of Sanov's theorem is provided in Figure~\ref{fig:ldp}.
In Section~\ref{sec:models} we will present a broad spectrum of additional data-generating stochastic processes for which there exists a statistic that satisfies an \ac{ldp} with a regular rate function.
\end{example}

We are now ready to demonstrate that if the statistic~$\widehat S$ satisfies an~\ac{ldp} with a regular rate function, then the restricted meta-optimization problems~\eqref{eq:optimal:compressed} admit Pareto dominant solutions. In Section~\ref{ssec:dd_prediction}, we first construct a compressed data-driven predictor that is strongly optimal in~\eqref{eq:optimal-predictor:compressed}. In Section~\ref{ssec:dd_prescription} we then construct a compressed data-driven predictor-prescriptor pair that is strongly optimal in~\eqref{eq:optimal-prescriptor:compressed}.

%% SSEC. DRO predictor
%%%%%%%%%%%%%%%%%
\subsection{Distributionally robust predictors} \label{ssec:dd_prediction}
In order to solve the meta-optimization problems~\eqref{eq:optimal:compressed} over compressed data-driven predictors and prescriptors, we assume that the statistic~$\widehat S$ satisfies an LDP and that the underlying rate function is regular. %Definition~\ref{def:rate_function}.

\begin{assumption}[LDP]
\label{ass:ldp}
The statistic~$\widehat S$ satisfies an LDP with a regular rate function.
\end{assumption}

%\TS{probably should introduce the concept of asymptotic sufficiency and Le Cam sufficiency}

We will impose Assumption~\ref{ass:ldp} throughout the rest of this section. We can now construct a compressed data-driven predictor, which will later be shown to represent a Pareto dominant solution for~\eqref{eq:optimal-predictor:compressed}.

\begin{definition}[Distributionally robust predictor]
  \label{def:dro-predictor}
  The function $\tilde c^\star:X\times\SS\to \Re$ defined through
  %Let $I$ be the regular rate function from Assumption~\ref{ass:ldp}. For any fixed $r\geq 0$, we define the compressed predictor $\tilde c^\star:X\times\SS\to \Re_+$ as
  \begin{equation}
    \label{eq:gen:dro}
    \tilde c^\star(x, s) = \left\{ \begin{array}{ll} \textstyle \max_{\theta \in \cl \Theta}\, \left\{ c(x, \theta) : I(s, \theta)\leq r \right\} & \text{if } \exists\, \theta\in\cl\Theta \text{ with } I(s,\theta)\leq r, \\
    \textstyle \sup_{\theta \in \cl \Theta}\, \phantom{\{} c(x, \theta) & \text{if } \nexists\, \theta\in\cl\Theta \text{ with } I(s,\theta)\leq r,
    \end{array}\right.
  \end{equation}
  is the distributionally robust predictor induced by the rate function~$I$ and the risk-aversion parameter~$r$.
\end{definition}

Note that the maximum of the first optimization problem in~\eqref{eq:gen:dro} is indeed attained because the feasible set is compact due to the level-compactness of the regular rate function~$I(s,\theta)$ and because the objective function~$c(x,\theta)$ is continuous in~$\theta$ on~$X\times\cl\Theta$ thanks to the discussion after Assumption~\ref{ass:continuity}. In addition, the supremum of the second optimization problem in~\eqref{eq:gen:dro} is finite because~$c(x,\theta)$ is bounded on~$X\times\cl\Theta$. The following proposition confirms that~$\tilde c^\star$ is a compressed data-driven predictor in the sense of Definition~\ref{def:compressed:dd_prediction}.

\begin{proposition}[Continuity of $\tilde c^\star$]
\label{prop:continuity_cr}
If the rate function $I$ is regular and $r>0$, then the distributionally robust predictor~$\tilde c^\star(x,s)$ is bounded and continuous in~$x$ on~$X\times\SS$ and continuous in~$(x,s)$ on~$X\times\SS_\infty$.
\end{proposition}

Intuitively, the compressed data-driven predictor~$\tilde c^\star(x,s)$ evaluates the worst-case objective function of the stochastic optimization problem~\eqref{eq:examples:c} over all probability measures~$\mb P_\theta$ corresponding to models~$\theta\in\cl\Theta$ that reside in an $I$-ball of radius~$r$ around~$s$. Thus, $\tilde c^\star(x,s)$ admits a distributionally robust interpretation, which justifies our terminology.

\begin{example}[Distributionally robust predictors for finite state i.i.d.\ processes] \label{ex:finitestate:iid:part5}
Consider the class of finite state \ac{iid} processes introduced in Example~\ref{ex:finitestate:iid:part1}, and let~$\widehat S_T$ be the empirical distribution defined in Example~\ref{ex:finitestate:iid:part2}. From Example~\ref{ex:finitestate:iid:LDP} we know that $\widehat S$ satisfies an LDP and that the underlying regular rate function coincides with the relative entropy. Thus, the distributionally robust predictor~\eqref{def:dro-predictor} simplifies to
\(
    \tilde c^\star(x, s) = \max_{\theta \in \Delta_d}\, \left\{ c(x, \theta) : \D{s}{\theta}\leq r \right\}.
\)
This problem is feasible for every possible estimator realization because~$\SS=\Delta_d$ (see Example~\ref{ex:finitestate:iid:part2}), and if $c(x,\theta)=\mb E_\theta[\ell(x,\xi)]$, then it is equivalent to the one-dimensional convex minimization problem
\(
\tilde c^\star(x,s) = \min_{\alpha\geq \bar{\ell}(x)} \alpha - e^{-r} \prod_{i=1}^d(\alpha-\ell(x,i))^{s_i}
\)
with $\bar\ell(x)=\max_{i\in\Xi}\ell(x,i)$, which can be solved efficiently via line search methods (see \cite[Proposition~2]{ref:vanParys:fromdata-17}). %and the minimizer $\alpha^\star$ satisfies $\bar\gamma(x)\leq \alpha^\star \leq \tfrac{(\bar \gamma(x) - e^{-r}c(x,s))}{(1-e^{-r})}$.
\end{example}

The following theorem establishes that the distributionally robust predictor~\eqref{eq:gen:dro} strikes indeed an optimal balance between expected in-sample performance and out-of-sample disappointment. 

%strongly optimal  We now analyze the performance of the distributionally robust compressed predictor $\tilde c^\star$ using arguments from large deviations theory. The parameter $r$ encoding the predictor $\tilde c^\star$ captures the fundamental trade-off between out-of-sample disappointment and conservatism, which is inherent to any approach to data-driven prediction. Indeed, as~$r$ increases, the predictor $\tilde c^\star$ becomes more reliable because its out-of-sample disappointment decreases. However, increasing~$r$ also results in more conservative predictions. In the following we will demonstrate that $\tilde c^\star$ strikes indeed an optimal balance between reliability and conservatism. 

% \begin{proposition}[Feasibility of $\tilde c^\star$]
% \label{prop:pred:feasibility}
% If Assumptions~\ref{ass:parametrization}, \ref{ass:continuity} and~\ref{ass:ldp} hold, then $\tilde c^\star$ is feasible in~\eqref{eq:optimal-predictor:compressed}. 
% \end{proposition}

\begin{theorem}[Optimality of $\tilde c^\star$]
\label{thm:optimality}
If Assumptions~\ref{ass:parametrization}, \ref{ass:continuity} and~\ref{ass:ldp} hold and if $r>0$, then $\tilde c^\star$ is a Pareto dominant solution of the meta-optimization problem~\eqref{eq:optimal-predictor:compressed}. 
\end{theorem}

%% SSEC. DRO prescriptor
%%%%%%%%%%%%%%%%%
\subsection{Distributionally robust prescriptors} \label{ssec:dd_prescription}

We will now demonstrate that if the statistic~$\widehat S$ satisfies an LDP with a regular rate function, then the distributionally robust predictor~$\tilde c^\star$ of Definition~\ref{def:dro-predictor} and any compressed data-driven prescriptor~$\tilde x^\star$ induced by~$\tilde c^\star$ represent a Pareto dominant solution for the meta-optimization problem~\eqref{eq:optimal-prescriptor:compressed}.

\begin{definition}[Distributionally robust prescriptor]
  \label{def:dro-prescriptor}
  If $\tilde c^\star$ is a distributionally robust predictor in the sense of Definition~\ref{def:dro-predictor}, then any function $\tilde x^\star:\SS\to X$ that is quasi-continuous on~$\SS_\infty$ and satisfies
  \begin{equation}
    \label{eq:gen:dro_prescriptor}
	\tilde x^\star(s) \in \arg\min_{x\in X} \, \tilde c^\star(x, s) \quad \forall s\in\SS
  \end{equation}
  is a distributionally robust prescriptor. % induced by~$\tilde c^\star$.
\end{definition}

One can show that any distributionally robust predictor~$\tilde c^\star$ induces at least one distributionally robust prescriptor~$\tilde x^\star$. To see this, note first that the multifunction~$\arg\min_{x\in X} \tilde c^\star(x, s)$ is non-empty-valued because~$X$ is compact and~$\tilde c^\star(x,s)$ is continuous in~$x$ on~$X\times\SS$; see Proposition~\ref{prop:continuity_cr}. Moreover, the restriction of this multifunction to~$\SS_\infty$ admits a quasi-continuous selector. This follows from the reasoning after Definition~3 in~\cite{ref:vanParys:fromdata-17}, which applies here because~$\tilde c^\star$ is continuous on~$X\times \SS_\infty$ and~$X$ is compact.
Therefore, $(\tilde c^\star,\tilde x^\star)$ belongs to the family $\mc X$ of all compressed data-driven predictor-prescriptor-pairs. %We start by showing that $(\tilde c^\star,\tilde x^\star)$ is a feasible solution for the multi-objective optimization problem~\eqref{eq:optimal-prescriptor:compressed}. 

% \begin{corollary}[Feasibility of $(\tilde c^\star, \tilde x^\star)$]
% \label{thm:presc:feasibility}
% If Assumptions~\ref{ass:parametrization}, \ref{ass:continuity} and~\ref{ass:ldp} hold, then $(\tilde c^\star,\tilde x^\star)$ is feasible in~\eqref{eq:optimal-prescriptor:compressed}. 
% \end{corollary}
\begin{theorem}[Optimality of $(\tilde c^\star, \tilde x^\star)$]
\label{thm:optimality_prescriptor}
If Assumptions~\ref{ass:parametrization}, \ref{ass:continuity} and~\ref{ass:ldp} hold and if $r>0$, then $(\tilde c^\star,\tilde x^\star)$  is a Pareto dominant solution of the meta-optimization problem~\eqref{eq:optimal-prescriptor:compressed}. 
\end{theorem}

An interesting question arises as to whether an alternative parametrization for either the statistic $\widehat{S}$ or the model class $\Theta$ would impact the optimal data-driven predictor-prescriptor pair. Notably, an invariance principle can be demonstrated, indicating that the optimal solution remains unchanged under homeomorphic coordinate transformations. A detailed discussion of this invariance is relegated to Appendix~\ref{ssec:representation:invariance}.

%********************************
\section{Separation of estimation and optimization} \label{sec:equivalence}
%**************************************
We are now ready to tackle a fundamental question in data-driven decision-making that is of theoretical as well as practical interest: Under what conditions on the statistic~$\widehat S$ can we restrict the class of {\em all} data-driven predictors and prescriptors to the subclass of all {\em compressed} data-driven predictors and prescriptors induced by~$\widehat S$ without incurring any loss of optimality? In other words, we aim to identify conditions under which any decision-relevant information contained in the raw data~$\xivec$ is also contained in the summary statistic~$\widehat S_T$ for every~$T\in\mb N$, such that the meta-optimization problems~\eqref{eq:optimal-predictor} and~\eqref{eq:optimal-prescriptor} become equivalent to~\eqref{eq:optimal-predictor:compressed} and~\eqref{eq:optimal-prescriptor:compressed}, respectively. 
%Depending on the data-generating process, perhaps some part of the data carries no information about the unknown parameter $\theta$.
%That is, decision-making based on an summary statistic of the data instead of the data itself does not cause any loss either in terms of out-of-sample disappointment or conservatism.
In statistical estimation it is well known that the possibility of lossless compression is intimately related to the existence of a sufficient statistic; see, {\em e.g.}, \cite{ref:Lehmann-98}. In the following, we will argue that such a result also holds in the context of data-driven decision-making. Although this result has intuitive appeal, it seems not to have been established before, and we find it surprisingly difficult to prove.

\begin{definition}[Sufficient statistic] \label{def:sufficient:statistic}
A statistic $\widehat S$ with state space~$\SS$ is called sufficient for $\theta$ if the conditional distribution of $\xivec$ given $\widehat S_T=s$ under $\mb P_\theta$ is independent of $\theta\in \Theta$ for all $s\in\SS$ and $T\in\N$. 
\end{definition}

Intuitively, $\widehat S$ is a sufficient statistic for~$\theta$ if knowing the full observation history~$\xivec$ provides no advantage for estimating~$\theta$ over only knowing~$\widehat S_T$. In other words, compressing~$\xivec$ into~$\widehat S_T$, which is equivalent to a Borel-measurable function of~$\xivec$, does not destroy any information that could be useful for estimating~$\theta$. 
% BART: WE do not actually use the next theorem and in my opinion I fail to see how it is more insightful.
%
% There is an alternative characterization of sufficient statistics, which is easier to verify in practice than the basic definition.
%
% \begin{theorem}[{Factorization criterion \cite[Theorem 2.6.2]{ref:lehmann-2005}}] \label{thm:factorization:criterion}
% An estimator $\widehat\theta$ in the sense of Definition~\ref{def:estimator} is a sufficient statistic if for every $T\in \N$ and $\theta\in\Theta$ there exists a $\sigma$-finite measure $\lambda_T$ on $(\Xi^T,\mc B (\Xi^T))$ and a non-negative Borel-measurable function $g^\theta_T:\Theta'\to\Re_+$ such that $\rho^\theta_T(\xivec) = g^\theta_T(\widehat\theta_T)$, where $\rho^\theta_T$ is the probability density function of $\xivec$ under $\mb P_\theta$ with respect to $\lambda_T$. 
% \end{theorem}
%It is natural to wonder how restrictive the previously discussed sufficiency assumption is. 
The Pitman-Koopman-Darmois theorem~\cite{ref:Koopman-36} implies that if the observed data is \ac{iid} over time, then there exists a sufficient statistic if and only if the data generation process belongs to an exponential family. Even though we do not restrict attention to \ac{iid} processes, this result prompts us to require that the ambiguity set~$\mc P$ represents an exponential family of stochastic processes. To formalize this requirement, we henceforth denote by~$\mb P^T_\theta$ the restriction of the probability measure~$\mb P_\theta$ to the $\sigma$-algebra~$\mc F_T$ generated by~$\xivec$ for all~$T\in\mb N$ and~$\theta\in\Theta$. In addition, for  any~$T\in\mb N$ we define the log-moment generating function~$\Lambda_T:\mb R^d\times\Theta\to(-\infty,+\infty]$ of~$\widehat S_T$ through~$\Lambda_T(\lambda,\theta) = \log\mb E_{\theta}[\exp(\iprod{\lambda}{\widehat S_T})]$ if the expectation is finite and $\Lambda_T(\lambda,\theta) =+\infty$ otherwise. As~$\Lambda_T(0,\theta)=0$ by construction, the function~$\Lambda_T(\lambda,\theta)$ is proper in~$\lambda$. Moreover, $\Lambda_T(\lambda,\theta)$ is convex and lower semi-continuous in~$\lambda$ thanks to~\cite[Theorem~7.1]{barndorff2014information}.

\begin{assumption}[Exponential family of stochastic processes]
  \label{ass:suff}
 The ambiguity set $\mc P$ represents a time-homogeneous exponential family of stochastic processes. This means that there exist a baseline model~$\bar\theta\in\Theta$, a continuous parametrization function $g:\Theta\to\Re^d$ and a sequence of log-partition functions~$A_T:\mb R^d\to(-\infty,+\infty]$ for~$T\in\mb N$ defined through~$A_T(\lambda)=\Lambda_T(\lambda,\bar\theta)$ such that $Tg(\theta)\in \dom(A_T)$ for all~$\theta\in\Theta$ and %such that the Radon-Nikodym derivative of~$\mb P^T_\theta$ with respect to~$\mb P^T_{\theta_0}$ satisfies
\begin{equation}\label{eq:RND:exp:fam}
  %\frac{\d \mb P_\theta}{\d \mb P_0}( \xivec ) 
  \frac{\d \mb P^T_\theta}{\d \mb P^T_{\bar\theta}} = \exp\left(\textstyle \iprod{Tg(\theta)}{\widehat S_T}- A_T(T g(\theta))\right)
  \quad \forall T\in \N,~\theta\in\Theta.
\end{equation}
\end{assumption}

Exponential families that obey Assumption~\ref{ass:suff} are called time-homogeneous because the parametrization function~$g$ is independent of~$T$  \cite[Section 3.1]{kuchler2006exponential}. As the Radon-Nikodym derivative~\eqref{eq:RND:exp:fam} is strictly positive, all probability measures within a given exponential family are mutually equivalent. For every~$T\in\mb N$ and~$\theta\in\Theta$, the log-partition function~$A_T$ ensures that the probability measure~$\mb P^T_\theta$ is normalized, and it inherits properness, convexity and lower semi-continuity from the log-moment generating function~$\Lambda_T$.
% \[
%   A_T(\vartheta) = \log\E{{\theta_0}}{\exp(\iprod{\vartheta}{\widehat S_T})}.
% \]
% Consider a fixed parameter $\theta$ and its associated natural parameter $\vartheta=g(\theta)$. Evidently, the log-moment generation function merely serves to ensure that $\mb P_\theta$ is a probability measure and is given by
%From the previous it must follow that in order for our exponential family to be well defined, we must have that $\{g(\theta):\theta\in\Theta\}\subseteq \dom(A_T(T\cdot))$ for all $T\in \N$.
Even though the log-partition function was defined as the log-moment generating function corresponding to the baseline model~$\bar\theta$, any other log-moment generating function corresponding to an arbitrary model~$\theta\in\Theta$ can be recovered from~$A_T$. This follows from the change of measure formula~\eqref{eq:RND:exp:fam} and the observation that expectations of~$\mc F_T$-measurable functions with respect to~$\mb P_\theta$ depend only on the restriction of~$\mb P_\theta$ to~$\mc F_T$, {\em i.e.},
\begin{equation}
  \label{eq:log-moment-function}
  \begin{aligned}
    \Lambda_T(\lambda, \theta) %& = \log\mb E_{\theta}\left[\exp(\iprod{\lambda}{\widehat S_T})\right]\\
    & = \log\mb E_{\bar\theta}\left[\exp\left(\iprod{\lambda}{\widehat S_T}+\iprod{Tg(\theta)}{\widehat S_T}- A_T(T g(\theta))\right)\right]\\
    & = A_T(\lambda+Tg(\theta))-A_T(Tg(\theta)).
  \end{aligned}
\end{equation}
Assumption~\ref{ass:suff} guarantees via the Fisher-Neyman factorization theorem~\cite[Theorem~6.5]{ref:Lehmann-98} that the statistic~$\widehat S$ is sufficient. This can also be verified directly. Indeed, if it is known that~$\widehat S_T=s$ for some~$s\in\SS$, then the Radon-Nikodym derivative~\eqref{eq:RND:exp:fam} reduces to a deterministic function, and therefore the conditional distribution of~$\xivec$ given~$\widehat S_T=s$ is identical under~$\mb P^T_\theta$ and~$\mb P^T_{\bar\theta}$, that is, it does not depend on~$\theta\in \Theta$. As this argument holds for every~$s\in\SS$ and~$T\in\N$, we may conclude that~$\widehat S$ is indeed a sufficient statistic.

The next assumption will ensure via the celebrated G\"artner-Ellis theorem that~$\widehat S$ also satisfies an LDP.

\begin{assumption}[Log-moment generating functions]\label{ass:LDP2}
The log-moment generating functions corresponding to the statistic~$\widehat S$ display the following properties.
% \begin{enumerate}[(i)]
% \item \label{item:ass:suff:ii} \textbf{G\"artner-Ellis regularity conditions.} 
First, we have~$\Lambda_T(\lambda, \theta) <\infty$ for all $\lambda\in \Re^d$ and $T\in \N$, and the limiting log-moment generating function $\Lambda:\mb R^d\times\Theta\to(-\infty,\infty]$ defined as the limit
  \begin{equation}
    \label{eq:def:Lambda}
    \Lambda(\lambda, \theta) = \lim_{T\to\infty} \frac 1T \Lambda_T(T\lambda, \theta)
  \end{equation}
  exists as an extended real number for all~$\lambda\in\Re^d$ and~$\theta\in\Theta$. In addition,  the origin belongs to the interior of $\dom\Lambda(\cdot,\theta)$ for all~$\theta\in\Theta$. Finally, the gradient $\nabla_\lambda\Lambda(\lambda, \theta)$ exists on the interior of $\dom\Lambda(\cdot,\theta)$, and its norm tends to infinity when~$\lambda$ approaches the boundary of $\dom\Lambda(\cdot,\theta)$ for all $\theta\in\Theta$.
  %\item \label{item:ass:suff:iii} \textbf{Continuous gradient.} The gradient~$\nabla_\lambda \Lambda(0, \theta)$ depends continuously on~$\theta\in\Theta$.
  %\item \label{item:ass:suff:iii} \textbf{Continuity in~$\theta$.} The limiting log-moment generating function~$\Lambda(\lambda, \theta)$ is continuous in~$\theta$ on~$\mb R^d\times\Theta$, and its gradient~$\nabla_\lambda \Lambda(0, \theta)$ evaluated at~$\lambda=0$ is continuous in~$\theta$ on~$\Theta$.
%\end{enumerate}
\end{assumption}

As~$\Lambda_T(0, \theta) =0$ for all~$T\in\mb N$, it is clear that~$\Lambda(0, \theta) =0$, that is, the origin belongs to~$\dom\Lambda(\cdot,\theta)$. Note that Assumption~\ref{ass:LDP2} imposes the stronger condition that the origin belongs to the interior of~$\dom\Lambda(\cdot,\theta)$. Recall next that the log-moment generating functions $\Lambda_T(\lambda, \theta)$ are convex in~$\lambda$ for all~$T\in \N$. By~\cite[Lemma~2.3.9]{dembo2009large}, their asymptotic counterpart~$\Lambda(\lambda, \theta)$ inherits convexity in~$\lambda$. %{\color{blue} Assumption~\ref{ass:LDP2}\ref{item:ass:suff:ii} requires the asymptotic log-moment generating function $\Lambda$ to be essentially smooth in $\lambda$, see \cite[Definition~2.3.5]{dembo2009large}. Note that any smooth convex function on $\Re^d$, i.e., a convex function that is finite and differentiable throughout $\Re^d$, is in particular essentially smooth~\cite[Chapter~26]{rockafellar1970convex}.}
Assumption~\ref{ass:LDP2} further stipulates that~$\dom\Lambda_T(\cdot, \theta)=\mb R^d$, which implies via~\cite[Theorem~7.2]{barndorff2014information} that~$\Lambda_T(\lambda, \theta)$ is analytical in~$\lambda$, throughout all of~$\mb R^d$, for all~$T\in\mb N$. By leveraging the dominated convergence theorem, it is then easy to prove that~$\nabla_\lambda \Lambda_T(0, \theta)=\mb E_{\theta}[\widehat S_T]$
for all~$T\in \N$. The following lemma extends this result to the gradient~$\nabla_\lambda \Lambda(0, \theta)$ of the limiting log-moment generating function, which exists thanks to Assumption~\ref{ass:LDP2}.

\begin{lemma}
  \label{eq:exp:limit}
  If Assumption~\ref{ass:LDP2} holds, then we have~$\nabla_\lambda \Lambda(0, \theta) =\lim_{T\to\infty}\mb E_{\theta}[\widehat S_T]$ for all $\theta\in \Theta$.
\end{lemma}

\begin{remark}
\label{ex:exp:limit-genralized}
Lemma~\ref{eq:exp:limit} admits the following generalization. If Assumptions~\ref{ass:suff} and~\ref{ass:LDP2} hold and~$\eta=g(\theta')-g(\theta)$ for some~$\theta,\theta'\in \Theta$, then one can proceed as in the proof of Lemma~\ref{eq:exp:limit} to show that
\(
    \nabla_\lambda [ \Lambda(\lambda, \theta)]_{\lambda=\eta} = \lim_{T\to\infty}\mb E_{\theta}[\widehat S_T \cdot\exp(\iprod{\eta}{T\widehat S_T}- \Lambda_T(T \eta,\theta)) ].
\)
\end{remark}

% The condition $\lim_{T\to\infty}\mb E_{\theta}[\widehat \theta_T]=\theta$ ensures that $\widehat \theta_T$ is an asymptotically unbiased statistic for $\theta$; this assumption forces the paramatrization function $g$ to be judiciously chosen. 

The following example shows that Assumptions~\ref{ass:suff} and \ref{ass:LDP2} are satisfied if the observable data is governed by an~\ac{iid} process with a finite state space and if~$\widehat S_T$ denotes the empirical distribution. 

\begin{example}[Exponential families of finite state \ac{iid} processes] \label{ex:finitestate:iid:exp:fam}
Consider the class of finite state \ac{iid} processes introduced in Example~\ref{ex:finitestate:iid:part1}, and let~$\widehat S_T$ be the empirical distribution defined in Example~\ref{ex:finitestate:iid:part2}. In this case, the Assumptions~\ref{ass:suff} and \ref{ass:LDP2} are satisfied. To see this, set the baseline model~$\theta_0$ to the uniform probability vector, that is, set~$(\theta_0)_i = 1/d$ for all~$i=1,\hdots d$. Recalling that the probability of observing~$\xivec$ is given by~$\prod_{t=1}^T\theta_{\xi_t}$ under an arbitrary model~$\theta\in\Theta$ and by~$d^{-T}$ under the baseline model~$\bar\theta$, we then find
\(
  \tfrac{\d \mb P^T_\theta}{\d \mb P^T_{\bar\theta}} 
  =  d^T \prod_{t=1}^T \theta_{\xi_t} 
  =  d^T\prod_{j=1}^d  \theta_j^{\sum_{t=1}^T \Indic{\xi_t=j}} =   \textstyle \exp(\iprod{T \log \theta}{\widehat S_T} + T \log d),
\)
where~$\log \theta$ is evaluated component-wise. In addition, the $T^{\rm th}$ log-moment generating function is given by
\begin{align*}
    \Lambda_T(\lambda,\theta)%& =  \log \mb E_\theta \left[\exp(\textstyle \sum_{j=1}^d\lambda_j (\widehat S_T)_j)\right]\\
    & = \log \mb E_\theta \left[\exp(\textstyle \frac 1T \sum_{t=1}^T \sum_{j=1}^d\lambda_j \Indic{\xi_t=j})\right]  = \log \mb E_\theta \left[\textstyle \prod_{t=1}^T \exp( \frac 1T \sum_{j=1}^d\lambda_j \Indic{\xi_t=j})\right]\\
    & = T\log \mb E_\theta \left[\textstyle \exp(\frac 1T \sum_{j=1}^d\lambda_j \Indic{\xi_1=j})\right]  = \textstyle T \log \sum_{i=1}^d \theta_i \exp( \frac 1T\sum_{j=1}^d\lambda_j \Indic{i=j})=T\log\sum_{i=1}^d\theta_i e^{\lambda_i/T},
\end{align*}
where the second equality follows from the serial independence of the observations, and the third inequality holds because all observations have the same marginal distribution as~$\xi_1$. Thus, the family of all finite state \ac{iid} processes corresponding to the models~$\theta\in \Theta$ form a time-homogeneous exponential family with parametrization function~$g(\theta)=\log\theta$ and log-partition function~$A_T(\lambda) = \Lambda_T(\lambda,\theta_0)=T\log\frac1d\sum_{i=1}^d e^{\lambda_i/T}$, which ensures that~$A_T(T g(\theta)) = -T\log d$. This confirms Assumption~\ref{ass:suff} and consequently shows that the empirical distribution is a sufficient statistic for~$\theta$. Next, observe that $\Lambda_T(\lambda,\theta)<\infty$ for all~$T\in\mb N$ and~$\lambda\in\mb R^d$, 
\(
  \Lambda(\lambda,\theta)=\lim_{T\to \infty}\frac 1T \Lambda_T(T\lambda, \theta) = \log \sum_{i=1}^d \theta_i e^{\lambda_i}<\infty
\)
and~$\partial_{\lambda_i} \Lambda(\lambda, \theta)=\theta_i e^{\lambda_i}/(\sum_{j=1}^d\theta_j e^{\lambda_j})$ for all~$i=1,\ldots,d$ and~$\lambda\in\mb R^d$. Therefore, $\Lambda(\lambda,\theta)$ is smooth and convex in~$\lambda$ and continuous in~$\theta$ on~$\mb R^d\times\Theta$. These findings imply that Assumption~\ref{ass:LDP2} holds.
\end{example}

Assumption~\ref{ass:LDP2} guarantees  via the celebrated G\"artner-Ellis theorem that~$\widehat S$ satisfies an \ac{ldp}.

\begin{theorem}[{G\"artner-Ellis theorem \cite[Theorem~2.3.6]{dembo2009large}}] \label{thm:Gaertner-Ellis}
If the limiting log-moment generating function~$\Lambda$ satisfies Assumption~\ref{ass:LDP2}, then the statistic $\widehat S$ satisfies an \ac{ldp} with good rate function 
\begin{equation}\label{eq:GE:ratefct}
    I(s,\theta) = \textstyle \sup_{\lambda\in \Re^d} \; \iprod{\lambda}{s}-\Lambda(\lambda, \theta).
\end{equation}
\end{theorem}

%When defining the function $I$ through \eqref{eq:GE:ratefct} on $\SS\times\Theta$ it can be extended to $\cl \Theta$ as a lower semi-continuous function \cite[Proposition~3.1]{ref:Ferrer-95}. 
Note that the limiting log-moment generating function~$\Lambda(\lambda,\theta)$ and the rate function~$I(s,\theta)$ of Theorem~\ref{thm:Gaertner-Ellis} are only defined on~$\mb R^d\times\Theta$. However, it is usually easy to extend~$I(s,\theta)$ to~$\mb R^d\times\cl\Theta$ so that it becomes a rate function in the sense of Definition~\ref{def:rate_function:original}. In Section~\ref{sec:models} we will provide several examples where~$I(s,\theta)$ can even be extended to a {\em regular} rate function on~$\mb R^d\times\cl\Theta$. Note that~$I(s,\theta)$ displays the following properties for every fixed~$\theta\in\Theta$. First, it coincides with the convex conjugate of the limiting log-moment generating function~$\Lambda(\lambda, \theta)$ with respect to~$\lambda$. Consequently, $I(s,\theta)$ represents a pointwise supremum of affine functions and is thus convex and lower semi-continuous in~$s$. By Assumption~\ref{ass:LDP2}, $\Lambda(\lambda, \theta)$ is essentially smooth in~$\lambda$, that is, the gradient~$\nabla_\lambda\Lambda(\lambda, \theta)$ exists on the interior of $\dom\Lambda(\cdot,\theta)$, and its norm tends to infinity when~$\lambda$ approaches the boundary of~$\dom\Lambda(\cdot,\theta)$. This implies via~\cite[Theorem 26.3]{rockafellar1970convex} that the rate function~$I(s,\theta)$ is strictly convex in~$s$ on the relative interior of~$\dom I(\cdot,\theta)$. Conversely, if~$I(s,\theta)$ is strictly convex in~$s$, then the same theorem guarantees that~$\Lambda(\lambda,\theta)$ is essentially smooth in~$\lambda$. This implication is sometimes useful to verify Assumption~\ref{ass:LDP2}. As~$\Lambda(0, \theta)=0$, we further have $I(s,\theta)\geq 0$ for all~$s\in\SS$.
Finally, as we will show in the following lemma, Assumption~\ref{ass:LDP2} implies via the G\"artner-Ellis theorem that $S_\infty(\theta)=\nabla_\lambda \Lambda(0, \theta)$.
\begin{lemma}[Asymptotic consistency of~$\widehat S$]
  \label{lemm:ge-consistent}
  If Assumption~\ref{ass:LDP2} holds, then, as~$T$ grows, $\widehat S_T$ converges in probability under $\mb P_\theta$ to $\nabla_\lambda \Lambda(0, \theta)$ for every~$\theta\in\Theta$. This implies that~$S_\infty(\theta)=\nabla_\lambda \Lambda(0, \theta)$.
%   \[
%     \lim_{T\to\infty} \mb P_\theta[\tnorm{\widehat S_T-\nabla_\lambda \Lambda(0, \theta)}<\varepsilon]=1 \quad \forall \varepsilon>0,~\theta\in \Theta.
%   \]
\end{lemma}

The following example shows that the G\"artner-Ellis theorem subsumes Sanov's theorem as a special case.

\begin{example}[An LDP for finite state \ac{iid} processes revisited] \label{ex:finitestate:iid:gaertner:ellis}
Consider the class of finite state \ac{iid} processes of Example~\ref{ex:finitestate:iid:part1}, and let~$\widehat S_T$ be the empirical distribution defined in Example~\ref{ex:finitestate:iid:part2}. From Example~\ref{ex:finitestate:iid:exp:fam} we know that the limiting log-moment generating function is given by~$\Lambda(\lambda,\theta) = \log \sum_{i=1}^d \theta_i e^{\lambda_i}$ and that Assumptions~\ref{ass:suff} and~\ref{ass:LDP2} are satisfied. By Theorem~\ref{thm:Gaertner-Ellis}, $\widehat S$ thus satisfies an LPD with good rate function
\(
    I(s,\theta) = \sup_{\eta\in \Re^d} \; \iprod{\eta}{s}-\Lambda(\eta, \theta) = \D{s}{\theta},
\)
where the second equality follows from an elementary but tedious calculation. This reasoning reveals that Sanov's theorem \cite[Theorem~2.1.10]{dembo2009large}, which describes an LDP for the empirical distributions on \ac{iid} data, emerges as a special case of the G\"artner-Ellis theorem. Recall also from Example~\ref{ex:finitestate:iid:LDP} that the relative entropy admits a lower semi-continuous extension to~$\SS\times\cl\Theta=\Delta_d\times\Delta_d$ and constitutes a regular rate function.
\end{example}

We will now demonstrate that if the statistic~$\widehat S$ not only satisfies an \ac{ldp} with a regular rate function but is also sufficient, then even the original meta-optimization problems~\eqref{eq:optimal} admit Pareto dominant solutions that are available in closed form. To this end, denote as usual by~$\tilde c^\star$ the distributionally robust predictor of Definition~\ref{def:dro-predictor}, and introduce a data-driven predictor~$\widehat c^{\,\star}$ defined through~$\widehat c_T^{\,\star}(x)=\tilde c^\star(x,\widehat S_T)$ for all~$T\in\mb N$.

\begin{theorem}[Optimality of~$\widehat c^{\,\star}$] \label{thm:predictor:equivalence}
If the Assumptions~\ref{ass:parametrization}, \ref{ass:continuity}, \ref{ass:suff} and \ref{ass:LDP2} hold, the rate function~\eqref{eq:GE:ratefct} is regular and $r>0$, then~$\widehat c^{\,\star}$ is a Pareto dominant solution of the meta-optimization problem~\eqref{eq:optimal-predictor}.
\end{theorem}

The assumptions of Theorem~\ref{thm:predictor:equivalence} ensure via the G\"artner-Ellis theorem that~$\widehat S$ satisfies an \ac{ldp}, and thus they imply the assumptions of Theorem~\ref{thm:optimality}. From Theorem~\ref{thm:optimality} we further know that~$\tilde c^\star$ represents a Pareto dominant solution to the restricted meta-optimization problem~\eqref{eq:optimal-predictor:compressed} over compressed data-driven predictors. The discussion after Example~\ref{eq:compressed:naive-predictor} finally implies that the objective function value of~$\widehat c^{\,\star}$ in~\eqref{eq:optimal-predictor} coincides with that of~$\tilde c^\star$ in~\eqref{eq:optimal-predictor:compressed} for every fixed decision~$x\in X$ and model~$\theta\in\Theta$, that is, we have
\begin{equation*}
\lim_{T\to\infty} \mb E_\theta[\widehat c_T^{\,\star}(x)] = \tilde c^\star(x,S_\infty(\theta)).
\end{equation*}
As Theorem~\ref{thm:predictor:equivalence} identifies~$\widehat c^{\,\star}$ as a Pareto dominant solution to~\eqref{eq:optimal-predictor}, the above identity thus implies that the original meta-optimization problem~\eqref{eq:optimal-predictor} is indeed equivalent to the restricted meta-optimization problem~\eqref{eq:optimal-predictor:compressed}. In other words, compressing the raw data~$\xivec$ into~$\widehat S_T$ incurs no loss of optimality.

Theorem~\ref{thm:predictor:equivalence} can be interpreted as establishing a separation principle that enables a decoupling of estimation and optimization. Instead of directly solving a data-driven optimization problem of the form~$\min_{x\in X}\widehat c_T(x)$ constructed from the raw data~$\xivec$, which may become increasingly difficult as~$T$ grows, we can first solve an estimation problem that evaluates the statistic~$\widehat S_T$ and subsequently solve an optimization problem~$\min_{x\in X}\tilde c(x,\widehat S_T)$ constructed merely from~$\widehat S_T$. Theorem~\ref{thm:predictor:equivalence} guarantees that if these two data-driven optimization problems are designed optimally, then no optimality is sacrificed by this separation.

Next, we show that  the meta-optimization problem~\eqref{eq:optimal-prescriptor} over data-driven predictor-prescriptor pairs also admits a Pareto dominant solution. To this end, define the distributionally robust predictor~$\tilde c^\star$ and the corresponding data-driven predictor~$\widehat c^\star$ as before, and let~$\tilde x^\star$ be a distributionally robust prescriptor as in Definition~\ref{def:dro-prescriptor}. Then, introduce a data-driven prescriptor~$\widehat  x^{\,\star}$ defined through~$\widehat x_T^{\,\star} =\tilde x^\star(\widehat S_T)$ for all~$T\in\mb N$.

\begin{theorem}[Optimality of~$(\widehat c^{\,\star}, \widehat x^\star)$] \label{thm:prescriptor:equivalence}
If the Assumptions~\ref{ass:parametrization}, \ref{ass:continuity}, \ref{ass:suff} and \ref{ass:LDP2} hold, the rate function~\eqref{eq:GE:ratefct} is regular and $r>0$, then~$(\widehat c^{\,\star}, \widehat x^\star)$ is a Pareto dominant solution of the meta-optimization problem~\eqref{eq:optimal-prescriptor}.
% If Assumptions~\ref{ass:parametrization}, \ref{ass:continuity}, \ref{ass:suff} and \ref{ass:LDP2} hold, the rate function in \eqref{eq:GE:ratefct}, and if $r>0$ then a Pareto dominant solution of the meta-optimization problem~\eqref{eq:optimal-prescriptor} is given by $(\widehat c^{\,\star},\widehat x^{\,\star})$, where $\widehat c_T^{\,\star}(x)=\tilde c^\star(x,\widehat S_T)$, $\widehat x_T^{\,\star}\in \arg\min_{x\in X} \tilde c^\star(x,\widehat S_T)$ and $\tilde c^\star$ is the function defined in \eqref{eq:gen:dro}.
\end{theorem}

The assumptions of Theorem~\ref{thm:prescriptor:equivalence} imply the assumptions of Theorem~\ref{thm:optimality_prescriptor}, which in turn implies that~$(\tilde c^\star, \tilde x^\star)$ represents a Pareto dominant solution to the restricted meta-optimization problem~\eqref{eq:optimal-prescriptor:compressed}. The discussion after Definition~\ref{def:compressed:dd_prediction} further implies that the objective function value of~$(\widehat c^{\,\star},\widehat x^{\,\star})$ in~\eqref{eq:optimal-prescriptor} coincides with that of~$(\tilde c^\star,\tilde x^\star)$ in~\eqref{eq:optimal-prescriptor:compressed} for every fixed model~$\theta\in\Theta$, that is, we have
\begin{equation*}
\lim_{T\to\infty} \mb E_\theta[\widehat c_T^{\,\star}(\widehat x_T^{\,\star})] = \tilde c^\star(\tilde x^\star(S_\infty(\theta)),S_\infty(\theta)).
\end{equation*}
As Theorem~\ref{thm:prescriptor:equivalence} identifies~$(\widehat c^{\,\star}, \widehat x^\star)$ as a Pareto dominant solution to~\eqref{eq:optimal-prescriptor}, the original meta-optimization problem~\eqref{eq:optimal-prescriptor} is thus equivalent to the restricted meta-optimization problem~\eqref{eq:optimal-prescriptor:compressed}. Therefore, Theorem~\ref{thm:prescriptor:equivalence} establishes another separation principle that enables a decoupling of estimation and optimization.

Theorems~\ref{thm:predictor:equivalence} and~\ref{thm:prescriptor:equivalence} are reminiscent of the celebrated Rao-Blackwell theorem~\cite{ref:Rao-45, ref:Blackwell-47}, which asserts that any given estimator~$\widehat \theta_T$ of the unknown parameter~$\theta$ can be improved by conditioning it on a sufficient statistic~$\widehat S_T$. The resulting estimator~$\mb E_{\theta}[\widehat \theta_T|\widehat S_T]$ is non-inferior to~$\widehat \theta_T$ with respect to the mean squared error criterion and depends on the available data only through~$\widehat S_T$. The proof of the Rao-Blackwell theorem critically relies on Jensen's inequality, which is applicable because the mean squared error is convex in~$\widehat \theta_T$. Unfortunately, it is not possible to improve a given data-driven predictor~$\widehat c_T(x)$ by simply conditioning it on~$\widehat S_T$. This approach fails because the out-of-sample disappointment is {\em non}-convex in~$\widehat c_T(x)$. The proofs of Theorems~\ref{thm:predictor:equivalence} and~\ref{thm:prescriptor:equivalence} are therefore substantially more involved than that of the Rao-Blackwell theorem.

\begin{example}[Optimal predictors and prescriptors for finite state \ac{iid} processes] \label{ex:finitestate:iid:optimality}
Consider the class of finite state \ac{iid} processes of Example~\ref{ex:finitestate:iid:part1}, and let~$\widehat S_T$ be the empirical distribution defined in Example~\ref{ex:finitestate:iid:part1}. We know from Example~\ref{ex:finitestate:iid:LDP} that $\widehat S_T$ satisfies an LDP with regular rate function~$\D{s}{\theta}$. By Theorems~\ref{thm:optimality} and~\ref{thm:optimality_prescriptor}, the distributionally robust predictor~$\tilde c^\star$ with a relative entropy ambiguity set and the corresponding prescriptor~$\tilde x^\star$ thus provide Pareto dominant solutions for the restricted meta-optimization problems~\eqref{eq:optimal:compressed}. From Example~\ref{ex:finitestate:iid:exp:fam} we further know that Assumptions~\ref{ass:suff} and \ref{ass:LDP2} hold. By Theorems~\ref{thm:predictor:equivalence} and~\ref{thm:prescriptor:equivalence} the data-driven predictor~$\widehat c^\star$ and the corresponding prescriptor~$\widehat x^\star$ induced by~$\tilde c^\star$ and~$\tilde x^\star$, respectively, thus provide Pareto dominant solutions for the original meta-optimization problems~\eqref{eq:optimal}.
\end{example}

%%%%%%%%%%%%%%%%%
%% SEC. Models
%%%%%%%%%%%%%%%%%
\section{Data-generating processes} \label{sec:models}

We now describe several data-generating processes for which the restricted meta-optimization problems~\eqref{eq:optimal:compressed} or even the original meta-optimization problems~\eqref{eq:optimal} admit Pareto dominant solutions.

\subsection{Finite-state Markov chains} \label{sssec:Markovchain}

Assume that $\{\xi_t\}_{t=1}^T$ represents a time-homogeneous ergodic Markov chain with state space~$\Xi=\{1,\ldots,m\}$ and dummy deterministic initial state $\xi_0=i_0\in\Xi$ satisfying $\lim_{t\to\infty}\mb P_\star[\xi_t=i,\,\xi_{t+1}=j]=(\theta_\star)_{ij}>0$ for all $i,j\in\Xi$. The matrix~$\theta_\star$ encodes the stationary probability mass function of the doublet $(\xi_t,\xi_{t+1})$, and thus
\[
	\sum_{j\in\Xi}(\theta_\star)_{ij} = \lim_{t\to\infty}\sum_{j\in\Xi} \mb P_\star[\xi_t=i,\,\xi_{t+1}=j]  =  \lim_{t\to\infty}\mb P_\star[\xi_t=i] = \lim_{t\to\infty}\sum_{j\in\Xi} \mb P_\star[\xi_{t-1}=j,\,\xi_{t}=i] = \sum_{j\in\Xi}(\theta_\star)_{ji},
\]
{\em i.e.}, the row sums of~$\theta_\star$ coincide with the respective column sums. %, which are strictly positive due to irreducibility. 
These properties of~$\theta_\star$ prompt us to define
\(
	\textstyle\Theta=\{\theta\in\mb R^{m\times m}_{++}: \sum_{i,j\in\Xi} \theta_{ij}=1,~ \sum_{j\in\Xi}\theta_{ij} = \sum_{j\in\Xi} \theta_{ji}~\forall i\in\Xi \}
\)
as the set of all strictly positive doublet probability mass functions with balanced marginals. Note that every~$\theta\in\Theta$ induces a unique row vector~$\pi_\theta\in\mb R^{1\times m}_{++}$ of stationary state probabilities and a unique transition probability matrix~$P_\theta\in\mb R^{m\times m}_{++}$ defined through~$(\pi_\theta)_i=\sum_{j\in\Xi}\theta_{ij}$ and $(P_\theta)_{ij}=\theta_{ij}/(\pi_\theta)_i$, respectively. By construction, $P_\theta$ is a stochastic matrix whose rows represent strictly positive probability vectors, and the stationary distribution~$\pi_\theta$ satisfies~$\pi_\theta P_\theta =\pi_\theta$; see Ross~\cite[Chapter~4]{ross2010introduction} for further details on Markov chains. We conclude that~$\mb P_\star$ belongs to a finitely parametrized ambiguity set of the form~$\mc P=\{\mb P_\theta:\theta\in\Theta\}$, where each model~$\theta\in\Theta$ encodes a probability measure~$\mb P_\theta$ on~$(\Omega, \mc F)$ with
\begin{equation*}
  \textstyle \mb P_\theta[ \xivec = (i_1,\hdots,i_T) ] = \prod_{t=1}^{T} (P_\theta)_{i_{t-1} i_{t+1}} \quad \forall (i_1,\hdots,i_T)\in \Xi^T , ~ T\in\mb N.
\end{equation*}
Note also that~$\Theta$ is embedded in a Euclidean space of finite dimension~$d=m^2$. In summary, we have thus shown that Assumption~\ref{ass:parametrization} holds. Next, we define the empirical doublet distribution~$\widehat S_T\in\mb R^{m\times m}$ through
\begin{equation} \label{MC:estimator}
  \textstyle(\widehat S_T)_{ij} = \frac{1}{T} \sum_{t=1}^{T} \Indic{(\xi_{t-1},\xi_{t})=(i,j)} \quad \forall i,j\in\Xi.
\end{equation}
By construction, $\widehat S=\{\widehat S_T\}_{T\in\N}$ constitutes a statistic with state space
\(
  \textstyle \SS =\cl\left(\cup_{T\in\N} \Delta_{m\times m}\cap(\mb Z^{m\times m}/T) \right)= \cl\left( \Delta_{m\times m}\cap \mb Q^{m\times m}\right) = \Delta_{m\times m}. %\{\theta\in\mb R^{d\times d}_+: \sum_{i,j\in\Xi} \theta_{ij}=1\}
\)
We emphasize that~$\SS$ is a strict superset of the model space~$\Theta$. The ergodic theorem for Markov chains further ensures that the empirical doublet distribution~$\widehat S_T$ converges $\mb P_\theta$-almost surely to the true doublet distribution~$\theta$ as $T$ grows; see \cite[Theorem 4.1]{ross2010introduction}. Consequently, we have~$S_\infty(\theta)=\theta$ for all~$\theta\in\Theta$, which implies that~$\widehat S$ is a consistent model estimator in the sense of Definition~\ref{def:estimator} and that the set~$\SS_\infty$ of all asymptotic realizations of~$\widehat S$ coincides with $\Theta$. In addition, $S_\infty$ is clearly a local homeomorphism.

We now follow the reasoning in~\cite{ref:Billingsley-61} to show that the ambiguity set~$\mc P$ represents a time-homogeneous exponential family. Specifically, we define the baseline model~$\bar\theta\in\Theta$ through~$\bar\theta_{ij} = 1/m^2$ for all~$i,j\in\Xi$. The observations~$\xi_t$, $t\in\N$, are thus serially independent and uniformly distributed under~$\mb P_{\bar\theta}$, and the corresponding transition probability matrix satisfies~$(P_{\bar\theta})_{ij}=1/m$ for all~$i,j\in\Xi$. In addition, the probability of observing~$\xivec$ under~$\mb P_{\bar \theta}$ is given by~$1/m^{T}$, and
\(
	\tfrac{\d \mb P^T_\theta}{\d \mb P^T_{\bar\theta}} 
	= m^T  \prod_{t=1}^{T} (P_\theta)_{\xi_{t-1} \xi_{t}}
	= m^{T}\prod_{i,j\in\Xi} (P_\theta)_{ij}^{\sum_{t=1}^{T} \Indic{(\xi_{t-1},\xi_{t})=(i,j)}}
	= m^{T}\prod_{i,j\in\Xi} (P_\theta)_{ij}^{T(\widehat S_T)_{ij}} = \exp(\iprod{T \log(P_\theta)}{\widehat S_T} + T\log m),
\)
where the logarithm of the matrix~$P_\theta$ is evaluated element-wise. This reveals that~$\mc P$ constitutes an exponential family in the sense of Assumption~\ref{ass:suff} with parametrization function~$g(\theta)=\log(P_\theta)$ and that~$\widehat S$ is a sufficient statistic. % and log-partition function~$A_T(T g(\theta)) = -T\log d$. 
The $T^{\rm th}$ log-moment generating function~$\Lambda_T(\lambda,\theta)$---and thus also the log-partition function~$A_T(\lambda)$---admit no concise closed-form expression. However, the proof of~\cite[Theorem~3.1.2]{dembo2009large} implies that the limiting log-moment generating function~$\Lambda(\lambda, \theta) = \lim_{T\to\infty} \frac1T \Lambda_T(T\lambda, \theta)$ is everywhere finite and differentiable in~$\lambda$ for all~$\theta\in\Theta$. %Consequently, Assumption~\ref{ass:LDP2}\ref{item:ass:suff:ii} is satisfied. 
In addition, we have~$\nabla_\lambda\Lambda(0,\theta) =\lim_{T\to\infty} \mb E_\theta[\widehat S_T] = \mb E_\theta[\lim_{T\to\infty} \widehat S_T]=\theta$, where the three equalities follow from Lemma~\ref{eq:exp:limit}, the dominated convergence theorem and our insight that~$\widehat S$ converges $\mb P_\theta$-almost surely to~$\theta$, respectively. Hence, Assumption~\ref{ass:LDP2} holds, which ensures via the G\"artner-Ellis theorem that~$\widehat S$ satisfies an \ac{ldp}; see also~\cite[Theorem~3.1.13]{dembo2009large}. The corresponding rate function~$I(s,\theta)$ is given by the convex conjugate of the limiting log-moment generating function~$\Lambda(\lambda,\theta)$ with respect to~$\lambda$, which coincides with conditional relative entropy of~$s$ with respect to~$\theta$ \cite[Section~3.1.3]{dembo2009large}.

%As we will prove in Proposition \ref{prop:properties:cond:rel:entropy} the associated rate function is regular as well.

%The \ac{ldp} reported in Theorem~\ref{thm:sldp-mc} below essentially determines the speed of convergence, and the underlying rate function is given by the conditional relative entropy. Also note that under the assumption that the initial state $\sigma$ of the Markov chain is given, the estimator~\eqref{MC:estimator} is known to form a sufficient statistic. 

%Thanks to Theorem \ref{thm:Gaertner-Ellis} (G\"artner-Ellis) we know that this rate function associated with the sufficient statistic is precisely the Fenchel conjugate of $\Lambda(\lambda, \theta)$ with respect to the first parameter. Even more, the authors \cite[Section~3.1.3]{dembo2009large} indicate this Fenchel conjugate admits an explicit representation as the conditional relative entropy.
 
\begin{definition}[Conditional relative entropy] 
\label{def:conditional_relative_entropy}
%If~$\theta\in \Theta$ and~$s\in \SS$ are doublet distributions with corresponding transition probability matrices~$P_{\theta},P_{s}\in\mb R^{d\times d}_+$ and invariant distributions $\pi_{\theta},\pi_{s}\in\mb R^{1\times d}_+$, respectively, then %\footnote{The construction of $P_s$ and $\pi_s$ for $s\in\SS$ is analogous as for $P_\theta$ and $\pi_\theta$ for $\theta\in\Theta$.} 
Using the standard convention that~$0\log(0/p)=0$ for any $p\ge 0$, the conditional relative entropy of~$s\in\SS$ with respect to~$\theta\in\Theta$ is defined~as
\begin{align*}
	\Dc{s}{\theta} =\sum_{i,j\in\Xi} s_{ij} \left( \log\left( \frac{s_{ij}}{\sum_{k\in\Xi} s_{ik}} \right) -  \log\left( \frac{\theta_{ij}}{\sum_{k\in\Xi} \theta_{ik}}\right) \right).
\end{align*}
\end{definition}

If we denote the $i^{\rm th}$ rows of the transition probability matrices~$P_{s}$ and~$P_\theta$ by~$(P_{s})_{i\cdot}$ and~$(P_\theta)_{i\cdot}$, respectively,\footnote{If $(\pi_s)_i=\sum_{j\in\Xi}s_{ij}=0$, then we may define without loss of generality $(P_s)_{ij}=1$ if $j=i$ and $(P_s)_{ij}=0$ otherwise.} and if we denote the relative entropy as usual by~$\D{\cdot}{\cdot}$, then an elementary calculation reveals that
\(
    \Dc{s}{\theta} =\sum_{i\in\Xi} (\pi_{s})_i \, \D{(P_{s})_{i\cdot}}{(P_\theta)_{i\cdot}}.
\)
Thus, $\Dc{s}{\theta}$ can be viewed as the relative entropy distance between the transition probability vectors under~$s$ and~$\theta$ emanating from a random state of the Markov chain, averaged by the invariant state distribution associated with~$s$. This interpretation justifies the name `conditional relative entropy.' Note also that Definition~\ref{def:conditional_relative_entropy} specifies~$\Dc{s}{\theta}$ only on~$\SS\times\Theta$ and that~$\Dc{s}{\theta}$ is continuous on~$\SS\times\Theta$ thanks to our standard conventions for the logarithm. We emphasize that~$\Dc{s}{\theta}$ cannot be continuously extended beyond~$\SS\times\Theta$. However, $\Dc{s}{\theta}$ admits a unique lower semi-continuous extension to~$\SS\times\cl\Theta$, which is obtained by setting
\[
    \Dc{s}{\theta} = \lim_{\delta\downarrow 0} \inf_{(s',\theta')\in\SS\times\Theta} \left\{\Dc{s'}{\theta'} : \|(s',\theta')-(s,\theta)\|\leq \delta \right\}\quad \forall (s,\theta)\in\SS\times (\cl\Theta\backslash \Theta);
\]
see also~\cite[Definition~1.5]{rockafellar1998variational}. In the following, we will always mean this lower semi-continuous extension to~$\SS\times\cl\Theta$ when referring to the conditional relative entropy~$\Dc{s}{\theta}$. The next proposition establishes that the conditional relative entropy represents a regular rate function in the sense of Definition~\ref{def:rate_function}. 

\begin{proposition}[Properties of the conditional relative entropy] \label{prop:properties:cond:rel:entropy}
The conditional relative entropy $\Dc{s}{\theta}$ is a regular rate function in the sense of Definition~\ref{def:rate_function}. In addition, $\Dc{s}{\theta}$ is convex in~$s$.
\end{proposition}

By Theorems~\ref{thm:optimality} and~\ref{thm:optimality_prescriptor}, we may now conclude that the distributionally robust predictor~$\tilde c^\star$ with a conditional relative entropy ambiguity set and the corresponding prescriptor~$\tilde x^\star$ provide Pareto dominant solutions for the restricted meta-optimization problems~\eqref{eq:optimal:compressed}. Moreover, by Theorems~\ref{thm:predictor:equivalence} and~\ref{thm:prescriptor:equivalence} the data-driven predictor~$\widehat c^\star$ and the corresponding prescriptor~$\widehat x^\star$ induced by~$\tilde c^\star$ and~$\tilde x^\star$, respectively, provide Pareto dominant solutions for the original meta-optimization problems~\eqref{eq:optimal}. As~$\Dc{s}{\theta}$ fails to be convex in~$\theta$, computing~$\tilde c^\star(x,s)$ for a fixed~$x\in X$ and~$s\in\SS$ necessitates the solution of a challenging non-convex optimization problem with~$\mathcal O(m^2)$ decision variables \cite{ICML-Li-21}. In Appendix~\ref{subsec:AR} we show that the restricted meta-optimization problems sometimes admit Pareto dominant solutions even if the training data is generated by an autoregressive process with an uncountable state space instead of a finite-state Markov chain.

\subsection{Independent observations with identical parametric distribution functions}
\label{sec:iid-parametric}
As a last example, assume that the observations~$\{\xi_t\}_{t=1}^T$ are valued in~$\Re^m$ and that they are serially independent and share the same distribution function~$F_{\theta_\star}$ under~$\mb P_\star$, that is, we have $\mb P_\star[\xi_t\leq z]=F_{\theta_\star}(z)$ for all~$z\in\Re^m$ and $t\in\mb N$. Here, $F_{\theta}$, $\theta\in\Theta$, is a family of distribution functions with common support~$\Xi\subseteq\Re^m$, where the parameter~$\theta$ ranges over the relative interior~$\Theta$ of a convex subset of~$\Re^d$, and $\theta_\star$ denotes the unknown true parameter. Clearly, the mean value of~$F_\theta$ must be a function of~$\theta$ and can thus be expressed as~$S_\infty(\theta)$. Throughout this section we assume that the function~$S_\infty$ constitutes a homeomorphism from the set~$\Theta$ to its image~$\SS_\infty=\{S_\infty(\theta):\theta\in\Theta\}$. As any homeomorphism is invertible, this assumption means that the parameter~$\theta$ is uniquely determined by the mean value of~$F_\theta$. We may then conclude that~$\mb P_\star$ belongs to an ambiguity set $\{\mb P_\theta:\theta\in\Theta\}$, where each~$\theta\in\Theta$ encodes a probability  measure $\mb P_\theta$ on $(\Omega, \mc F)$ satisfying
\begin{equation*}
	\mb P_\theta[ \xi_t\leq z_t~\forall t=1,\ldots,T] = \textstyle \prod_{t=1}^T F_\theta(z_t) \qquad \forall z\in \Re^{mT} , ~ T\in\mb N.
\end{equation*}
In order to estimate the mean value~$S_\infty(\theta)$ (and thereby implicitly also~$\theta$) we use the sample mean
\begin{equation} \label{cts:sample:mean}
	\widehat S_T= \textstyle\frac{1}{T}\sum_{t=1}^T \xi_t \quad \forall T\in\N.
\end{equation}
By our standard conventions, the state space~$\SS$ of~$\widehat S$ is given by the closure of the convex hull of~$\Xi$. In the following we assume that the distribution function~$F_\theta$  has exponentially bounded tails for every~$\theta\in \Theta$. The strong law of large numbers then implies that~$\widehat S_T$ converges $\mb P_\theta$-almost surely to~$S_\infty(\theta)$. %, which justifies our notation. %implies that the set of asymptotic realizations of the statistic~$\widehat S$ is given by~$\SS_\infty$. 
More specifically, we henceforth focus on several popular families of distribution functions that are susceptible to analytical treatment:
\begin{enumerate}[label={(\alph*)}]
  \setlength\itemsep{-0.25em}
    \item normal distributions on~$\Re^m$ with an unknown mean vector~$\theta\in\Re^m$ and a positive definite covariance matrix~$\Sigma\in\mathbb \Re^{m\times m}$;
    \item exponential distributions on~$\Re_+$ with an unknown rate parameter~$\theta>0$;
    \item Gamma distributions on~$\Re_+$ with an unknown scale parameter~$\theta>0$ and a shape parameter~$k>0$;
    \item Poisson distributions on~$\N\cup\{0\}$ with an unknown rate parameter~$\theta\in\Re_{++}$;
    \item Bernoulli distributions on~$\{0,1\}$ with an unknown success probability $\theta\in(0,1)$;
    \item geometric distributions on~$\mathbb{N}$ with an unknown success probability~$\theta\in(0,1)$;
    \item binomial distributions on~$\N\cup\{0\}$ with an unknown success probability~$\theta\in (0,1)$ and~$N\in\mathbb{N}$ trials.
\end{enumerate}
Clearly, each of these examples satisfies Assumption~\ref{ass:parametrization}. It is also well known that each of these examples gives rise to a time-homogeneous exponential family in the sense of Assumption~\ref{ass:suff} and that the sample mean~\eqref{cts:sample:mean} is a sufficient statistic for~$\theta$. To see that the sample mean also satisfies an \ac{ldp} with a regular rate function, note that for \ac{iid} data the limiting log-moment generating function simplifies to
\begin{equation}
    \label{log-mgf-iid}
    \Lambda(\lambda,\theta) = \textstyle\lim_{T\to\infty} \frac{1}{T}\log\mb E_{\theta}\left[\exp(\langle T\lambda, \widehat S_T \rangle)\right]= \log\left( \int_{\Re^m} e^{\lambda\tpose \xi} \, \d F_\theta(\xi)\right).
\end{equation}
As~$F_\theta$ is assumed to have exponentially bounded tails, $\Lambda(\lambda,\theta)$ is finite on a neighborhood of~$\lambda=0$ for every fixed~$\theta\in\Theta$. Moreover, $\Lambda(\lambda,\theta)$ is available in closed form for all families of distribution functions listed above; see Appendix~\ref{app:table}. In each case one can therefore verify by inspection that the gradient $\nabla_\lambda\Lambda(\lambda, \theta)$ exists on the interior of $\dom\Lambda(\cdot,\theta)$ and that its norm tends to infinity when~$\lambda$ approaches the boundary of $\dom\Lambda(\cdot,\theta)$. Thus, Assumption~\ref{ass:LDP2} holds, which ensures via the G\"artner-Ellis theorem that~$\widehat S$ satisfies an \ac{ldp}. The corresponding rate function~$I(s,\theta)$ coincides with the Cram\'er function~$\Lambda^*(s,\theta)$, that is, the convex conjugate of the limiting log-moment generating function~\eqref{log-mgf-iid} with respect to~$\lambda$. The Cram\'er function is again available in closed form for all examples listed above; see Table~\ref{tab:rate_functions}. In each case one can  verify by inspection that~$\Lambda^*(s,\theta)$ represents in fact a regular rate function. By Theorems~\ref{thm:optimality} and~\ref{thm:optimality_prescriptor}, the distributionally robust predictor~$\tilde c^\star$ constructed from the Cram\'er function and the corresponding prescriptor~$\tilde x^\star$ thus provide Pareto dominant solutions for the restricted meta-optimization problems~\eqref{eq:optimal:compressed}. Moreover, by Theorems~\ref{thm:predictor:equivalence} and~\ref{thm:prescriptor:equivalence} the data-driven predictor~$\widehat c^\star$ and the corresponding prescriptor~$\widehat x^\star$ induced by~$\tilde c^\star$ and~$\tilde x^\star$, respectively, provide Pareto dominant solutions for the original meta-optimization problems~\eqref{eq:optimal}.

\section{Conclusions}
\label{sec:conclusion}
%Consider for the sake of argument a fixed decision $x\in X$.
This paper proposes a rigorous framework for identifying optimal estimators for the objective functions and the optimal solutions of data-driven decision problems. To conclude we provide recommendations for practitioners and discuss potential generalizations of our results.

% To circumvent the possibility of the nonexistence of efficient estimators one can take a minimax perspective \cite{puhalskii1998large} instead and denote an estimator as \textit{minimax efficient} when it is an optimal solution in the optimization problem
% \begin{equation*}
%   \displaystyle\mathop{\text{minimize}}_{\widehat c}{} \max_{\varepsilon>0, \, \theta\in \Theta} ~ e(\varepsilon, \theta, \widehat c(x)).
% \end{equation*}
% However, a minimax efficient estimator is only efficient with respect to a worst case error size $\varepsilon$ and parameter $\theta\in \Theta$.
% Clearly, should a Bahadur efficient estimator exist it should be preferred over any minimax efficient alternative.

  Our paper offers the following three-step guideline for practitioners faced with a data-driven decision problem. First, users should identify a finitely parametrized time series model consistent with the observable data. Second, they should find a statistic for the unknown parameters of the time series model that satisfies an LDP. Third, they should construct efficient data-driven predictors and prescriptors by solving the DRO problems~\eqref{eq:gen:dro} and~\eqref{eq:gen:dro_prescriptor}, which involve an ambiguity set constructed form the rate function of the LDP. The out-of-sample disappointment of these predictors and prescriptors is guaranteed to be equal to $e^{-rT+o(T)}$, where $r$ is the radius of the ambiguity set. Due to its direct physical interpretation, we believe that it is natural for decision-makers to {\em choose}~$r$ in view of their risk tolerance instead of {\em calibrating} it algorithmically. Nevertheless, some decision-makers may want to calibrate $r$ via cross-validation with the goal to minimize the out-of-sample risk. In doing so, however, direct control over the out-of-sample disappointment is lost.

The main results of this paper rely on several assumptions, some of which could be generalized. Assumption~\eqref{ass:parametrization} requires that $\Theta$ constitutes a {\em finitely} parametrized ambiguity set. However, we believe that the results of Section~\ref{sec:optimal:dd:prescriptors} extend to {\em infinitely} parametrized ({\em i.e.}, non-parametric) ambiguity sets. For example, in \cite{ref:vanParys:fromdata-17} our results for finite-state i.i.d.\ processes are extended to i.i.d.\ processes with a continuous state space. This generalization does not require fundamentally new ideas but requires more sophisticated topological arguments that make the proofs less accessible. Assumption~\ref{ass:continuity} requires $c(x,\theta)$ to be uniformly continuous and bounded. It is non-restrictive for practical purposes. We believe that it can be relaxed to requiring that $c(x,\theta)$ be lower semi-continuous at the expense of complicating the proofs of Proposition~\ref{prop:continuity_cr}, Theorem~\ref{thm:optimality} and Theorem~\ref{thm:optimality_prescriptor}. Assumption~\ref{ass:ldp} requires the statistic $\widehat S$ to satisfy an LDP with a regular rate function and thus guarantees that the restricted meta-optimization problems~\eqref{eq:optimal:compressed} are solvable. This assumption seems more difficult to relax as our results critically rely on large deviations theory. Assumption~\ref{ass:suff} requires $\mathcal P = \{\mathbb P_\theta :
\theta \in\Theta\}$ to represent an exponential family,
and Assumption~\ref{ass:LDP2} captures standard technical conditions required for the G\"artner-Ellis Theorem (Theorem~\ref{thm:Gaertner-Ellis}). Together, these assumptions imply that $\widehat S$ is a
sufficient statistic satisfying an LDP, and thus they imply Assumption~\ref{ass:ldp}. Clearly, the statistic $\widehat S$ must satisfy
some notion of sufficiency for Theorems~\ref{thm:predictor:equivalence} and~\ref{thm:prescriptor:equivalence} to hold.  Nevertheless, we believe that Assumptions~\ref{ass:suff} and~\ref{ass:LDP2} can be relaxed and that Theorems~\ref{thm:predictor:equivalence} and~\ref{thm:prescriptor:equivalence} remain valid if~$\widehat S$ is only sufficient in an {\em asymptotic} sense. Finally, the meta-optimization problems~\eqref{eq:optimal} and~\eqref{eq:optimal:compressed} involve two {\em asymptotic} performance criteria, that is, the {\em asymptotic} in-sample risk and the {\em asymptotic} decay rate of the out-of-sample disappointment. While the asymptotic nature of these performance criteria is undesirable from a modeling perspective, the meta-optimization problems corresponding to a fixed sample size~$T$ may no longer admit Pareto dominant solutions. However, if the statistic~$\widehat S_T$ enjoys a finite sample guarantee, then the distributionally robust predictors and prescriptors~\eqref{eq:gen:dro} and~\eqref{eq:gen:dro_prescriptor} may still be {\em approximately} Pareto dominant.

%via the meta-optimization problems~\eqref{eq:optimal} and its compressed counterpart~\eqref{eq:optimal:compressed}. These optimality criteria have specific strengths but also certain limitations. While it would be desirable to find prescriptors that minimize the out-of-sample cost, minimizing the in-sample cost subject to the out-of-sample disappointment probability decay rate of our meta-optimization problems serves as a proxy for controlling the out-of-sample cost. One striking benefit of the meta-optimization problems (under the assumptions considered) is that they admit Pareto dominant solutions in closed form, which is of crucial practical importance. Interestingly, for specific settings our theory recovers specific prescriptors which have been widely and successfully used --- such as DRO prescriptors based on ellipsoidal ambiguity sets. We would like to point out that the performance criteria in \eqref{eq:optimal} and \eqref{eq:optimal:compressed} are asymptotic. Therefore, while these criteria are not directly applicable when working with finite data, the underlying clean optimality principle can guide the design of good prescriptors that can then be practically used on finite data. Our results (Theorems~3.2 and 4.3) show, that any data-driven prescriptor leading to an exponential decay of the out-of-sample disappointment probability necessarily will be biased. Recent results~\cite{ref:Amine-21} have shown that unbiasedness can be achieved if the decay rate is reduced.  

\textbf{Acknowledgements.} This research was supported by the Swiss National Science Foundation under the NCCR Automation, grant agreement~51NF40\_180545.

\bibliographystyle{plain}
\bibliography{references}

%%%%%%%%%%%%%%%%%%%%%%%%%%%%%%%%%%%%%%%%%%%%%%%%%%%%%%%%%%%%%
\appendix

%*********************************
\section{Invariance under coordinate transformations} \label{ssec:representation:invariance}
We now demonstrate that the restricted meta-optimization problems~\eqref{eq:optimal:compressed} are invariant under homeomorphic coordinate transformations of the state space~$\SS$ and the model space~$\cl\Theta$. Note first that if~$\widehat S$ is a statistic and~$\psi:\SS\to\SS$ is a homeomorphism, then~$\psi(\widehat S)=\{\psi(\widehat S_T)\}_{T\in\mb N}$ is also a statistic in the sense of Definition~\ref{def:estimator}. Indeed, $\psi\circ S_\infty$ is a local homeomorphism because~$\psi$ is continuous and~$S_\infty$ is a local homeomorphism. In addition, for any fixed~$\theta\in\Theta$, we know that~$\widehat S_T$ converges in probability to~$S_\infty(\theta)$ under~$\mb P_\theta$. The continuous mapping theorem~\cite[Theorem~3.2.4]{durrett_book} thus implies that~$\psi(\widehat S_T)$ converges in probability to~$\psi(S_\infty(\theta))$ under~$\mb P_\theta$. 

If the transformed statistic~$\psi(\widehat S)$ satisfies an \ac{ldp} with a regular rate function, then Theorems~\ref{thm:optimality} and~\ref{thm:optimality_prescriptor} imply that the corresponding distributionally robust predictors and prescriptors must provide Pareto dominant solutions to the compressed meta-optimization problems~\eqref{eq:optimal:compressed}. In the following we demonstrate that these Pareto dominant solutions corresponding to different homeomorphisms~$\psi$ are indeed all equivalent.

\begin{proposition}[Invariance under coordinate transformations of~$\SS$] \label{prop:invariance:predictor}
If Assumptions~\ref{ass:parametrization}, \ref{ass:continuity} and~\ref{ass:ldp} hold and $\psi:\SS\to\SS$ is a homeomorphism, then the statistic~$\psi(\widehat S)$ satisfies an~\ac{ldp} with regular rate function~$I_\psi(s,\theta) = I(\psi^{-1}(s),\theta)$. In addition, $\tilde c^\star_\psi (x, s) =\tilde c^\star (x,\psi^{-1}(s))$ is the distributionally robust predictor induced by~$I_\psi$, and~$\tilde x^\star_\psi(s)=\tilde x^\star (\psi^{-1}(s))$ is a corresponding distributionally robust prescriptor.
\end{proposition}
\begin{proof}[Proof of Proposition~\ref{prop:invariance:predictor}]
By the contraction principle \cite[Theorem~4.2.1]{dembo2009large}, which applies because~$\psi$ is continuous, the transformed statistic~$\psi(\widehat S)$ satisfies an \ac{ldp} with rate function~$I_\psi(s,\theta) = I(\psi^{-1}(s),\theta)$. As the homeomorphism~$\psi$ has a continuous inverse and preserves compactness, one readily verifies that $I_\psi$ inherits the radial monotonicity in~$\theta$, the continuity on~$\SS\times \Theta$, and the level-compactness from~$I$. Thus, $I_\psi$ is regular in the sense of Definition~\ref{def:rate_function}. By Definition~\ref{def:dro-predictor}, the distributionally robust predictor induced by~$I_\psi$ satisfies
  \begin{equation*}
    \tilde c^\star_\psi (x, s) = \left\{ \begin{array}{ll} \displaystyle \max_{\theta \in \cl \Theta}\, \left\{ c(x, \theta) : I_\psi(s, \theta)\leq r \right\} & \text{if } \exists\, \theta\in\cl\Theta \text{ with } I_\psi(s,\theta)\leq r, \\
    \displaystyle \sup_{\theta \in \cl \Theta}\, \phantom{\{} c(x, \theta) & \text{if } \nexists\, \theta\in\cl\Theta \text{ with } I_\psi(s,\theta)\leq r.
    \end{array}\right.
  \end{equation*}
Clearly, we have~$\tilde c^\star_\psi(x,s)=\tilde c^\star(x,\psi^{-1}(s))$ by the definition of~$I_\psi$. Next, define~$\tilde x^\star_\psi(s)=\tilde x^\star (\psi^{-1}(s))$, and note that~$\tilde x^\star_\psi$ inherits quasi-continuity from~$\tilde x^\star$ because~$\psi$ is continuous. As~$\tilde x^\star$ satisfies~\eqref{eq:gen:dro_prescriptor}, we further have
\[
    \tilde x^\star_\psi(s)=\tilde x^\star (\psi^{-1}(s))\in \arg \min_{x\in X} \tilde c^\star(x, \psi^{-1}(s)) = \arg\min_{x\in X} \tilde c^\star_\psi(x, s),
\]
and thus~$\tilde x^\star (\psi^{-1}(s))$ is a distributionally robust prescriptor corresponding to~$\tilde c^\star_\psi$.
\end{proof}

Proposition~\ref{prop:invariance:predictor} implies that the data-driven predictor induced by~$\tilde c^\star_\psi$ and the transformed statistic~$\psi(\widehat S)$ coincides with that induced by~$\tilde c^\star$ and the original statistic~$\widehat S$ because~$\tilde c^\star_\psi(x,\psi(\widehat S_T))=\tilde c^\star(x,\widehat S_T)$ for all~$T\in\mb N$. Similarly, we have~$\tilde x^\star_\psi(\psi(\widehat S_T))=\tilde x^\star(\widehat S_T)$ for all~$T\in\mb N$. Thus, homeomorphic transformations of the estimator~$\widehat S$ have no impact on how we map the raw data~$\xivec$ to a prediction of the cost or to a decision.

Similar invariance properties hold under coordinate transformations of the model space. To see this, note that if~$\varphi:\cl \Theta \to\cl \Theta$ is a homeomorphism, then~$\varphi$ maps~$\Theta$ onto~$\Theta$ thanks to a simple generalization of~\cite[Exercise~5.4]{massey1991basic} and because~$\cl\Theta$ is convex. This implies that the transformed ambiguity set~$\mc P_\varphi= \{ \mb P_{\varphi(\theta)} : \theta\in\Theta\}$ coincides with the original ambiguity set~$\mc P$. We now show that the key properties of ambiguity sets, model-based predictors and regular rate functions are preserved and %if the statistic~$\widehat S$ satisfies an \ac{ldp} with a regular rate function, then the regularity of the rate function is preserved and
that the distributionally robust predictors and prescriptors are invariant under homeomorphic coordinate transformations of~$\cl\Theta$.

\begin{proposition}[Invariance under coordinate transformations of~$\cl\Theta$] \label{prop:invariance:predictor:model}
If Assumptions~\ref{ass:parametrization}, \ref{ass:continuity} and~\ref{ass:ldp} hold and~$\varphi:\cl \Theta\to\cl \Theta$ is a homeomorphism, then~$\mc P_\varphi= \{ \mb P_{\varphi(\theta)} : \theta\in\Theta\}$ is a finitely parametrized ambiguity set in the sense of Assumption~\ref{ass:parametrization}, the model-based predictor~$c_\varphi(x,\theta)=c(x,\varphi^{-1}(\theta))$ satisfies Assumption~\ref{ass:continuity}, and the rate function~$I_\varphi(s,\theta) = I(s,\varphi^{-1}(\theta))$ is regular. In addition, the distributionally robust predictor and any corresponding distributionally robust prescriptor are invariant under this coordinate~transformation.
\end{proposition}
\begin{proof}[Proof of Proposition~\ref{prop:invariance:predictor:model}]
The assertions concerning~$\mc P_\varphi$ and $c_\varphi(x,\theta)$ follow directly from the defining properties of a homeomorphism. In addition, the transformed rate function~$I_\varphi(x,\theta)$ is non-negative and lower semi-continuous in~$s$ on~$\SS\times\cl\Theta$, and it satisfies the continuity and level-compactness conditions of Definition~\ref{def:rate_function}. All these properties are inherited from the original rate function~$I(x,\theta)$ because~$\varphi^{-1}$ is continuous. To show that~$I_\varphi(x,\theta)$ satisfies the radial monotonicity condition of Definition~\ref{def:rate_function}, we introduce the sets
\begin{equation*}
    A=\{\theta\in\Theta : I(s,\theta) <r\} \quad \text{and} \quad
    B=\{\theta\in\cl\Theta : I(s,\theta)\leq r\}
\end{equation*}
and note that~$\cl A=B$ because the original rate function is radially monotonic. Similarly, we introduce
\begin{equation*}
    A_\varphi=\{\theta\in\Theta : I_\varphi(s,\theta) <r\} \quad \text{and} \quad
    B_\varphi=\{\theta\in\cl\Theta : I_\varphi(s,\theta)\leq r\}.
\end{equation*}
By the definition of~$I_\varphi(s,\theta)$ and because~$\varphi$ maps~$\Theta$ onto~$\Theta$, we have~$A_\varphi=\varphi(A)$. Similarly, as~$\varphi$ maps~$\cl\Theta$ onto~$\cl\Theta$, we have~$B_\varphi=\varphi(B)$. To prove that the new rate function is radially monotonic, we need to show that~$\cl A_\varphi=B_\varphi$. 
As~$A_\varphi\subseteq B_\varphi$ and~$B_\varphi$ is closed thanks to the level-compactness of~$I_\varphi(s,\theta)$, we have~$\cl A_\varphi\subseteq B_\varphi$. To prove the converse inclusion, select any~$\theta \in B_\varphi$, and note that~$\varphi^{-1}(\theta)\in B=\cl A$. Thus, there exist~$\theta_k\in A$, $k\in\N$, such that~$\lim_{k\to \infty} \theta_k=\varphi^{-1} (\theta)$.  As~$\varphi(A)=A_\varphi$, we then have~$\varphi(\theta_k)\in A_\varphi$ for all~$k\in\mb N$, and as~$\varphi$ is continuous, we have~$\lim_{k\to \infty} \varphi(\theta_k)=\theta$. This implies that~$\theta\in\cl A_\varphi$. As~$\theta\in B_\varphi$ was chosen arbitrarily, we have thus shown that~$B_\varphi\subseteq \cl A_\varphi$ and consequently that~$I_\varphi(x,\theta)$ is regular.

The distributionally robust predictor of Definition~\ref{def:dro-predictor} and the distributionally robust prescriptor of Definition~\ref{def:dro-prescriptor} are thus manifestly invariant under homeomorphic coordinate transformations of~$\cl\Theta$.
\end{proof}

Propositions~\ref{prop:invariance:predictor} and~\ref{prop:invariance:predictor:model} testify to the reasonableness of Assumptions~\ref{ass:parametrization}, \ref{ass:continuity} and~\ref{ass:ldp}.

%%%%%%%%%%%%%%%%%%%%%%%%%%%%%
\section{Log-moment generating functions and Cram\'er functions}\label{app:table}
Table~\ref{tab:rate_functions} lists log-moment generating functions and their conjugates for popular distribution families.
%\begin{sidewaystable}[ht]
\begin{table} {\scriptsize{
\centering
\renewcommand{\arraystretch}{2}
\begin{tabular}{c@{\;}l|l@{~~}l@{\hspace{2em}}l@{~~}l@{\hspace{2em}}l}
 \hline\hline
 & Law of $\xi_t$ & $S_\infty(\theta)$ \ \ \  & log-MGF $\Lambda(\lambda,\theta)$ & $\text{dom}(\Lambda(\cdot,\theta))$& Cram\'er Function $\Lambda^*(s,\theta)$ & $\text{dom}(\Lambda^*(\cdot,\theta))$ \\
 \hline  &&&&&\\[-2ex]
 (a) & Normal & $\theta$ & $\theta\tpose \lambda +\frac12 \lambda\tpose \Sigma \lambda$ &$\Re^d$ & $\frac12 (s-\theta)\tpose \Sigma^{-1} (s-\theta)$ & $\Re^d$\\[1ex]
 (b) & Exponential & $1/\theta$ & $\log(\frac{\theta}{\theta-\lambda})$ & $(-\infty, \theta)$ & $\theta s - 1 - \log(\theta s)$ & $\Re_{++}$\\[1ex]
 (c) & Gamma & $k \theta$ & $-k \log(1-\theta\lambda)$ & $(-\infty,1/\theta)$ & $s/\theta - k +k \log(k\theta/s)$ & $\Re_{++}$\\[1ex]
 (d) & Poisson & $\theta$ & $\theta(e^\lambda -1)$ & $\Re$ & $s\log(s/\theta) - s+\theta$ & $\Re_{++}$ \\[0.0ex]
 (e) & Bernoulli & $\theta$ & $\log(1-\theta + \theta e^\lambda)$ & $\Re$ & $s\log( \frac{s(1-\theta)}{\theta(1-s)})-\log(\frac{1-\theta}{1-s})$ & $(0,1)$\\
 (f) & Geometric & $1/\theta$ & $\lambda + \log(\frac{\theta}{1-(1-\theta)e^\lambda})$ & $(-\infty, -\log(1-\theta))$ & $(s-1)\log(\frac{1-s}{s(\theta-1)})-\log(\theta s)$ & $(1,\infty)$ \\
 (g) & Binomial & $N\theta$ & $N \log(1-\theta + \theta e^\lambda)$ & $\Re$ & $s \log(\frac{s(\theta-1)}{\theta(s-N)})-N\log(\frac{N(1-\theta)}{N-s})$   & $(0,N)$ \\
 \hline
 \end{tabular}
\caption{Log-moment generating functions (log-MGFs) and their conjugates for popular distribution families.}
\label{tab:rate_functions}
}}\end{table}
%\end{sidewaystable} 

%%%%%%%%%%%%%%%%%%%%%%%%%%%%%%%%%%%%%%%%%%
\section{Autoregressive processes} \label{subsec:AR}
We now show that the restricted meta-optimization problems sometimes admit Pareto dominant solutions even if the training data is generated by an autoregressive process with an uncountable state space.

\subsection{Vector autoregressive processes with unknown drift} \label{subsec:AR:part1}
Assume now that the observable data $\{\xi_t\}_{t=1}^T$ follows a vector autoregressive process of the form
\begin{equation}
    \label{AR:process:general:general}	
    \xi_{t+1} = \theta_\star+\A \xi_t + \varepsilon_{t+1}\quad \forall t\in\N
\end{equation}
with state space~$\Xi=\Re^d$, where the drift term~$\theta_\star\in\Re^d$ is deterministic but unknown. Assume further that the disturbances~$\{\varepsilon_t\}_{t\in\N}$ are normally distributed with zero mean and known positive definite covariance matrix~$\Sigma\in\mathbb \Re^{d\times d}$ and that the initial state~$\xi_1$ and all disturbances are mutually independent under~$\mb P_\star$. Finally, assume that~$A\in\Re^{d\times d}$ is asymptotically stable in the sense that all of its eigenvalues reside strictly inside the complex unit circle. Hence, the process~$\{\xi_t\}_{t\in\N}$ is ergodic and admits a unique stationary distribution \cite[Example~3.43 and Proposition~3.44]{white2001asymptotic}. It is well known that the stationary distribution is Gaussian with mean vector~$(\mathbb{1}_d-\A)^{-1}\theta_\star$ and that its covariance matrix~$R_0$ is the unique solution to the discrete Lyapunov equation~$R_0 = A R_0 A\tpose + \Sigma$,
see, {\em e.g.}, \cite[Section~6.10\,E]{ref:Antsaklis-06}. In the remainder of this section we will assume that the process~$\{\xi_t\}_{t\in\N}$ is stationary under~$\mb P_\star$. This means that~$\xi_t$ follows the stationary distribution for every~$t\in\N$. An elementary calculation further reveals that the cross-covariance matrix~$R_{\delta}\in\Re^{d\times d}$ of any~$\xi_t$ and~$\xi_s$ with~$\delta =t-s$ is given by~$R_{\delta} = A^{\delta} R_0$ if~$\delta\geq 0$ and~$R_{\delta} =  R_0(A^{-\delta})\tpose$ if~$\delta < 0$.
% \begin{assumption}[Asymptotic stability] \label{ass:AR:stability}
% The matrix $\A$ is such that $\rho(\A)< 1$.
% \end{assumption}
% Note that $\rho(\A)$ denotes the largest absolute eigenvalue (spectral radius) of $\A$. 
%Each parameter $\theta=\{A,\mu,\Sigma\}$ encodes a probability measure $\mb P_\theta$ on $(\Omega, \mc F)$ such that the marginal distribution of $\xi_t$ under $\mb P_\theta$ converges weakly to the invariant distribution.
% \begin{equation*}
% \mb P_\theta[\xivec\in B ] 
% %= \int_\Gamma \prod_{t=0}^{T-1} \pi_\theta(\xi_t,\d \xi_{t+1}) 
% = \int_B \prod_{t=1}^{T-1} \frac{1}{(2\pi)^{m/2}\sqrt{\det \Sigma}}\exp\left( -\frac{1}{2}(\xi_{t+1} - \A \xi_t-\mu)\tpose\Sigma^{-1}(\xi_{t+1} - \A \xi_t-\mu)\right) \d \xi_{t+1},
% \end{equation*}
% for all Borel sets $B\subseteq \Xi^T$ and $T\in \N$. The process \eqref{AR:process:general:general} admits a unique stationary  distribution, since $\{\xi_t\}_{t\in\N}$ is an ergodic process \cite[Example 3.43 and Proposition 3.44]{white2001asymptotic}. The stationary distribution of the process \eqref{AR:process:general:general} is also normal with mean $(\mathbb{1}_m-\A)^{-1}\mu$, where the inverse exists as $A$ is asymptotically stable. 

% We consider two different  families of data-generating processes~\eqref{AR:process:general:general} with different statistics and corresponding \ac{ldp}s:
% \begin{enumerate}
% \item $\mu$ unknown, $A$ known, $\Sigma$ known; \label{AR:scenario1}
% \item $A$ unknown, $\Sigma$ known, $\mu=0$ known. \label{AR:scenario2}
% %\item $C=0$, $A$ unknown. \label{AR:scenario3}
% \end{enumerate}

Assume now that the drift~$\theta_\star$ is known to belong to an open convex set~$\Theta\subseteq \Re^d$ that captures any available structural information. We then define an ambiguity set~$\mc P=\{\mb P_\theta:\theta\in\Theta\}$, where each~$\theta\in\Theta$ encodes a probability measure~$\mb P_\theta$ on~$(\Omega, \mc F)$ under which the observations~$\{\xi_t\}_{t\in\N}$ are jointly normally distributed with mean vector~$\mb E_\theta[\xi_t] = (\mb 1_d-A)^{-1}\theta$ for all~$t\in\N$ and cross-covariance matrix~$\mb E_\theta[(\xi_t-\mb E_\theta[\xi_t])(\xi_{s}-\mb E_\theta[\xi_s])^\top] = R_{t-s}$ for all~$s,t\in\N$.
%$\{\xi_t\}_{t\in\N}$ constitutes a stationary autoregressive process satisfying a variant of~\eqref{AR:process:general:general} with drift~$\theta$ instead of~$\theta_\star$. 
In this setting, a natural estimator for~$\theta$ is the scaled sample mean 
\begin{align}
    \label{estimator:SAA:scaled}
    \widehat{S}_T= & (\mathbb{1}_d -\A)\frac{1}{T}\sum_{t=1}^T \xi_t \quad \forall T\in\N
\end{align}
with state space $\SS=\Re^d$. By~\cite[Theorem 3.34]{white2001asymptotic}, which applies because the data process is ergodic, $\widehat S$ represents a consistent model estimator in the sense of Definition~\ref{def:estimator}. Consequently, we have~$S_\infty(\theta)=\theta$ for all~$\theta\in\Theta$, and~$\SS_\infty=\Theta$. In addition, the function~$S_\infty$ is clearly a local homeomorphism. The next proposition asserts that the statistic~$\widehat S$ also satisfies an \ac{ldp} with a regular quadratic rate function.

%implies that Assumption~\ref{ass:ldp} holds.

\begin{proposition}[\ac{ldp} for stationary autoregressive processes with unknown drift] \label{thm:AR:multi-dim:modelclass:2}
If $\{\xi_t\}_{t\in\N}$ follows a stationary autoregressive process of the form~\eqref{AR:process:general:general} with drift~$\theta\in \Theta$, then the scaled sample mean~\eqref{estimator:SAA:scaled} satisfies an \ac{ldp} with regular convex quadratic rate function~$I(s,\theta)= \frac{1}{2}(s-\theta)\tpose  \Sigma^{-1} (s-\theta)$.
\end{proposition}

Proposition~\ref{thm:AR:multi-dim:modelclass:2} generalizes \cite[Exercise~2.3.23]{dembo2009large}, which focuses on scalar autoregressive processes without drift. The results of this section imply via Theorems~\ref{thm:optimality} and~\ref{thm:optimality_prescriptor} that the distributionally robust predictor~$\tilde c^\star$ with an ellipsoidal ambiguity set for~$\theta$ around the scaled sample mean~\eqref{estimator:SAA:scaled} and the corresponding prescriptor~$\tilde x^\star$ provide Pareto dominant solutions for the restricted meta-optimization problems~\eqref{eq:optimal:compressed}. As the scaled sample mean fails to be a sufficient statistic for~$\theta$, however, we are unable to find Pareto dominant solutions for the original meta-optimization problems~\eqref{eq:optimal}. Details are omitted for brevity.

\begin{remark}[I.i.d.\ processes as degenerate autoregressive processes]
Any \ac{iid} process of multivariate normal random variables of the kind studied in Section~\ref{sec:iid-parametric} can alternatively be interpreted as a degenerate vector autoregressive process with a vanishing coefficient matrix~$A=0$. It is therefore not surprising that if~$A=0$, then the scaled sample mean~\eqref{estimator:SAA:scaled} coincides with the ordinary sample mean~\eqref{cts:sample:mean}, and the rate function derived in Proposition~\ref{thm:AR:multi-dim:modelclass:2} coincides with the Cram\'er function in Table~\ref{tab:rate_functions}(a).
\end{remark}

\subsection{Scalar autoregressive processes with unknown coefficient} 
\label{subsec:AR2}
Assume now that the observable data $\{\xi_t\}_{t=1}^T$ follows a scalar autoregressive process of the form
\begin{equation}
    \label{AR:process:general:general2}
    \xi_{t+1} = \theta_\star\xi_t + \varepsilon_{t+1}\quad \forall t\in\N
\end{equation}
with state space~$\Xi=\Re$, where the autoregressive coefficient~$\theta_\star\in(-1,1)$ is deterministic but unknown. Assume further that the disturbances~$\{\varepsilon_t\}_{t\in\N}$ are normally distributed with known mean~$\mu\in\Re$ and variance~$\sigma^2>0$ and that the initial state~$\xi_1$ and all disturbances are mutually independent under~$\mb P_\star$. As in Section~\ref{subsec:AR}, we finally assume that the process~$\{\xi_t\}_{t=1}^T$ is stationary under~$\mb P_\star$. In this case~$\mb P_\star$ belongs to an ambiguity set~$\mc P=\{\mb P_\theta:\theta\in\Theta\}$, where each~$\theta\in\Theta=(-1,1)$ encodes a probability measure~$\mb P_\theta$ on~$(\Omega, \mc F)$ under which the observations~$\{\xi_t\}_{t\in\N}$ are jointly normally distributed with mean~$\mb E_{\theta}[\xi_t]=\mu/(1-\theta)$ for all~$t\in\N$ and autocovariance~$\mb E_{\theta}[(\xi_t-\mb E_{\theta}[\xi_t])(\xi_s-\mb E_{\theta}[\xi_s])] = \sigma^2 \theta^{|t-s|}/(1-\theta^2)$ for all~$s,t\in\N$. 

In the following we investigate two complementary estimators for the autoregressive coefficient~$\theta$. We first study the least squares estimator, which is defined through
\begin{equation}
    \label{least-squares:estimator}
    \widehat{S}_T=\frac{\sum_{t=2}^T \xi_t \xi_{t-1}}{\sum_{t=2}^T \xi^2_{t-1}} \quad \forall T\in\N.
\end{equation}
By construction, the state space of~$\widehat S$ is given by~$\SS=\Re$, and it is well known that~$\widehat{S}_T$ converges $\mb P_\theta$-almost surely to~$\theta$ %and $\sqrt{T}(\widehat{S}_T-\theta)$ converges in law under~$\mb P_\theta$ to a normal distribution with mean~$0$ and variance~$1-\theta^2$
for all~$\theta\in\Theta$; see, {\em e.g.}, \cite{bercu1997quadraticLDP}. Thus, $S_\infty(\theta)=\theta$ represents a local homeomorphism, the set of asymptotic estimator realizations simplifies to~$\SS_\infty=\Theta$ and $\widehat S$ is a consistent model estimator in the sense of Definition~\ref{def:estimator}. Moreover, it is also well known that~$\widehat{S}$ satisfies an \ac{ldp}.

\begin{proposition}[\ac{ldp} for stationary autoregressive processes with unknown coefficient (I)] \label{thm:AR:multi-dim:modelclass:3}
If $\{\xi_t\}_{t\in\N}$ follows a stationary autoregressive process of the form~\eqref{AR:process:general:general2} with autoregressive parameter~$\theta\in \Theta$, then the least squares estimator~\eqref{least-squares:estimator} satisfies an \ac{ldp} with regular rate function
\begin{equation}
    \label{ratefunction:least-squares}
    I(s,\theta)= \left\{ \begin{array}{ll}
    \frac{1}{2} \log\left( \frac{1- 2\theta s+\theta^2}{1-s^2} \right) & \text{if }s\in[a(\theta),b(\theta)], \\
    \log\left( | \theta - 2s | \right) &\text{otherwise},
    \end{array} \right.
\end{equation}
where $a(\theta)=\frac{1}{4}(\theta - \sqrt{\theta^2+8})$ and $b(\theta)=\frac{1}{4}(\theta + \sqrt{\theta^2+8})$.
\end{proposition}

%The ambiguity set $\mathcal{U}_r(s):=\{ \theta\in\Theta :  I(s,\theta)\leq r\}$ induced by the rate function~\eqref{ratefunction:least-squares}, can be easily seen to be an ellipsoid in the regime $s\in[a(\theta),b(\theta)]$ and the intersection of two half-spaces otherwise, see Figure~\ref{fig:uncertaintysets:AR2} for a pictorial representation.

An alternative estimator for the autoregressive coefficient~$\theta$ is the Yule-Walker estimator defined through
\begin{equation}
    \label{Yule-Walker:estimator}
    \widehat{S}_T=\frac{\sum_{t=2}^T \xi_t \xi_{t-1}}{\sum_{t=1}^T \xi^2_{t}} \quad \forall T\in\N.
\end{equation}
By construction, the state space of~$\widehat S$ is given by~$\SS=\Re$, and~$\widehat S_T$ converges $\mb P_\theta$-almost surely to~$\theta$ for all~$\theta\in\Theta$ \cite{bercu1997quadraticLDP}.
%and that $\sqrt{T}(\widehat{S}_T-\theta) \to \mathcal{N}(0,1-\theta^2)$. 
Hence, $S_\infty(\theta)=\theta$ and $\SS_\infty=\Theta$, which means that the Yule-Walker estimator constitutes a consistent model estimator. Interestingly, it satisfies a different \ac{ldp} than the least squares estimator.

\begin{proposition}[\ac{ldp} for stationary autoregressive processes with unknown coefficient (II)] \label{thm:AR:multi-dim:modelclass:4}
If $\{\xi_t\}_{t\in\N}$ follows a stationary autoregressive process of the form~\eqref{AR:process:general:general2} with autoregressive parameter~$\theta\in \Theta$, then the Yule-Walker estimator~\eqref{Yule-Walker:estimator} satisfies an \ac{ldp} with regular rate function
\begin{equation} \label{ratefunction:Yule:Walker}
I(s,\theta)= \left\{ \begin{array}{ll}
\frac{1}{2} \log\left( \frac{1- 2\theta s+\theta^2}{1-s^2} \right) & \text{if }s\in(-1,1), \\
0 & \text{if }s=\theta=1\text{ or } s=\theta=-1, \\
+\infty &\text{otherwise}.
\end{array} \right.
\end{equation}
\end{proposition}

% \item[(iii)] \label{item:sf:estimator} Consider the base stochastic process $\mb P_0$ associated with the parameter value $\theta=0$. We remark that we have the following well known, c.f., \cite[Equation 3.2.7]{kuchler2006exponential} exponential representation
% \[
%   \frac{\d\mb P_\theta}{\d\mb P_0}[\xivec] = \exp\left(\textstyle\theta \cdot \sum_{t=1}^T \xi_t\xi_{t-1} -\frac{\theta^2}{2}\cdot \sum_{t=1}^T\xi_{t-1}^2\right).
% \]

Figure~\ref{fig:AR:estimators} visualizes the rate functions~\eqref{ratefunction:least-squares} and~\eqref{ratefunction:Yule:Walker} for fixed values of the estimator realization~$s$ and the model~$\theta$. As is also evident from their definitions, the two rate functions coincide whenever~$s\in[a(\theta,b(\theta)]$. In general, however, the rate function corresponding to the Yule-Walker estimator majorizes the one corresponding to the least squares estimator. This indicates that the probability of unlikely estimator realizations decays faster when we use the Yule-Walker estimator. One can show that neither the least squares nor the Yule-Walker estimator represent a sufficient statistic for~$\theta$. Therefore, the corresponding distributionally robust predictors and prescriptors cannot be used to construct Pareto dominant solutions for the original meta-optimization problems~\eqref{eq:optimal}. As both statistics satisfy an \ac{ldp} with a regular rate function (as shown in Propositions~\ref{thm:optimality} and~\ref{thm:optimality_prescriptor}), however, the corresponding distributionally robust predictors and prescriptors are strongly optimal in the respective {\em restricted} meta-optimization problems~\eqref{eq:optimal:compressed}. We also emphasize that these predictor-prescriptor pairs are {\em not} equivalent. Indeed, for any desired decay rate~$r\ge 0$ of the out-of-sample disappointment and for any fixed estimator realization~$s\in\SS$, the predictor induced by the Yule-Walker estimator is less conservative than the one induced by the least squares estimator because the rate ball of radius~$r$ around~$s$ corresponding to the Yule-Walker estimator is always contained in the rate ball of radius~$r$ around~$s$ corresponding to the least squares estimator. Intuitively, the Yule-Walker estimator thus results in a less conservative predictor with the same guarantees on the out-of-sample disappointment.

%%%%%%%% 
\begin{figure}[!htb] 
\centering
   \subfloat[$I(s,\theta)$ for fixed $\theta=0.5$; see also {\cite[Figure~1]{bercu1997quadraticLDP}}.]{\input{figure_LDP.tex} } \qquad
   \subfloat[$I(s,\theta)$ for fixed $s=0.8$.]{\input{figure_LDP_2.tex} }
    \caption[]{Comparison of the rate functions corresponding to the least squares and Yule-Walker estimators. }
    \label{fig:AR:estimators}
\end{figure}

\section{Proofs of Section~\ref{sec:optimal:dd:prescriptors}}

%\subsection{Proof of Proposition~\ref{prop:continuity_cr}} \label{app:proof:continuity:cr}

In order to prove Proposition~\ref{prop:continuity_cr}, we have to recall several notions of continuity for set-valued mappings.

\begin{definition}[Continuity of set-valued mappings {\cite[Chapter~VI, Section~1]{berge1997topological}}] 
\label{def:hemicontinuity}
Consider a set-valued mapping $\Gamma:X\rightrightarrows Y$ between two topological spaces.
\begin{itemize}
\item[(i)] $\Gamma$ is called lower semi-continuous (lsc) at $x_0$ if for every open set $V\subseteq Y$ with $\Gamma(x_0)\cap V\neq \emptyset$ there exists an open neighborhood $U\subseteq X$ of $x_0$ such that $\Gamma(x)\cap V\neq \emptyset$ for all $x\in U$;
\item[(ii)] $\Gamma$ is called upper semi-continuous (usc) at $x_0$ if for every open set $V\subseteq\Theta$ with $\Gamma(x_0)\subseteq V$ there exists an open neighborhood $U\subseteq Y$ of $x_0$ such that $\Gamma(x)\subseteq V$ for all $x\in U$;
\item[(iii)] $\Gamma$ is called continuous at $x_0$ if it is both lsc and usc at $x_0$;
\item[(iv)] $\Gamma$ is called lsc  if it is lsc at every point $x_0\in X$;
\item[(v)] $\Gamma$ is called usc  if it is compact-valued and usc at every point $x_0\in X$;
\item[(vi)] $\Gamma$ is called continuous  if it is lsc and usc.
\end{itemize}
\end{definition}

\begin{proof}[Proof of Proposition~\ref{prop:continuity_cr}]
We first show that the distributionally robust predictor~$\tilde c^\star(x,s)$ is continuous in~$(x,s)$ on~$X\times\SS_\infty$. To this end, define the set-valued mapping $\Gamma: \SS\rightrightarrows\cl\Theta$ through $\Gamma(s)=\{\theta\in\cl\Theta : I(s,\theta)\leq r \}$ for every $s\in \SS$. Note that the graph $\{ (s,\theta)\in \SS \times \cl \Theta : I(s,\theta)\leq r\}$ of~$\Gamma$ is compact because the regular rate function~$I(s,\theta)$ has compact sublevel sets. % Definition~\ref{def:rate_function} 
Hence, $\Gamma$ has a closed graph and is compact-valued, which implies via~\cite[Proposition~1.4.8]{ref:Aubin-09} that~$\Gamma$ is usc. Recall now that~$\cl\Theta$ is equipped with the subspace topology induced by the Euclidean topology on~$\mb R^d$, and choose any~$s_0\in \SS_\infty$
%Let $S_\infty(\theta_0')=s_0\in \SS$ be arbitrary.
and any open set~$V\subseteq \cl\Theta$ with~$\Gamma(s_0)\cap V\neq \emptyset$. As~$\Theta$ is the relative interior of a convex subset of~$\mb R^d$ (see Assumption~\ref{ass:parametrization}), it is open with respect to the subspace topology on~$\cl\Theta$. Thus, both~$V$ and~$\Theta$ are open. This implies that~$V\subseteq \interior \cl \Theta = \Theta$, where the equality follows from \cite[Theorem~6.3]{rockafellar1970convex}, which applies because~$\Theta$ is convex and open. Hence, there exists~$\theta_0\in V\subseteq \Theta$ with~$I(s_0, \theta_0)\leq r$. In the following we may assume without loss of generality that~$I(s_0, \theta_0)<r$. Suppose to the contrary that~$I(s_0,\theta_0)=r$. Since~$s_0\in\SS_\infty$ and~$r>0$ and since the regular rate function~$I(s,\theta)$ is radially monotonic in~$\theta$, there exist~$\theta_k\in \Theta$, $k\in\N$, such that $I(s_0,\theta_{k})<r$ for all~$k\in\N$ and $\lim_{k\to\infty}\theta_{k}=\theta_0$. As~$\theta_0\in V$ and~$V$ is open, there further exists~$k_V\in\N$ such that~$\theta_k\in V$ and $I(s_0,\theta_k)<r$ for all~$k\ge k_V$. We may thus re-define~$\theta_0$ as~$\theta_k$ for any~$k\ge k_V$. Next, define~$U= \{s\in \SS_\infty:I(s, \theta_0)<r\}$, and note that~$U$ is open because~$\SS_\infty$ is open and the regular rate function~$I(s,\theta)$ is continuous on~$\SS\times\Theta$. By construction, we have that~$s_0\in U$ and~$\theta_0\in \Gamma(s)\cap V$ for all~$s\in U$. Thus, $\Gamma$ is lsc at~$s_0$. As~$s_0\in\SS_\infty$ was chosen arbitrarily, $\Gamma$ is indeed lsc on~$\SS_\infty$. Being lsc as well as usc, $\Gamma$ represents a continuous set-valued mapping on~$\SS_\infty\subseteq\SS$. In addition, $\Gamma(s)$ is non-empty for every~$s\in \SS_\infty$ because there exists~$\theta\in\Theta$ with~$s=S_\infty(\theta)$. Indeed, by the discussion after Definition~\ref{def:wldt}, we have~$I(s, \theta)=0$ and thus~$\theta\in \Gamma(s)$. As the model-based predictor~$c(x,\theta)$ is continuous on $X\times \cl\Theta$ due to the arguments outlined after Assumption~\ref{ass:continuity}, we may finally invoke Berge's maximum theorem \cite[pp.~115--116]{berge1997topological}) to conclude that the distributionally robust predictor $\tilde c^\star(x, s) = \max_{\theta \in \Gamma(s)} c(x, \theta)$ is continuous on~$X\times \SS_\infty$. 

We may use a similar but significantly simpler reasoning to demonstrate that~$\tilde c^\star(x,s)$ is continuous in~$x$ on~$X\times\SS$. The simplification arises because the set-valued mapping~$\Gamma(s)$ is constant and thus trivially continuous in~$x$ for any fixed~$s\in\SS$. Finally, $\tilde c^\star(x,s)$ inherits boundedness from~$c(x,\theta)$; see Assumption~\ref{ass:continuity}.
\end{proof}

Theorems~\ref{thm:optimality} and~\ref{thm:optimality_prescriptor} significantly generalize Theorems~3, 4, 6 and~7 in~\cite{ref:vanParys:fromdata-17}, which apply only to finite-state i.i.d.\ processes and where $\widehat S_T$ reduces to the empirical distribution that satisfies an LDP with rate function~$I(s,\theta)=\D{s}{\theta}$. Even though we can adopt similar proof techniques as in \cite{ref:vanParys:fromdata-17}, the proofs of Theorems~\ref{thm:optimality} and~\ref{thm:optimality_prescriptor} require more care because we have to handle general statistics and LDPs that admit general regular rate functions. Indeed, the mere notion of a regular rate function is a new concept introduced in this paper; see Definition~\ref{def:rate_function}. In particular, unlike in~\cite{ref:vanParys:fromdata-17}, we need to account here for the possibility that $\Theta$ differs from~$\mathbb S$, that the set~$\mathbb S_\infty$ of asymptotic estimator realizations is a strict subset of the interior of~$\mathbb S$ and that the regular rate function~$I(s,\theta)$ fails to be convex in~$s$. In addition, unlike in~\cite{ref:vanParys:fromdata-17}, the LDP bound~\eqref{eq:ldp_exponential_rates_ub:old} involves the closure of the atypical set~$\mathcal D$, which requires more subtle topological arguments.

\begin{proof}[Proof of Theorem~\ref{thm:optimality}] 
We first show (Step~1) that $\tilde c^\star$ is feasible in problem~\eqref{eq:optimal-predictor:compressed}, and subsequently (Step~2) we demonstrate that~$\tilde c^\star$ Pareto dominates any other feasible solution of problem~\eqref{eq:optimal-predictor:compressed}.

{\em Step~1.} Proposition~\ref{prop:continuity_cr} readily implies that~$\tilde c^\star\in\tilde{\mc C}$. 
It remains to be shown that the out-of-sample disappointment of $\tilde c^\star$ decays at a rate of at least $r$. To this end, fix any~$x\in X$ and~$\theta\in\Theta$, and define the~sets
\[
    A(x,\theta)=\{s\in\SS : c(x,\theta)>\tilde c^\star(x,s)\}\quad\text{and}\quad B(x,\theta)=\{s\in\SS : I(s,\theta)>r\}.
\]
We may assume without loss of generality that $A(x,\theta)\neq \emptyset$ for otherwise the out-of-sample disappointment~$\mb P_\Theta[\widehat S_T\in A(x,\theta)]$ vanishes for all~$T\in \mb N$ and thus decays at any exponential rate. We will now show that~$A(x,\theta)\subseteq B(x,\theta)$. To this end, choose any~$s\in A(x,\theta)$, and assume that~$I(s,\theta)\leq r$. Thus, we~have
\[
    c(x,\theta)> \tilde c^\star(x,s)=\max_{\theta' \in \cl \Theta}\, \left\{ c(x, \theta') : I(s, \theta')\leq r \right\}\geq c(x,\theta),
\]
where the strict inequality holds because~$s\in A(x,\theta)$, whereas the equality and the weak inequality follow from the definition of~$\tilde c^\star(x,s)$ and the assumptions that~$\theta\in\Theta$ and~$I(s,\theta)\leq r$. The resulting conclusion is manifestly false, which implies that~$I(s,\theta)> r$, that is, $s\in B(x,\theta)$. As~$s\in A(x,\theta)$ was chosen arbitrarily, we have thus shown that~$A(x,\theta)\subseteq B(x,\theta)$. This result further implies that
\begin{equation}
    \label{eq:A-B-inclusion}
    \cl A(x,\theta)\subseteq \cl B(x,\theta) \subseteq \{s\in\SS : I(s,\theta)\geq r\},
\end{equation}
where the second inclusion holds because the set on the right hand side covers~$B(x,\theta)$ and is closed thanks to the continuity of the regular rate function~$I(s,\theta)$ in~$s$ on~$\SS$. The above reasoning implies that
\begin{align*}
\limsup_{T\to\infty}\frac 1T \log \mb P_\theta [ c(x,\theta) > \tilde c^\star(x,\widehat S_T) ]&=\limsup_{T\to\infty}\frac 1T \log \mb P_\theta [\widehat S_T\in A(x,\theta) ] \leq - \inf_{s \in \cl A(x,\theta)} ~ I(s,\theta) \leq -r ,
\end{align*}
where the first inequality follows from~\eqref{eq:ldp_exponential_rates_ub:old}, which applies because~$\widehat S$ satisfies an \ac{ldp} with rate function~$I(s,\theta)$, and the second inequality is a direct consequence of~\eqref{eq:A-B-inclusion}. As~$x\in X$ and~$\theta\in\Theta$ were chosen arbitrarily, we may thus conclude that~$\tilde c^\star$ is feasible in problem~\eqref{eq:optimal-predictor:compressed}.

% To establish the second inequality recall that we have to show that for any $\theta\in\Theta$ and $x\in X$
% \begin{equation}\label{eq:feasiblity:pred:step:to:show}
%     \cl A(\theta)\subseteq \{s\in\SS : I(s,\theta)\geq r\}.
% \end{equation}
% To show, that \eqref{eq:feasiblity:pred:step:to:show} indeed holds, by the definition of $\tilde c^\star$, \eqref{eq:gen:dro}, we have $A(\theta)\subseteq B(\theta)$. Moreover, we recall that since $I$ is continuous in the first argument on $\SS$ the set $\{s\in\SS : I(s,\theta)\geq r\}$ is closed and hence $\cl B(\theta)\subseteq \{s\in\SS : I(s,\theta)\geq r\}$. Using $A(\theta)\subseteq B(\theta)$, this implies $\cl A(\theta)\subseteq \{s\in\SS: I(s,\theta)\geq r\}$ and finally \eqref{eq:feasiblity:pred:step:to:show} indeed holds, which completes the proof.
%\end{proof}

%The following theorem establishes that $\tilde c^\star$ is not only a feasible but also a strongly optimal solution for the multi-objective optimization problem~\eqref{eq:optimal-predictor:compressed}. This means that if an arbitrary compressed predictor $\tilde c$ predicts a lower expected cost than $\tilde c^\star$ even for a single asymptotically relevant realization $s\in \SS_\infty$ of the statistic $\widehat S$, then $\tilde c$ must suffer from a larger out-of-sample disappointment rate than $\tilde c^\star$.

%\begin{proof}

{\em Step~2.} We now prove that~$\tilde c^\star$ Pareto dominates every other feasible solution of problem~\eqref{eq:optimal-predictor:compressed}. Assume to the contrary that there exists a compressed data-driven predictor $\tilde c\in\tilde{\mc C}$ that is feasible in~\eqref{eq:optimal-predictor:compressed} but is {\em not} Pareto dominated by~$\tilde c^\star$. Hence, there exist a decision~$x_0\in X$ and a model~$\theta_0\in\Theta$ with 
\[
    \varepsilon=\tilde c^\star(x_0,S_\infty(\theta_0))-\tilde c(x_0,S_\infty(\theta_0))>0.
\]
As we will see, the above inequality implies that~$\tilde c$ is infeasible in~\eqref{eq:optimal-predictor:compressed}. This contradiction will reveal that our initial assumption must have been false and that there cannot be any feasible~$\tilde c$ that Pareto dominates~$\tilde c^\star$.

Define~$s_0= S_\infty(\theta_0)$, and recall from the discussion below Definition~\ref{def:wldt} that $I(s_0,\theta_0)=0$. Thus, we have
\begin{equation}
    \label{eq:dro-predictor-s_0}
     \tilde c^\star(x_0, s_0)=\max_{\theta \in \cl \Theta}\, \left\{ c(x_0, \theta) : I(s_0, \theta)\leq r \right\}.
\end{equation}
By construction, $\theta_0$ is feasible in~\eqref{eq:dro-predictor-s_0}. From the discussion after Definition~\ref{def:dro-predictor} we further know that the maximization problem~\eqref{eq:dro-predictor-s_0} is solvable. In the following we denote by~$\theta^\star\in\cl\Theta$ an arbitrary maximizer. Feasibility of~$\theta^\star$ then guarantees that~$I(s_0,\theta^\star)\leq r$, and optimality implies the identity
\begin{equation}
	\label{eq:inequality:pred:2}
	\tilde c^\star(x_0, s_0)= c(x_0, \theta^\star).
\end{equation}
% We claim that there exists a model $\theta_0\in\Theta$ such that $I(s'_0,\theta_0)=r_0<r$ and such that 
% \begin{equation}
% 	\label{eq:inequalityPred3}
% 	c(x, \theta^\star)< c(x, \theta_0) +\varepsilon.
% \end{equation}
% Suppose $\theta^\star$ is such that $I(s_0',\theta^\star)=r$, otherwise we choose $\theta_0=\theta^\star$. 
Recall now that~$s_0\in\SS_\infty$ and that~$r>0$. Recall also that the rate function~$I(s,\theta)$ is regular and thus radially monotonic thanks to Assumption~\ref{ass:ldp}. This implies that there exist~$\theta^\star_{k}\in \Theta$, $k\in\N$, such that $I(s_0,\theta^\star_{k})<r$ for all~$k\in\N$ and $\lim_{k\to\infty}\theta^\star_{k}=\theta^\star$. In addition, as~$c(x,\theta)$ is continuous on~$X\times\cl\Theta$ thanks to the discussion after Assumption~\ref{ass:continuity}, this further implies that there exists~$\theta_0^\star\in\Theta$ with~$I(s_0,\theta_0^\star)=r_0<r$ and 
\begin{equation*}
	c(x_0, \theta^\star)< c(x_0, \theta_0^\star) +\varepsilon.
\end{equation*}
%Hence for any $\delta>0$ there exists $\theta_0\in\Theta$ such that $I(s_0,\theta_0)<r$ and $\|\theta_0-\theta^\star\|\leq \delta$, which implies \eqref{eq:inequalityPred3} by the continuity of $c$, Assumption~\ref{ass:continuity}.
Using this inequality, we then find
\begin{align}
	\label{eq:theta'-feasibility}
	\tilde c(x_0,s_0) = \tilde c^\star(x_0,s_0)-\varepsilon =  c(x_0,\theta^\star) -\varepsilon < c(x_0,\theta^\star_0), %\leq \tilde c^\star(x_0,s_0)
\end{align}
where the first equality follows from the definitions of~$\varepsilon$ and~$s_0$, and the second equality holds due to~\eqref{eq:inequality:pred:2}. %and the second inequality holds because~$\theta_0^\star$ is (strictly) feasible in~\eqref{eq:dro-predictor-s_0}.
In the following we will use~\eqref{eq:theta'-feasibility} to show that the prediction disappointment~$\mb P_{\theta_0^\star} [c(x, \theta^\star_0) > \tilde c(x,\widehat S_T)]$ of~$\tilde c$ under decision~$x_0$ and model~$\theta_0^\star$ decays no faster than~$e^{r_0T}$ for large sample sizes~$T$. As~$r_0<r$, this will imply that~$\tilde c$ is {\em in}feasible in~\eqref{eq:optimal-predictor:compressed}. To this end, we define the set of disappointing realizations of~$\widehat S$ as
\[
	 \mc D(x_0,\theta^\star_0) = \{s \in \SS:c(x_0, \theta_0^\star) > \tilde c(x_0, s)\}.
\]
By~\eqref{eq:theta'-feasibility}, this disappointment set contains $s_0$. As~$s_0\in\SS_\infty\subseteq\SS$ and as $\SS_\infty$ is open, we further know that~$s_0$ resides in the interior of~$\SS$. In addition, $\tilde c(x,s)$ is continuous in~$s$ on~$X\times\SS_\infty$. We may thus conclude that~$s_0$ belongs in fact to the interior of~$\mc D(x_0,\theta^\star_0)$ and that
\[
	\inf_{s\in \interior \mc D(x_0,\theta^\star_0)} \, I(s,\theta^\star_0) \leq I(s_0,\theta^\star_0)= r_0,
\]
where the equality holds by the definition of $r_0$. As the statistic~$\widehat S$ satisfies an \ac{ldp} with rate function~$I(s,\theta)$, the above inequality in conjunction with~\eqref{eq:ldp_exponential_rates_lb:old} finally implies that
\[
	-r<-r_0\leq-\inf_{s\in \interior \mc D(x_0,\theta^\star_0)} \, I(s,\theta^\star_0) \leq \liminf_{T\rightarrow \infty}\frac 1T \log \mb P_{\theta^\star_0} [  \widehat S_T \in \mc D(x_0,\theta^\star_0) ].
\]
By the definition of~$\mc D(x_0,\theta^\star_0)$, this means that the out-of-sample disappointment of~$\tilde c$ corresponding to~$x_0$ and~$\theta_0^\star$ decays strictly slower than~$e^{rT}$ as~$T$ grows. This in turn contradicts our assumption that~$\tilde c$ is feasible in~\eqref{eq:optimal-predictor:compressed} and thus implies that~$\tilde c^\star$ is a Pareto dominant solution of problem~\eqref{eq:optimal-predictor:compressed}.
\end{proof}

%%%%%
\begin{proof}[Proof of Theorem~\ref{thm:optimality_prescriptor}]
We first show that $(\tilde c^\star,\tilde x^\star)$ is feasible in problem~\eqref{eq:optimal-prescriptor:compressed} (Step~1), and subsequently we demonstrate that~$(\tilde c^\star,\tilde x^\star)$ Pareto dominates any other feasible solution of problem~\eqref{eq:optimal-prescriptor:compressed} (Step~2).

{\em Step~1.} From the discussion after Definition~\ref{def:dro-prescriptor} we already know that $(\tilde c^\star,\tilde x^\star)\in\tilde{\mc X}$. It remains to be shown that the out-of-sample disappointment of~$(\tilde c^\star,\tilde x^\star)$ decays at a rate of at least~$r$. To this end, fix any~$\theta\in\Theta$, and define the set of all estimator realizations that lead to disappointment for {\em some} decision~$x\in X$ as
\[
    A(\theta)=\left\{s\in\SS : \max_{x\in X}\left\{ c(x,\theta)-\tilde c^\star(x,s)\right\}>0\right\}.
\]
The maximum in this definition is indeed attained because~$X$ is compact and because the model-based predictor~$c(x,\theta)$ and the distributionally robust predictor~$\tilde c^\star(x,s)$ are  continuous in~$x$ thanks to Assumption~\ref{ass:continuity} and Proposition~\ref{prop:continuity_cr}, respectively. Note also that~$A(\theta)=\cup_{x\in X}A(x,\theta)$, where~$A(x,\theta)$ is defined as in Step~1 of the proof of Theorem~\ref{thm:optimality}. The inclusion~\eqref{eq:A-B-inclusion} and the continuity of $I(s,\theta)$ in~$s$ thus ensure that
\begin{equation}\label{eq:feasiblity:pred:step:to:show:pres}
    \cl A(\theta)\subseteq \cl\{s\in\SS : I(s,\theta)\geq r\}= \{s\in\SS : I(s,\theta)\geq r\},
\end{equation}
respectively. 
%It remains to be shown that the out-of-sample disappointment of $\tilde c^\star(\tilde x^\star(\cdot),\cdot)$ decays at a rate of at least $r$. Since $\{\widehat{S}_T\}_{T\in\N}$ satisfies an \ac{ldp} with rate function $I(\cdot,\theta)$ according to Assumption~\ref{ass:ldp}, for any $\theta\in\Theta$
The above reasoning implies that
\begin{align*}
& \limsup_{T\to\infty}\frac 1T \log \mb P_\theta \left[ c(\tilde x^\star(\widehat S_T),\theta) > \tilde c^\star(\tilde x^\star(\widehat S_T),\widehat S_T) \right] \\
&\hspace{3cm}\leq \limsup_{T\to\infty}\frac 1T \log \mb P_\theta \left[ \max_{x\in X} \left\{ c(x,\theta) - \tilde c^\star(x,\widehat S_T) \right\} >0 \right] \leq - \inf_{s\in\cl A(\theta)} ~ I(s,\theta)  \leq -r.
\end{align*}
Indeed, the first inequality holds because $\tilde x^\star(\widehat S_T)\in X$ $\mb P_\theta$-almost surely, the %$c(\tilde x^\star(s), \theta) > \tilde c^\star (\tilde x^\star(s), s)$ implies that there exists $x\in X \text{ with } c(x, \theta) > \tilde c^\star (x, s)$. The 
second follows from the definition of~$A(\theta)$ and the LDP bound~\eqref{eq:ldp_exponential_rates_ub:old} and the third inequality is a direct consequence of~\eqref{eq:feasiblity:pred:step:to:show:pres}. As~$\theta\in\Theta$ was chosen arbitrarily, we may thus conclude that~$(\tilde c^\star,\tilde x^\star)$ is feasible in problem~\eqref{eq:optimal-prescriptor:compressed}.

% To establish the third inequality recall that we have to show that for any $\theta\in\Theta$ and $x\in X$

% To show, that \eqref{eq:feasiblity:pred:step:to:show:pres} indeed holds, by the definition of $\tilde c^\star$, \eqref{eq:gen:dro}, we have $A(\theta)\subseteq B(\theta)$, where $B(\theta):=\{s\in\SS : I(s,\theta)>r\}$. Recall that since $I$ is continuous in the first argument  and since $\SS$ is closed the closure $\cl B(\theta) \subseteq \{s\in\SS: I(s,\theta)\geq r\}$. Using $A(\theta)\subseteq B(\theta)$, this implies $\cl A(\theta)\subseteq \{s\in\SS : I(s,\theta)\geq r\}$ and finally \eqref{eq:feasiblity:pred:step:to:show:pres} indeed holds, which completes the proof.

{\em Step~2.} We now prove that~$(\tilde c^\star, \tilde x^\star)$ Pareto dominates every other feasible solution of problem~\eqref{eq:optimal-prescriptor:compressed}. Assume to the contrary that there exists a compressed data-driven predictor-prescriptor pair $(\tilde c,\tilde x)\in\tilde{\mc X}$ that is feasible in~\eqref{eq:optimal-prescriptor:compressed} but is {\em not} Pareto dominated by~$(\tilde c^\star,\tilde x^\star)$. Hence, there exists a model~$\theta_0\in\Theta$ with 
\begin{equation} \label{feasibility:prescriptor:step1}
\tilde c^\star (\tilde x^\star(S_\infty(\theta_0)),S_\infty(\theta_0)) -\tilde c(\tilde x(S_\infty(\theta_0)),S_\infty(\theta_0))>0.
\end{equation}
As we will see, \eqref{feasibility:prescriptor:step1} implies that~$(\tilde c,\tilde x)$ is infeasible in~\eqref{eq:optimal-prescriptor:compressed}. This contradiction will reveal that our initial assumption must have been false and that there cannot be any feasible~$(\tilde c,\tilde x)$ that Pareto dominates~$(\tilde c^\star,\tilde x^\star)$.

As~$X$ is compact and independent of~$s$ and as~$\tilde c^\star(x,s)$ is continuous on~$X\times\SS_\infty$ by virtue of Proposition~\ref{prop:continuity_cr}, Berge's maximum theorem \cite[pp.~115--116]{berge1997topological}) implies that~$\tilde c^\star (\tilde x^\star(s),s)=\min_{x\in X}\tilde c^\star(x,s)$ is continuous on~$\SS_\infty$. In addition, $S_\infty(\theta)$ is continuous on~$\Theta$ by Definition~\ref{def:estimator}, and thus the combination~$\tilde c^\star (\tilde x^\star(S_\infty(\theta)),S_\infty(\theta))$ is also continuous on~$\Theta$. The exact same arguments can be used to show that~$\tilde c(\tilde x(S_\infty(\theta)),S_\infty(\theta))$ is continuous on~$\Theta$ as well. This implies that the strict inequality~\eqref{feasibility:prescriptor:step1} remains valid under small perturbations of~$\theta_0$. Next, set~$s_0= S_\infty(\theta_0)$, and recall from Definition~\ref{def:estimator} that $S_\infty(\theta)$ is a local homeomorphism and thus locally surjective. %Thus, any sufficiently small perturbation of~$s_0$ can be explained by a small perturbation of~$\theta_0$.
Recall also from Definition~\ref{def:compressed:dd_prediction} that~$\tilde x(s)$ is quasi-continuous on~$\SS_\infty$ and therefore continuous on a dense subset of~$\SS_\infty$ \cite{ref:Blendsoe-52}. By perturbing~$\theta_0$ if necessary, we may thus assume without loss of generality that~$\tilde x(s)$ is continuous at~$s_0=S_\infty(\theta_0)$ while still maintaining the strict inequality~\eqref{feasibility:prescriptor:step1}.

Identifying~$x_0$ with~$\tilde x(s_0)$, we can now reuse the arguments from the proof of Theorem~\ref{thm:optimality} that led to the inequality~\eqref{eq:theta'-feasibility} to show that there exists a  model~$\theta_0^\star\in\Theta$ with~$I(s_0,\theta_0^\star)=r_0<r$ and 
\begin{align}
	\label{eq:theta_0^star-feasibility}
	\tilde c (\tilde x(s_0),s_0) < c(\tilde x(s_0),\theta^\star_0).
\end{align}
Details are omitted to avoid redundancy. 
In the following we will use this inequality to show that the prediction disappointment~$\mb P_{\theta_0^\star} [c(\tilde x(\widehat S_T), \theta^\star_0) > \tilde c(\tilde x(\widehat S_T),\widehat S_T)]$ of the predictor-prescriptor pair~$(\tilde c,\tilde x)$ under model~$\theta_0^\star$ decays no faster than~$e^{r_0T}$ for large sample sizes~$T$. As~$r_0<r$, this will imply that~$(\tilde c,\tilde x)$ is {\em in}feasible in~\eqref{eq:optimal-prescriptor:compressed}. To this end, we define the set of disappointing realizations of~$\widehat S$ as
\[
\mc D(\theta^\star_0) = \left\{ s\in\SS : c(\tilde x(s),\theta^\star_0) > \tilde c (\tilde x(s),s)  \right\}.
\]
By~\eqref{eq:theta_0^star-feasibility}, this disappointment set contains~$s_0$. As~$s_0\in\SS_\infty\subseteq\SS$ and as $\SS_\infty$ is open, we further know that~$s_0$ resides in the interior of~$\SS$. In addition, $\tilde c(x,s)$ is continuous in~$s$ on~$X\times\SS_\infty$ and~$\tilde x(s)$ is continuous at~$s_0\in\SS_\infty$ by the construction of~$s_0$. We may thus conclude that~$s_0$ belongs in fact to the interior of~$\mc D(\theta^\star_0)$ and that
\begin{equation*}
\inf_{s \in \interior \mc D(\theta^\star_0)} I(s, \theta^\star_0)  \leq I(s_0,\theta^\star_0) = r_0,
\end{equation*}
where the equality holds by the definition of $r_0$. Together with the LDP bound~\eqref{eq:ldp_exponential_rates_lb:old}, this implies that
\begin{equation*}
-r < -r_0 \leq - \inf_{s \in \interior \mc D(\theta^\star_0)} I(s, \theta^\star_0) \leq \liminf_{T\to\infty} \frac{1}{T} \log \mb P_{\theta^\star_0}  [ \widehat S_T \in \mc D(\theta^\star_0) ].
\end{equation*}
By the definition of~$\mc D(\theta^\star_0)$, this means that the out-of-sample disappointment of~$(\tilde c,\tilde x)$ corresponding to~$\theta_0^\star$ decays strictly slower than~$e^{rT}$ as~$T$ grows. This in turn contradicts our assumption that~$(\tilde c,\tilde x)$ is feasible in~\eqref{eq:optimal-prescriptor:compressed} and thus implies that~$(\tilde c^\star,\tilde x^\star)$ is a Pareto dominant solution of problem~\eqref{eq:optimal-prescriptor:compressed}.
\end{proof}

%%%%%%%%%%%%%%%%%%%%%%%%%%%%%
\section{Proofs of Section~\ref{sec:equivalence}}

\begin{proof}[Proof of Lemma~\ref{eq:exp:limit}]
  Fix any~$\theta\in\Theta$, and recall that $\mb E_{\theta}[\widehat S_T]=\nabla_\lambda \Lambda_T(0, \theta)= \frac{1}{T} \nabla_\lambda[ \Lambda_T(T\lambda, \theta)]_{\lambda=0}$ for all~$T\in\mb N$. Driving~$T$ to infinity, the claim follows if we can interchange the limit and the gradient on the right hand side of this identity to obtain $\lim_{T\to\infty} \nabla_\lambda[\frac{1}{T}  \Lambda_T(T\lambda, \theta)]_{\lambda=0} = \nabla_\lambda \Lambda(0, \theta)$.
  %By Assumption~\ref{ass:LDP2}\ref{item:ass:suff:ii} and as~$\theta$ is fixed, we know that~$\Lambda(\lambda, \theta)$ is continuous at~$\lambda=0$.
  To this end, select a tolerance~$\varepsilon>0$ and a direction~$b\in \Re^d$. By the definition of the directional derivative of $\Lambda(\lambda,\theta)$ at~$\lambda=0$ along the direction~$b$, there exists a step size~$h>0$ such that
  \(
    \tfrac{(\Lambda(h b,\theta)-\Lambda(0,\theta))}{h}  < \iprod{\nabla_\lambda \Lambda(0, \theta)}{b} + \varepsilon .
  \)
  In addition, by the definition~\eqref{eq:def:Lambda} of the limiting log-moment generating function~$\Lambda$, there exists~$T_0\in\mb N$ such that 
  \begin{align*}
     \tfrac{(\Lambda_T(h b T, \theta)-\Lambda_T(0, \theta))}{(Th)}  < \iprod{\nabla_\lambda \Lambda(0, \theta)}{b}  + \varepsilon\quad \forall T\ge T_0.
  \end{align*}
  Next, the first-order condition of convexity for~$\Lambda_T(T\lambda, \theta)/T$ guarantees that
  \begin{align*}
    \iprod{\nabla_\lambda \Lambda_T(0, \theta)}{b} \leq & \tfrac{(\Lambda_T(h b T, \theta)-\Lambda_T(0, \theta))}{(Th)}\quad \forall T\in\mb N. 
  \end{align*}
  Combining the last two inequalities then yields the estimate $\lim_{T\to\infty} \iprod{\nabla_\lambda\Lambda_T(0, \theta)}{b}<\iprod{\nabla_\lambda \Lambda(0, \theta)}{b}  + \varepsilon$. As~$\varepsilon>0$ was chosen arbitrarily, this implies that $\lim_{T\to\infty}\iprod{\nabla_\lambda\Lambda_T(0, \theta)}{b}\leq \iprod{\nabla_\lambda \Lambda(0, \theta)}{b}$, and as $b\in \Re^d$ was also chosen arbitrarily, we may in fact conclude that~$\lim_{T\to\infty} \nabla_\lambda\Lambda_T(0, \theta) = \nabla_\lambda \Lambda(0, \theta)$.
\end{proof}

\begin{proof}[Proof of Lemma~\ref{lemm:ge-consistent}]
  Fix any~$\theta\in \Theta$ and~$\varepsilon>0$, and set~$s_0=\nabla_\lambda \Lambda(0, \theta)$. %, and use~$S(\theta)$ as a shorthand for~$\nabla_\lambda \Lambda(0, \theta)$.
  By Theorem~\ref{thm:Gaertner-Ellis} we have
  \begin{equation}\label{eq:decay:lemma:cons}
    \limsup_{T\to\infty}\frac 1T \log \mb P_\theta[\tnorm{\widehat S_T-s_0}\geq \varepsilon]\leq -\inf_{s\in\mb R^d} \left\{ I(s, \theta):\tnorm{s-s_0}\geq \varepsilon\right\} = -I(s^\star, \theta)
  \end{equation}
  for some~$s^\star\neq s_0$, where the equality holds because the good rate function~$I(s,\theta)$ has compact sublevel sets. Next, we show that~$I(s^\star, \theta)>0$. %To this end, note that~\eqref{eq:GE:ratefct} constitutes an unconstrained maximization problem with a smooth concave objective function.  %By~\cite[Lemma~2.3.9(b)]{dembo2009large}
  %Hence, $I(s_0,\theta)=\inprod{0}{s_0}-\Lambda(0, \theta)=0$, where the first equality holds because~$\lambda=0$ satisfies the first-order optimality condition of~\eqref{eq:GE:ratefct} for~$s=s_0$. 
%   Next, we show that
%   \begin{equation}\label{eq:claim:lemma:consistency}
%     I(s^\star ,\theta) > 0.%I(s_0, \theta).
%   \end{equation}
%To show that \eqref{eq:claim:lemma:consistency}, which is correct we follow \cite[Lemma~2.3.9(b)]{dembo2009large} and 
Suppose for the sake of argument that~$I(s^\star,\theta) =0$. This implies that
\[
     \inprod{\lambda}{s^\star} = \inprod{\lambda}{s^\star} - I(s^\star, \theta) \leq \Lambda(\lambda,\theta) = \Lambda(\lambda,\theta)-\Lambda(0,\theta) \quad\forall\lambda\in\mb R^d,
\]
where the inequality exploits the definition of~$I(s^\star,\theta)$. Setting~$\lambda=\delta v$ for~$\delta>0$ and~$v\in\mb R^d$, we then find
\begin{equation*}
    \inprod{v}{s^\star}\leq \lim_{\delta\downarrow 0}\frac{1}{\delta}\left( \Lambda(\delta v,\theta)-\Lambda(0,\theta) \right)
    =\inprod{v}{\nabla_{\lambda}\Lambda(0,\theta)} \quad \forall v\in\Re^d,
\end{equation*}
which in turn implies that~$s^\star = \nabla_{\lambda}\Lambda(0,\theta)=s_0$ and thus contradicts the construction of~$s^\star$. We therefore conclude that~$I(x^\star,\theta)>0$, which ensures via~\eqref{eq:decay:lemma:cons} that~$\mb P_\theta[\tnorm{\widehat S_T-s_0}\geq \varepsilon]\leq \exp(-T\cdot I(x^\star,\theta)+o(T))$ tends to~$0$ as~$T$ grows. The claim then follows because~$\theta\in\Theta$ and~$\varepsilon>0$ were chosen arbitrarily.
\end{proof}

\begin{proof}[Proof of Theorem~\ref{thm:predictor:equivalence}]
As~$\tilde c^\star\in\tilde{\mc C}$ is a compressed data-driven predictor in the sense of Definition~\ref{def:compressed:dd_prediction}, $\widehat c^{\,\star}$ constitutes a data-driven predictor in the sense of Definition~\ref{def:dd_prediction}, that is, $\widehat c^{\,\star}\in\widehat{\mc C}$. This follows from the discussion after Definition~\ref{def:compressed:dd_prediction}. Similarly, as~$\tilde c^\star$ satisfies the rate constraint in~\eqref{eq:optimal-predictor:compressed}, one readily verifies that~$\widehat c^{\,\star}$ satisfies the rate constraint in~\eqref{eq:optimal-predictor}. We may thus conclude that~$\widehat c^{\,\star}$ is feasible in~\eqref{eq:optimal-predictor}. In the remainder of the proof we will show that~$\widehat c^{\,\star}$ Pareto dominates every other feasible solution of problem~\eqref{eq:optimal-predictor}.

Assume for the sake of contradiction that there exists a data-driven predictor~$\widehat c$ that is feasible in~\eqref{eq:optimal-predictor} but not dominated by~$\widehat c^{\,\star}$. Hence, there exist a decision~$x_0\in X$ and a model~$\theta_0\in\Theta$ with 
\begin{equation*}
    \lim_{T\to\infty} \mb E_{\theta_0}[\widehat c_T^\star(x_0)] - \lim_{T\to\infty} \mb E_{\theta_0}[\widehat c_T(x_0)] > 0.
\end{equation*}
As we will see, the above inequality implies that~$\widehat c$ is infeasible in~\eqref{eq:optimal-predictor}. This contradiction will reveal that our initial assumption must have been false and that there cannot be any feasible~$\widehat c$ that Pareto dominates~$\widehat c^{\,\star}$.

Since the data-driven predictor~$\widehat c$ must satisfy the conditions~\ref{def:dd_prediction:ii:boundedness} and~\ref{def:dd_prediction:iii:convergence} of Definition~\ref{def:dd_prediction}, we may conclude via Lemma~\ref{lem:aux:prob} that~$\lim_{T\to\infty} \mb E_{\theta_0}[\widehat c_T(x_0)] = c_\infty(x_0,\theta_0)$, where~$c_\infty$ is the Borel-measurable function whose existence is postulated in Definition~\ref{def:dd_prediction}\ref{def:dd_prediction:iii:convergence}. Similarly, since~$\widehat S_T$ converges in probability under~$\mb P_{\theta_0}$ to~$S_\infty(\theta_0)$ and since Proposition~\ref{prop:continuity_cr} ensures that~$\tilde c^\star(x_0, s)$ is bounded and continuous in~$s$ on~$\SS_\infty$, we may invoke the continuous mapping theorem~\cite[Theorem~3.2.4]{durrett_book} and the fact that convergence in probability implies convergence in distribution to conclude that~$\lim_{T\to\infty} \mb E_{\theta_0}[\widehat c^\star_T(x_0)]=\lim_{T\to\infty} \mb E_{\theta_0}[\tilde c^\star(x_0,\widehat S_T)]=\tilde c^\star(x_0,S_\infty(\theta_0))$. In summary, we have thus shown that~$\tilde c^\star(x_0,S_\infty(\theta_0))- c_\infty(x_0,\theta_0)>0$.

Defining~$s_0=S_\infty(\theta_0)$, %and identifying~$c_\infty(x_0,\theta_0)$ with~$\tilde c(x_0,S_\infty(\theta_0))=\tilde c(x_0,s_0)$, 
we may reuse the reasoning at the beginning of Step~2 in the proof of Theorem~\ref{thm:optimality} to show that there exists~$\theta_0^\star\in\Theta$ with~$I(s_0,\theta_0^\star)=r_0<r$ and~$c(x_0,\theta_0^\star)-c_\infty(x_0,\theta_0)>0$. In the following, we select any~$\varepsilon>0$ that is strictly smaller than~$c(x_0,\theta_0^\star)-c_\infty(x_0,\theta_0)$ and any~$\delta>0$. 
Thus, we have
\begin{align}
    \nonumber \mb P_{\theta_0^\star}\left[ \widehat c_T(x_0)<c(x_0,\theta_0^\star) \right] &\geq \mb P_{\theta_0^\star}\left[ \widehat c_T(x_0)<c(x_0,\theta_0^\star) \ \wedge \ \widehat S_T\in\mb B_\delta(s_0) \right] \\ 
    & \geq \mb P_{\theta_0^\star}\left[ \widehat c_T(x_0)<c_\infty(x_0,\theta_0)+\varepsilon \ \wedge \ \widehat S_T\in\mb B_\delta(s_0) \right] \label{eq:proof:optimality:step:middle} \\
    & \geq \mb P_{\theta_0^\star}\left[ |\widehat c_T(x_0)-c_\infty(x_0,\theta_0)|<\varepsilon\ \wedge \ \widehat S_T\in\mb B_\delta(s_0) \right], \nonumber
\end{align}
where~$\mb B_\delta(s_0)$ denotes the Euclidean ball of radius~$\delta$ around~$s_0$. Here, the second inequality holds because~$c_\infty(x_0,\theta_0)+\varepsilon< c(x_0,\theta_0^\star)$ thanks to the choice of~$\varepsilon$. The other two inequalities are elementary. 
% \begin{equation}\label{eq:proof:optimality:step:middle}
% \begin{aligned}
% &\liminf_{T\to\infty} \frac{1}{T} \log \mb P_{\theta_1}\left[ \widehat c_T(x_0)<c(x_0,\theta_1) \right] \\
% &\quad
% \geq \liminf_{T\to\infty} \frac{1}{T} \log  \mb P_{\theta_1} \left[ |\widehat c_T(x_0)-c_\infty(x_0,\theta_0)|<{\varepsilon}/{2} \ \wedge \ \widehat S_T\in\mb B_\delta(S_\infty(\theta_0)) \right].
% \end{aligned}
% \end{equation}
By Assumption~\ref{ass:suff}, the probability measures~$\mb P_{\theta_0}$ and $\mb P_{\theta_0^\star}$ both belong to an exponential family of the form~\eqref{eq:RND:exp:fam} and are therefore equivalent. The chain rule for Radon-Nikodym derivatives thus implies that
\begin{equation}
    \label{eq:radon-nikodym}
    \frac{\d \mb{P}_{\theta_0}^T}{\d \mb P_{\theta_0^\star}^T} = \frac{\d \mb{P}_{\theta_0}^T}{\d \mb P_{\bar \theta}^T} \left( \frac{\d \mb{P}_{\theta_0^\star}^T}{\d \mb P_{\bar \theta}^T} \right)^{-1} =\frac{\exp\left(\textstyle \iprod{Tg(\theta_0)}{\widehat S_T}- A_T(T g(\theta_0))\right)}{\exp\left(\textstyle \iprod{Tg(\theta_0^\star)}{\widehat S_T}- A_T(T g(\theta_0^\star))\right)} =   \exp\left(\iprod{\eta}{T \widehat S_T}- \Lambda_T(T\eta, \theta_0^\star)\right),
\end{equation}
where~$\eta=g(\theta_0)-g(\theta_0^\star)\in\Re^d$ characterizes an exponential tilting between $\mb P_{\theta_0}$ and $\mb P_{\theta_0^\star}$. Here, the second equality follows from~\eqref{eq:RND:exp:fam}, while the third equality exploits the relation~\eqref{eq:log-moment-function}, which ensures that
\[
    A_T(T g(\theta_0))- A_T(T g(\theta_0^\star)) = A_T(T\eta+ g(T \theta_0^\star))- A_T(T g(\theta_0^\star)) = \Lambda_T(T\eta,\theta_0^\star).
\]
In order to simplify the subsequent arguments, we introduce the $\mc F_T$-measurable Bernoulli random variable
\begin{equation*}
    \widehat \zeta_T= \left\{ \begin{array}{ll} 
    1 & \text{if }|\widehat c_T(x_0)-c_\infty(x_0,\theta_0)|<\varepsilon \text{ and } \widehat S_T\in\mb B_\delta(s_0), \\
    0 & \text{otherwise.}
    \end{array}\right.
\end{equation*}
Combining all preparatory results derived so far, we then obtain
\begin{subequations} \label{eq:opt:ctr:diappointment:part2}
\begin{align}
&\nonumber \limsup_{T\to\infty} \frac{1}{T} \log \mb P_{\theta_0^\star}\left[ \widehat c_T(x_0)<c(x_0,\theta_0^\star) \right] \\
&\label{eq:opt:stp:1}\quad \geq \limsup_{T\to\infty} \frac{1}{T} \log  \mb P_{\theta_0^\star }\left[ |\widehat c_T(x_0)-c_\infty(x_0,\theta_0)|<\varepsilon \ \wedge \ \widehat S_T\in\mb B_\delta(s_0) \right] \\
&\label{eq:opt:stp:2}\quad = \limsup_{T\to\infty} \frac{1}{T} \log \mb E_{\theta_0^\star}\left[\widehat \zeta_T \right]\\
&\label{eq:opt:stp:3}\quad = \limsup_{T\to\infty} \frac{1}{T} \log \mb E_{\theta_0^\star} \left[ \widehat \zeta_T \cdot \exp ( - \iprod{\eta}{T \widehat S_T} + \Lambda_T(T\eta, \theta_0^\star)) \cdot \exp ( \iprod{\eta}{T\widehat S_T} - \Lambda_T(T\eta, \theta_0^\star)) \right]\\
&\label{eq:opt:stp:4}\quad = \limsup_{T\to\infty} \frac{1}{T} \log \mb E_{\theta_0} \left[ \widehat \zeta_T \cdot \exp ( - \iprod{\eta}{T \widehat S_T} + \Lambda_T(T\eta, \theta_0^\star)) \right] \\
% &\label{eq:opt:stp:4}\quad = \lim_{T\to\infty} \frac{1}{T} \Lambda_T(T\eta,\theta_0^\star) + \limsup_{T\to\infty} \frac{1}{T}\log \mb E_{\theta_0} \left[\widehat \zeta_T \cdot \exp ( -\iprod{\eta}{T \widehat S_T}) \right]\\
&\label{eq:opt:stp:5}\quad = \Lambda(\eta, {\theta_0^\star})-\inprod{\eta}{s_0} +  \limsup_{T\to\infty}\frac{1}{T} \log \mb E_{\theta_0} \left[\widehat \zeta_T \cdot \exp (  \langle T \eta,s_0-\widehat S_T\rangle) \right],
\end{align}
\end{subequations}
where~\eqref{eq:opt:stp:1} and~\eqref{eq:opt:stp:2} follow from~\eqref{eq:proof:optimality:step:middle} and the definition of~$\widehat \zeta_T$, respectively, \eqref{eq:opt:stp:3} is obtained by multiplying~$\widehat \zeta_T$ by~$1$, and~\eqref{eq:opt:stp:4} follows from~\eqref{eq:radon-nikodym} and the Radon-Nikodym theorem. Equation~\eqref{eq:opt:stp:5}, finally, is obtained by extracting the deterministic factor~$\exp(\Lambda_T(T\eta, {\theta_0^\star})-\langle T \eta,s_0\rangle)$ from the expectation and recalling that~$\lim_{T\to\infty} \frac{1}{T} \Lambda_T(T\eta,\theta_0^\star) = \Lambda(\eta, {\theta_0^\star})$. As~$\widehat \zeta_T=1$ only if~$\|\widehat S_T - s_0\|\leq \delta$, the Cauchy-Schwartz inequality implies that~$\langle \eta,s_0 - \widehat S_T \rangle \geq -\|\eta\| \delta$ whenever~$\widehat \zeta_T=1$. Thus, \eqref{eq:opt:ctr:diappointment:part2} implies that
\begin{align}
  &\limsup_{T\to\infty} \frac{1}{T} \log \mb P_{\theta_0^\star}\left[ \widehat c_T(x_0)<c(x_0,\theta_0^\star) \right] \nonumber \\
  & \quad \geq \Lambda(\eta, {\theta_0^\star})-\inprod{\eta}{s_0} -\|\eta\|\delta+ \limsup_{T\to\infty} \frac{1}{T} \log \mb E_{\theta_0}\left[\widehat\zeta\right] \nonumber\\
  %&\quad =\Lambda(\eta, {\theta_0^\star})-\inprod{\eta}{s_0} -\|\eta\|\delta +  \liminf_{T\to\infty} \frac{1}{T} \log \mb P_{\theta_0} \left[ |\widehat c_T(x_0)-c_\infty(x_0,\theta_0)|< \varepsilon \ \wedge \ \widehat S_T\in\mb B_\delta(s_0) \right] \label{eq:optimality:final:ingred:1}\\
  &\quad \geq \Lambda(\eta, {\theta_0^\star})-\inprod{\eta}{s_0} -\|\eta\|\delta \label{eq:optimality:final:ingred:0}\\
  &\quad \quad \quad +  \limsup_{T\to\infty} \frac{1}{T} \log \bigg(  \mb P_{\theta_0} \left[ |\widehat c_T(x_0)-c_\infty(x_0,\theta_0)|<\varepsilon \right] + \mb P_{\theta_0} \left[ \widehat S_T\in\mb B_\delta(s_0) \right] -1 \bigg) \nonumber \\
  &\quad \geq \Lambda(\eta, {\theta_0^\star})-\inprod{\eta}{s_0} -\|\eta\|\delta, \nonumber
\end{align}
where the second inequality follows from the definition of the random variable~$\widehat \zeta_T$ and the elementary insight that~$\mb P_{\theta_0}[A\cap B] \geq \mb P_{\theta_0}[A] + \mb P_{\theta_0}[B] -1$ for all events~$A,B\in\mc F$. The last inequality holds because
\begin{equation*}
  \lim_{T\to\infty} \mb P_{\theta_0} \left[ |\widehat c_T(x_0)-c_\infty(x_0,\theta_0)|<\varepsilon \right]=1
\end{equation*}
thanks to the definition of a data-driven predictor (see Definition~\ref{def:dd_prediction}\ref{def:dd_prediction:iii:convergence}) and because
\begin{equation*}
  \lim_{T\to\infty} \mb P_{\theta_0}\left[ \widehat S_T\in\mb B_\delta(s_0) \right]=1 \quad \forall \delta>0
\end{equation*}
thanks to the definition of a statistic (see Definition~\ref{def:estimator}) and the definition of~$s_0=S_\infty(\theta_0)$. As~\eqref{eq:optimality:final:ingred:0} holds for every~$\delta>0$, we have effectively shown that
\begin{align*}
\limsup_{T\to\infty} \frac{1}{T} \log \mb P_{\theta_0^\star}\left[ \widehat c_T(x_0)<c(x_0,\theta_0^\star) \right] \geq \Lambda(\eta, {\theta_0^\star})-\inprod{\eta}{s_0} = -I(s_0,\theta_0^\star)= -r_0 > -r,
\end{align*}
which contradicts the feasibility of~$\widehat c_T$ in~\eqref{eq:optimal-predictor}. Here, the first equality holds because~$\eta=g(\theta_0)-g(\theta_0^\star)$ is a maximizer of the unconstrained convex optimization problem on the right hand side of~\eqref{eq:GE:ratefct} at~$s=s_0$ and~$\theta=\theta_0^\star$, which defines the rate function of the G\"artner-Ellis theorem. To see this, note that
\begin{align*}
    \nabla_\lambda\left[\inprod{\lambda}{s_0} - \Lambda(\lambda, {\theta_0^\star}) \right]_{\lambda=\eta}& = s_0 - \nabla_\lambda \left[\Lambda(\lambda,\theta^\star_0)\right]_{\lambda=\eta} \\
    & = s_0 - \lim_{T\to\infty}\mb E_{\theta_0^\star}[\widehat S_T \cdot \exp(\iprod{\eta}{T\widehat S_T}- \Lambda_T(T \eta,\theta_0^\star)) ]\\
    & = s_0- \lim_{T\to\infty}\mb E_{\theta_0}[\widehat S_T] = s_0 - \nabla_\lambda \left[\Lambda(\lambda,\theta_0)\right]_{\lambda=0} =0,
\end{align*}
where the second equality holds due to Remark~\ref{ex:exp:limit-genralized}, the third equality follows from~\eqref{eq:radon-nikodym} and the Radon-Nikodym theorem, the fourth equality exploits~Lemma~\ref{eq:exp:limit}, and the fifth equality follows from~Lemma~\ref{lemm:ge-consistent} and the definition of~$s_0$. In summary, we may conclude that our initial assumption was false and that~$\widehat c^{\,\star}$ indeed Pareto dominates every other feasible solution of problem~\eqref{eq:optimal-predictor}.
% \[
%   \lim_{T\to\infty} \mb E_{\theta_0}[\widehat S_T]  = \nabla_\lambda \Lambda(0, \theta_0) 
%   = \nabla_\lambda \Lambda(\eta, \theta_0^\star).
% \]
% The first equality follows from Lemma \ref{eq:exp:limit}. To establish the second equality denote by $A_T'(\lambda) = \nabla_\lambda A_T(\lambda)$ the gradient of $A_T$. Then, recall Equation \eqref{eq:log-moment-function} and observe $\nabla_\lambda[\Lambda_T(T\lambda, \theta_0)/T]_{\lambda=0}=A'_T(T g(\theta_0)) = A'_T(T g(\theta_1)+T\eta^\star) =\nabla_\lambda [\Lambda_T(T\lambda, \theta_1)/T]_{\lambda=\eta^\star}$ for all $T\in\N$. As the log-moment generating functions $\lambda\mapsto \Lambda_T(T\lambda, \theta_1)/T$ and their limit $\lambda\mapsto\Lambda(\lambda, \theta_1)$ are differentiable at $\eta^\star$ we can replicate the argument in the proof of Lemma~\ref{eq:exp:limit} to establish $\nabla_\lambda \Lambda(0, \theta_0) = \nabla_\lambda \Lambda(\eta^\star, \theta_1)$. From \cite[Lemma~2.3.9]{dembo2009large} it follows that
% \begin{equation}
%   \label{eq:rate:fct:r0}
%   r>r_0= I(s_0, \theta_0^\star) = \iprod{\eta}{s_0}-\Lambda(\eta,\theta_0^\star).
% \end{equation}
\end{proof}

\begin{proof}[Proof of Theorem~\ref{thm:prescriptor:equivalence}]
%Let $\widehat c^{\,\star}_T(x)=\tilde c^\star(x,\widehat S_T)$ and $\widehat x^{\,\star} \in \arg\min_{x\in X} \tilde c^\star(x,\widehat S_T)$ and for all $T\in\mb N$, where $\tilde c^\star$ is given by \eqref{eq:gen:dro}. We aim to show that $(\widehat c^{\,\star},\widehat x^{\,\star})$ is a Pareto dominant solution of \eqref{eq:optimal-prescriptor}.
As~$(\tilde c^\star, \tilde x^\star)\in\tilde{\mc X}$ comprises a compressed data-driven predictor and a compressed data-driven prescriptor in the sense of Definition~\ref{def:compressed:dd_prediction}, $(\widehat c^{\,\star},\widehat x^\star)$ comprises a data-driven predictor in the sense of Definition~\ref{def:dd_prediction}, and a data-driven prescriptor in the sense of Definition~\ref{def:dd_prescription}, that is, $(\widehat c^{\,\star},\widehat x^\star)\in\widehat{\mc X}$. This follows from the discussion after Definition~\ref{def:compressed:dd_prediction}. Similarly, as~$(\tilde c^\star, \tilde x^\star)$ satisfies the rate constraint in~\eqref{eq:optimal-prescriptor:compressed}, one readily verifies that~$(\widehat c^{\,\star},\widehat x^\star)$ satisfies the rate constraint in~\eqref{eq:optimal-prescriptor}. Hence, $(\widehat c^\star, \widehat x^\star)$ is feasible in~\eqref{eq:optimal-prescriptor}. Below we will further show that~$(\widehat c^\star, \widehat x^\star)$ Pareto dominates every other feasible solution of problem~\eqref{eq:optimal-prescriptor}.

Assume for the sake of contradiction that there exists a data-driven predictor-prescriptor pair~$(\widehat c, \widehat x)$ that is feasible in~\eqref{eq:optimal-prescriptor} but not dominated by~$(\widehat c^{\,\star}, \widehat x^\star)$. Hence, there exist a model~$\theta_0\in\Theta$ with 
\begin{equation}
    \label{eq:epsilon-definition}
    \lim_{T\to\infty} \mb E_{\theta_0}[\widehat c^\star_T(\widehat x^\star_T )] - \lim_{T\to\infty} \mb E_{\theta_0}[\widehat c_T(\widehat x_T)] > 0.
\end{equation}
In the following we will show that this inequality contradicts our assumption that~$(\widehat c, \widehat x)$ is feasible in~\eqref{eq:optimal-predictor}.

By the defining properties of data-driven predictors and prescriptors, $\widehat c_T(\widehat x_T)=\min_{x\in X} \widehat c_T(x)$ converges in probability under~$\mb P_{\theta_0}$ to~$v_\infty(\theta_0)$, where~$v_\infty$ is the Borel-measurable function whose existence is postulated in Definition~\ref{def:dd_prediction}\ref{def:dd_pres:iii:convergence}. As there exists a random variable~$\overline c$ with~$\mb E_{\theta_0}[\overline c]<\infty$ and~$|\widehat c_T(\widehat x_T)|\leq \overline c$ $\mb P_{\theta_0}$-almost surely for all~$T\in\mb N$ (see Definition~\ref{def:dd_prediction}\ref{def:dd_prediction:ii:boundedness}), Lemma~\ref{lem:aux:prob} implies that~$\lim_{T\to\infty} \mb E_{\theta_0}[\widehat c_T(\widehat x_T)] = v_\infty(\theta_0)$. Similarly, since~$\widehat S_T$ converges in probability under~$\mb P_{\theta_0}$ to~$S_\infty(\theta_0)$ and since~$\tilde c^\star(x, s)$ is bounded and continuous in~$(x,s)$ on~$\SS_\infty$ thanks to Proposition~\ref{prop:continuity_cr}, the continuous mapping theorem~\cite[Theorem~3.2.4]{durrett_book} implies that~$\widehat c^\star_T(\widehat x^\star_T)= \tilde c^\star(\tilde x^\star(\widehat S_T),\widehat S_T)$ converges in probability under~$\mb P_{\theta_0}$ to~$\tilde c^\star(\tilde x^\star(S_\infty(\theta_0)), S_\infty(\theta_0))$. This in turn ensures via Lemma~\ref{lem:aux:prob} that~$\lim_{T\to\infty} \mb E_{\theta_0}[\widehat c^\star_T(\widehat x^\star_T)]=\tilde c^\star(\tilde x^\star(S_\infty(\theta_0)), S_\infty(\theta_0))$.

We now introduce the optimal value function~$\tilde v^\star:\SS\rightarrow \mb R$ through~$\tilde v^\star(s)=\tilde c^\star(\tilde x^\star(s),s)=\min_{x\in X} \tilde c^\star(x,s)$. Note that~$\tilde v^\star$ inherits boundedness from~$\tilde c^\star$ and is continuous in $s\in\SS$ by Berge's maximum theorem \cite[pp.~115--116]{berge1997topological}. The above arguments show that~\eqref{eq:epsilon-definition} is equivalent to~$\varepsilon=\frac{1}{3}[\tilde v^\star(S_\infty(\theta_0))- v_\infty(\theta_0)]>0$.

By Lemma~\ref{lem:auxiliary:lem1}, which applies because all realizations of the random variable~$\widehat x_T$ fall into the compact set~$X$ for all~$T\in\mb N$, there exists a deterministic function~$x_\infty:\Theta\to X$ that satisfies
\begin{equation}\label{proof:prescr:optimality}
\limsup_{T\to\infty} \mb P_\theta \left[ \|\widehat x_T - x_\infty(\theta)\| <\rho \right]>0 \quad \forall \rho>0 \quad \forall \theta\in \Theta.
\end{equation}
Defining~$x_0=x_\infty(\theta_0)$ and~$s_0=S_\infty(\theta_0)$, we may reuse the reasoning at the beginning of Step~2 in the proof of Theorem~\ref{thm:optimality} to show that there exists~$\theta_0^\star\in\Theta$ with~$I(s_0,\theta_0^\star)=r_0<r$ and 
\begin{equation*}
	\tilde c^\star(x_0,s_0)
    < c(x_0, \theta_0^\star) +\varepsilon.
\end{equation*}
As~$x_0\in X$ by construction, we further have
\begin{equation*}
\tilde c^\star(x_0,s_0)
\geq  \min_{x\in X} \tilde c^\star(x,s_0) =  \tilde v^\star(s_0)= v_\infty(\theta_0) + 3\varepsilon,
\end{equation*}
where the two equalities follow from the definitions of~$\tilde v^\star$ and~$\varepsilon$, respectively. Combining the two inequalities above then yields~$c(x_0, \theta_0^\star) > v_\infty(\theta_0) + 2\varepsilon$. As~$c(x,\theta_0^\star)$ is continuous in~$x\in X$ thanks to Assumption~\ref{ass:continuity}, we may finally conclude that there exists a tolerance~$\rho>0$ such that
\begin{equation} \label{eq:proof:prescr:optimality:step2}
c(x,\theta_0^\star) > v_\infty(\theta_0)+\varepsilon \quad \forall x\in X:\; \|x-x_0\|<\rho.
\end{equation}
Armed with these preliminary results, we are now ready to show that~$(\widehat c,\widehat x)$ fails to be feasible in~\eqref{eq:optimal-predictor}. %Observe that 
% \begin{equation*}
% \mb P_{\theta_0^\star}[\widehat c_T(\widehat x_T) < c(\widehat x_T, \theta_0^\star)] 
% \geq \mb P_{\theta_0^\star}[ \widehat c_T(\widehat x_T) < c(\widehat x_T, \theta_0^\star) \ \wedge \ \|\widehat x_T - x_\infty(\theta_0)\| <\delta' \ \wedge \ \widehat S_T\in \mb B_{\delta}(s_0)],
% \end{equation*}
% where $\theta_0^\star\in\Theta$ is defined as above. Here, $\delta'>0$ is chosen sufficiently small to be able for \eqref{eq:proof:prescr:optimality:step2} to apply. Furthermore, $\delta>0$ is left arbitrarily. Hence, following the steps in the proof of the predictor optimality 
To this end, we may use a similar reasoning as in the proof of Theorem~\ref{thm:predictor:equivalence} to demonstrate that
\begin{equation}
\label{eq:proof:prescr:optimality:step3}
\begin{aligned}
&\limsup_{T\to\infty} \frac{1}{T} \log \mb P_{\theta_0^\star}[\widehat c_T(\widehat x_T) < c(\widehat x_T, \theta_0^\star)]\\
&\quad \geq  \limsup_{T\to\infty} \frac{1}{T} \log \mb P_{\theta_0^\star}\left[ \widehat c_T(\widehat x_T) < v_\infty(\theta_0)+\varepsilon  \ \wedge \ \widehat S_T\in \mb B_{\delta}(s_0) \ \wedge \ \|\widehat x_T - x_0\| <\rho\right] \\
& \quad \geq  \Lambda(\eta, \theta_0^\star)-\iprod{\eta}{s_0} -\|\eta\| \delta      \\
& \qquad \quad + \limsup_{T\to\infty} \frac{1}{T} \log  \mb P_{\theta_0}\left[ \widehat c_T(\widehat x_T) < v_\infty(\theta_0) +\varepsilon 
     \ \wedge \ \widehat S_T\in \mb B_{\delta}(s_0) \ \wedge \ \|\widehat x_T - x_0\| <\rho \right],
\end{aligned}
\end{equation}
where the two inequalities follow from~\eqref{eq:proof:prescr:optimality:step2} and from a change of measure argument akin to~\eqref{eq:opt:ctr:diappointment:part2}--\eqref{eq:optimality:final:ingred:0}. Details are omitted for brevity of exposition. As~$\widehat c_T(\widehat x_T)$ converges in probability to~$v_\infty(\theta_0)$ under~$\mb P_{\theta_0}$ (by the definition of a data-driven predictor) and as~$\widehat S_T$ converges in probability to~$S_\infty(\theta_0)$ under~$\mb P_{\theta_0}$ (by the definition of a statistic and by the construction of~$s_0=S_\infty(\theta_0)$), one readily verifies that
\begin{equation*} 
\lim_{T\to\infty} \mb P_{\theta_0} \left[\widehat c_T(\widehat x_T) < v_\infty(\theta_0) \ \wedge \ \widehat S_T\in \mb B_{\delta}(s_0)\right]=1.
\end{equation*}
As~$x_0=x_\infty(\theta_0)$, \eqref{proof:prescr:optimality} further implies that~$\limsup_{T\to\infty} \mb P_{\theta_0}[ \|\widehat x_T - x_0\| <\rho]>0$ for all~$\delta>0$. Thus, we have 
\begin{align*}
&\limsup_{T\to\infty}\frac{1}{T}\log \mb P_{\theta_0} \left[ \widehat c_T(\widehat x_T) < v_\infty(\theta_0)+\varepsilon \ \wedge \ \widehat S_T\in \mb B_{\delta}(s_0) \ \wedge \ \|\widehat x_T - x_0\| <\rho \right] \\
& \qquad \geq \limsup_{T\to\infty} \frac{1}{T} \log \bigg(  \mb P_{\theta_0} \left[\widehat c_T(\widehat x_T) < v_\infty(\theta_0) \ \wedge \ \widehat S_T\in \mb B_{\delta}(s_0)\right] + \mb P_{\theta_0}[ \|\widehat x_T - x_0\| <\rho] -1 \bigg)=0,
\end{align*}
where the inequality follows using the elementary insight that~$\mb P_{\theta_0}[A\cap B] \geq \mb P_{\theta_0}[A] + \mb P_{\theta_0}[B] -1$ for all events~$A,B\in\mc F$. Combining this estimate with~\eqref{eq:proof:prescr:optimality:step3} then yields
\begin{align*}
&\limsup_{T\to\infty}\frac{1}{T}\log \mb P_{\theta_0^\star}[\widehat c_T(\widehat x_T) < c(\widehat x_T, \theta_0^\star)]
\geq \Lambda(\eta, \theta_0^\star) - \inprod{\eta}{s_0}-\|\eta\| \delta.
\end{align*}
As the above inequality holds for every~$\delta>0$, we have effectively shown that
\begin{align*}
&\limsup_{T\to\infty}\frac{1}{T}\log \mb P_{\theta_0^\star}[\widehat c_T(\widehat x_T) < c(\widehat x_T, \theta_0^\star)]
\geq \Lambda(\eta, \theta_0^\star) - \inprod{\eta}{s_0} 
=-I(s_0,\theta_0^\star) = -r_0 >-r,
\end{align*}
where the two equalities follow from the relation~$I(s_0, \theta_0^\star)=\iprod{\eta}{s_0}-\Lambda(\eta, \theta_0^\star)$ established in the proof of Theorem~\ref{thm:predictor:equivalence} and the definition of~$r_0$, respectively. This contradicts the feasibility of $(\widehat c, \widehat x)$ in~\eqref{eq:optimal-prescriptor}.
\end{proof}

%%%%%%%%%%%%%%%%%%%%%%%%%%%%%
\section{Proofs of Section~\ref{sec:models}}

\begin{proof}[Proof of Proposition~\ref{prop:properties:cond:rel:entropy}]
%We know from~\cite[p.~79]{dembo2009large} that~$D_c(s\|\theta)$ is a good rate function. In the following we will show that~$D_c(s\|\theta)$ is also regular in the sense of Definition~\ref{def:rate_function}. 
The continuity of~$D_c(s\|\theta)$ on~$\SS \times \Theta$ follows directly from Definition~\ref{def:conditional_relative_entropy} and our standard conventions for the logarithm. Next, we show that the level sets of the form~$\{(s,\theta)\in \SS\times\cl \Theta:\Dc{s}{\theta}\leq r\}$ are compact for all fixed thresholds~$r\geq 0$. It is clear that all these level sets are bounded because both~$\SS$ and~$\cl\Theta$ are bounded. In addition, they are closed because~$\Dc{s}{\theta}$ is lower semi-continuous on~$\SS\times\cl\Theta$ by construction. 
To prove that the conditional relative entropy~$\Dc{s}{\theta}$ is radially monotonic in $\theta$, %\ref{def:regular:rate:i:radial:monotonicity} , 
we first observe that the following equivalent inequalities hold for all vectors~$v,w\in\mb R^d_{++}$ thanks to Jensen's inequality.
\begin{align}
	\label{eq:jensen}
	\sum_{j\in\Xi} \frac{w_j}{\sum_{k\in\Xi} w_k} \left(\frac{v_j}{w_j}\right)^2\geq \left(\sum_{j\in\Xi} \frac{w_j}{\sum_{k\in\Xi} w_k} \frac{v_j}{w_j}\right)^2  \quad \iff\quad \sum_{j\in\Xi}  \frac{v_j^2}{w_j} 
	\geq \frac{\left(\sum_{j\in\Xi} v_j\right)^2}{\sum_{k\in\Xi} w_k} 
\end{align}
Note further that these inequalities are strict unless~$v$ and~$w$ are parallel, in which case the fraction~$v_j/w_j$ is constant in~$j$. Next, select any~$\theta\in\cl\Theta$ and~$s\in\SS_\infty=\Theta$, and define $\theta(\lambda)= (1-\lambda)s+\lambda\theta$ for any $\lambda\in[0,1)$. By construction, $\theta(\lambda)\in\Theta$ for all~$\lambda\in [0,1)$. Basic algebra further implies that
\begin{align*}
	\frac{\d}{\d \lambda} \Dc{s}{\theta(\lambda)} ~&= \sum_{i,j\in\Xi} s_{ij} \left( \frac{\sum_{k\in\Xi} \theta_{ik}- s_{ik}}{\sum_{k\in\Xi} \theta_{ik}(\lambda)}  - \frac{\theta_{ij}-s_{ij}}{\theta_{ij}(\lambda)}\right) \\
	&= \frac{1}{\lambda}\sum_{i,j\in\Xi} s_{ij} \left( \frac{\sum_{k\in\Xi} \theta_{ik}(\lambda) - s_{ik}}{\sum_{k\in\Xi} \theta_{ik}(\lambda)}  - \frac{\theta_{ij}(\lambda)-s_{ij}}{\theta_{ij}(\lambda)}\right) \\
	&= \frac{1}{\lambda}\sum_{i\in\Xi} \left( \sum_{j\in\Xi} \frac{\left(s_{ij}\right)^2}{\theta_{ij}(\lambda)}  - \frac{\left(\sum_{j\in\Xi} s_{ij}\right)^2}{\sum_{j\in\Xi}\theta_{ij}(\lambda)}\right) \geq 0 \quad \forall \lambda\in(0,1),
\end{align*}
where the inequality follows from~\eqref{eq:jensen}. Note that this inequality collapses to an equality if and only if each row of $\theta$ is parallel to the corresponding row of $s$, that is, if and only if the transition probability matrices~$P_{\theta}$ and~$P_{s}$ induced by~$\theta$ and~$s$, respectively, are identical. In this special case we have~$\Dc{s}{\theta(\lambda)}=0$ for all~$\lambda\in[0,1)$. Otherwise, $\Dc{s}{\theta(\lambda)}$ is strictly monotonically increasing in $\lambda$. Thus, $\Dc{s}{\theta}$ satisfies the inequality~\eqref{eq:suff:radial:mon}, which is sufficient for radial monotonicity; see the discussion after Definition~\ref{def:rate_function}. In summary, we have thus shown that the rate function~$\Dc{s}{\theta}$ is indeed regular. Finally, $\Dc{s}{\theta}$ is convex in~$s$ because the perspective function $v\log(v/w)$ is convex in~$(v,w)\in\mb R_+^2$ \cite[Section~3.2.6]{boyd2004convex} and because convexity is preserved under combinations with linear functions \cite[Section~3.2.2]{boyd2004convex}.
\end{proof}

\begin{proof}[Proof of Proposition~\ref{thm:AR:multi-dim:modelclass:2}]
Fix any~$\theta\in\Theta$, and define a probability measure~$\mb P'_\theta$ on~$(\Omega,\mc F)$ under which the observations~$\{\xi_t\}_{t\in\N}$ are jointly normally distributed with~$\mb E'_\theta[\xi_t] = 0$ and~$\mb E'_\theta[\xi_t \xi_{s}^\top] = R_\delta$ for~$\delta=t-s$, where~$\mb E'_\theta[\cdot]$ denotes the expectation under~$\mb P'_\theta$. Note that~$\mb P'_\theta$ and~$\mb P_\theta$ assign different means to~$\xi_t$ but are otherwise indistinguishable. 
%onsider first the stochastic process~$\{\xi'_t\}_{t\in\N}$ defined through~$\xi'_t=\xi_t-(\mathbb{1}_d-\A)^{-1}\theta$ for all~$t\in\N$. By construction, we have~$\mb E_\theta[\xi'_t] = 0$ and~$\mb E_\theta[\xi_t' (\xi_{s}')^\top] = R_\delta$ for any~$\xi'_t$ and~$\xi'_s$ with~$\delta=t-s$.
The log-moment generating function of the sample mean~$\widehat \mu_T=\frac{1}{T} \sum_{t=1}^T \xi_t$ under~$\mb P'_\theta$ is then given by
\begin{align*}
    \Lambda'_T(\lambda,\theta) &= \log\mb E'_{\theta}\left[\exp(\lambda\tpose \widehat \mu_T)\right] = \frac{1}{2} \lambda\tpose \mb E'_\theta \left[\widehat \mu_T \widehat \mu_T\tpose \right]  \lambda = \frac{1}{2T^2} \sum_{s,t=1}^T \lambda\tpose R_{t-s}  \lambda=\frac{1}{2T} \sum_{\delta=-T}^T \left(1-\frac{|\delta|}{T}\right) \lambda\tpose R_\delta\lambda,
\end{align*}
where the second equality follows from the formula for the mean value of a lognormal random variable, while the third equality exploits the definitions of the sample mean~$\widehat \mu_T$ and the cross-covariance matrix of the observations~$\xi_t$ and~$\xi_s$.
This implies that the limiting log-moment generating function is representable as
\begin{align*}
    \Lambda'(\lambda,\theta) & = \lim_{T\to\infty} \frac{1}{T} \Lambda'_T( T \lambda,\theta) 
    = \lim_{T\to\infty} \frac{1}{2} \sum_{\delta=-T}^T \left(1-\frac{|\delta|}{T}\right)\lambda\tpose R_\delta \lambda \\ % = \frac{1}{2}\lambda\tpose Q \lambda,
    & = \frac{1}{2}\lambda\tpose \left( R_0+ \lim_{T\to\infty} \sum_{\delta=1}^T\left(1 - \frac{\delta}{T}\right) A^\delta R_0 + \lim_{T\to\infty} \sum_{\delta=1}^T \left(1 - \frac{\delta}{T}\right)  R_0 (A^\delta)\tpose \right)\lambda\\
    & = \frac{1}{2}\lambda\tpose \left( (\mathbb{1}_d -\A)^{-1}R_0 + R_0 (\mathbb{1}_d -\A\tpose)^{-1} - R_0\right) \lambda,
\end{align*}
where the second equality holds because~$R_{\delta} = A^{\delta} R_0$ for~$\delta> 0$ and~$R_{\delta} =  R_0(A^{-\delta})\tpose$ for~$\delta < 0$, while the third equality follows from the asymptotic stability of~$A$ and the geometric series formulas
\[
    \lim_{T\to\infty}\sum_{\delta=1}^T A^\delta=(\mb 1_d-A)^{-1}-\mb 1_d\qquad \text{and} \qquad \lim_{T\to\infty}\sum_{\delta=1}^T \frac{\delta}{T} A^\delta = 0.
\]
We have thus demonstrated that~$\Lambda'(\lambda,\theta)= \frac{1}{2}\lambda\tpose Q\lambda$ is independent of~$\theta$ and quadratic in~$\lambda$ with Hessian matrix~$Q=(\mathbb{1}_d -\A)^{-1}R_0 + R_0 (\mathbb{1}_d -\A\tpose)^{-1} - R_0$. By the definitions of~$Q$ and~$R_0$ we have
\begin{equation}
    \label{eq:Q-vs-Sigma}
    (\mathbb{1}_d -\A)Q(\mathbb{1}_d -\A\tpose) = R_0 - \A R_0 \A\tpose = \Sigma.
\end{equation}
As~$A$ is asymptotically stable, the matrix~$\mathbb{1}_m -\A$ is invertible. The above relation thus implies that~$Q$ inherits positive definiteness from~$\Sigma$. Consequently, the quadratic function~$\Lambda'(\lambda,\theta)$ is strictly convex in~$\lambda$ and is easily seen to satisfy Assumption~\ref{ass:LDP2}. The G\"artner-Ellis Theorem thus ensures that the sample mean satisfies an \ac{ldp} under~$\mb P'_\theta$ with good rate function~$I_{\widehat\mu}'(\mu,\theta)= \sup_{\lambda\in \Re^d} \; \iprod{\lambda}{\mu}-\Lambda'(\lambda, \theta) = \frac{1}{2}\mu\tpose Q^{-1}\mu$; see Theorem~\ref{thm:Gaertner-Ellis}. Note that under~$\mb P_\theta$ the expected value of the sample mean~$\widehat \mu_T$ is given by~$(\mb 1_d-A)^{-1} \theta$ instead of~$0$, but the distribution of~$\widehat \mu_T$ has the same shape under~$\mb P_\theta$ and~$\mb P'_\theta$. Therefore, $\widehat\mu_T$ also satisfies an \ac{ldp} under~$\mb P_\theta$ with good rate function~$I_{\widehat\mu}(\mu,\theta)= \frac{1}{2}(\mu-(\mb 1_d-A)^{-1} \theta)\tpose Q^{-1}(\mu-(\mb 1_d-A)^{-1} \theta)$.
% We next show that the matrix $Q$ can be expressed as
% \begin{equation}\label{eq:matrixQ}
% Q = (\mathbb{1}_m -\A)^{-1} V + V (\mathbb{1}_m -\A\tpose)^{-1}-V,
% \end{equation}
% where $V$ is the solution to the discrete Lyapunov equation $V=\A V \A\tpose + \Sigma$. Since $A$ is asymptotically stable, $V$ exist, is unique and positive definite \cite[p.~267]{ref:Vidyasagar-02}.
% Denote $V = \mathbb{E}_\theta[\zeta_t \zeta_t\tpose]$, then $V=\mathbb{E}_\theta[\zeta_t \zeta_t\tpose] = \A \mathbb{E}_\theta[\zeta_{t-1} \zeta_{t-1}\tpose]\A\tpose  + \mathbb{E}[\varepsilon_t \varepsilon_t\tpose] = \A V \A\tpose + \Sigma$. A simple induction argument shows that
% \begin{equation} \label{eq:def:R}
% R_t = \left\{ \begin{array}{ll} \A^t V, & t\geq 0 \\ V(\A\tpose)^{-t}, & t<0. \end{array} \right.
% \end{equation}
% Therefore, 
% \begin{align*}
% Q &= \lim_{T\to\infty} \sum_{t=1}^T R_t \left(1 - \frac{t}{T}\right) + \lim_{T\to\infty} \sum_{t=1}^T R_{-t} \left(1 - \frac{t}{T}\right) + R_0 \\
%  &= \lim_{T\to\infty} \sum_{t=1}^T \left( \A^t V (1 - \frac{t}{T}) + V(\A\tpose)^{t} (1 - \frac{t}{T}) \right) + V \\
%  & = (\mathbb{1}_m -\A)^{-1}V - V + V (\mathbb{1}_m -\A\tpose)^{-1} - V + V \\
%  & = (\mathbb{1}_m -\A)^{-1}V + V (\mathbb{1}_m -\A\tpose)^{-1} - V,
% \end{align*}
% where the first equality is the definition of the variable $Q$, the second equality invokes \eqref{eq:def:R} and the third equality uses the property 
% $\lim_{T\to\infty}\frac{1}{T}\sum_{t=1}^T t A^t =0$ and the fact that $\lim_{T\to\infty} \sum_{t=1}^T A^t = (\mathbb{1}_m -\A)^{-1}$.
As~$\widehat S_T = (\mb 1_d-A)\widehat \mu_T$ and as~$A$ is asymptotically stable, the contraction principle \cite[Theorem~4.2.1]{dembo2009large} further implies that the scaled sample mean~$\widehat S_T$ defined in~\eqref{estimator:SAA:scaled} satisfies an \ac{ldp} with good rate function
\[
    I(s,\theta)=I_{\widehat \mu}((\mb 1_d-A)^{-1} s,\theta) = \frac{1}{2}(s-\theta)\tpose (\mathbb{1}_m -\A\tpose)^{-1} Q^{-1} (\mathbb{1}_m -\A)^{-1}(s-\theta) = \frac{1}{2}(s-\theta)\tpose \Sigma^{-1}(s-\theta),
\]
where the last equality follows from~\eqref{eq:Q-vs-Sigma}. The regularity of the rate function~$I(s,\theta)$ is easy to check.
\end{proof}

\begin{proof}[Proof of Proposition~\ref{thm:AR:multi-dim:modelclass:3}]
By \cite[Proposition~8]{bercu1997quadraticLDP}, the least squares estimator satisfies an \ac{ldp} with rate function~\eqref{ratefunction:least-squares}. To see that this rate function is regular, note first that~$a(\theta)>-1$ and~$b(\theta)<+1$ for all~$\theta\in \Theta$. This implies that~$(1-2\theta s+\theta^2)/(1-s^2)\ge 1$ whenever~$s\in[a(\theta),b(\theta)]$. Note also that~$|\theta-2s|>1$ whenever~$s\notin[a(\theta),b(\theta)]$. Hence, the rate function~\eqref{ratefunction:least-squares} is indeed non-negative on~$\SS\times\cl\Theta$. In addition, one readily verifies that the two pieces of~$I(s,\theta)$ described in the first and the second line of~\eqref{ratefunction:least-squares}, respectively, match whenever $s = a(\theta)$ or~$s=b(\theta)$, and thus~$I(s,\theta)$ is continuous on~$\SS\times \cl \Theta$. This in turn implies that all sublevel sets of~$I(x,\theta)$ are closed. As~$\cl\Theta=[-1,1]$ is bounded and as~$\log(|\theta-2s|)$ is bounded below by the coercive function~$\log(|2s|-1)$ uniformly across all~$s\notin [a(\theta,b(\theta)]$ and $\theta\in\cl \Theta$, the sublevel sets of~$I(x,\theta)$ are also bounded. Thus, $I(x,\theta)$ satisfies the level-compactness condition of Definition~\ref{def:rate_function}. To prove radial monotonicity, fix any~$s\in \SS_\infty = \Theta$ and any~$\theta\in\cl\Theta$, and define $\theta(\lambda)=(1-\lambda)s + \lambda \theta$ for all~$\lambda\in[0,1)$. If~$s=\theta$, then~$I(s,\theta(\lambda))=0$ for all~$\lambda\in[0,1)$, and thus~\eqref{eq:suff:radial:mon} is satisfied. If~$s\neq\theta$, on the other hand, then~$I(s,\theta(\lambda))$ is strictly monotonically increasing in $\lambda$, which implies that~\eqref{eq:suff:radial:mon} holds as a strict inequality. Indeed, if~$s\in[a(\theta),b(\theta)]$, then we have
\[
    \frac{\d}{\d \lambda}I(s,\theta(\lambda)) = \frac{(\theta(\lambda)-s)(\theta - s)}{1-2\theta(\lambda)s +\theta(\lambda)^2}>0 \quad\forall\lambda\in(0,1).
\]
Note that the denominator of the above fraction is strictly positive because~$s,\theta\in(-1,1)$ for~$s\in[a(\theta),b(\theta)]$. The numerator is also strictly positive because~$\theta(\lambda)-s$ and~$\theta-s$ must have the same sign and because~$s\neq \theta$. If~$s<a(\theta)$, on the other hand, then we have~$\theta-2s>1$, which implies that~$s<0$ because~$\theta<1$. In addition, as~$a(\theta)\leq \theta$, the assumption~$s<a(\theta)$ also ensures that~$s<\theta$. Using~$s<0$ and~$\theta-2s>1$, we thus find
\[
    \theta(\lambda)-2s\geq \min\{s-2s,\theta-2s\}>\min\{1,0\}=0\quad\forall \lambda\in (0,1).
\]
These observations imply that~$\frac{\d}{\d \lambda}I(s,\theta(\lambda)) = (\theta-s)/(\theta(\lambda)-2s)>0$ for all~$\lambda\in (0,1)$. If~$s>b(\theta)$, finally, then we have~$2s-\theta>1$, which implies that~$s>0$ because~$\theta>-1$. In addition, as~$\theta\leq b(\theta)$, the assumption~$s>b(\theta)$ also ensures that~$s>\theta$. Using~$s>0$ and~$2s-\theta>1$, we thus find
\[
    2s-\theta(\lambda)\geq \min\{2s-s,2s-\theta\}>\min\{0,1\}=0\quad\forall \lambda\in (0,1).
\]
These observations imply that~$\frac{\d}{\d \lambda}I(s,\theta(\lambda)) = (s-\theta)/(2s-\theta(\lambda))>0$ for all~$\lambda\in (0,1)$. Thus, $I(s,\theta(\lambda))$ is indeed strictly increasing in~$\lambda$ whenever~$s\neq \theta$. In summary, $I(s,\theta)$ satisfies the inequality~\eqref{eq:suff:radial:mon}, which is sufficient for radial monotonicity; see the discussion after Definition~\ref{def:rate_function}. Thus, the claim follows.
\end{proof}

\begin{proof}[Proof of Proposition~\ref{thm:AR:multi-dim:modelclass:4}]
By~\cite[Proposition~8]{bercu1997quadraticLDP}, the Yule-Walker estimator satisfies an \ac{ldp} with rate function~\eqref{ratefunction:Yule:Walker}. By construction, this rate function is non-negative and continuous throughout~$\SS\times\cl\Theta$ except at the two points~$(1,1)$ and~$(-1,-1)$, where the function is only lower semi-continuous. All sublevel sets of~$I(x,\theta)$ are closed (as the function is lower semicontinuous) and bounded (as the function evaluates to~$\infty$ outside of the bounded set~$[-1,1]^2$), and thus~$I(s,\theta)$ satisfies the level-compactness condition of Definition~\ref{def:rate_function}. Radial monotonicity can be proved as in Proposition~\ref{thm:AR:multi-dim:modelclass:3} with obvious minor modifications. Details are omitted for brevity. 
% Moreover, we will show that the rate function is regular. First, the function $I(s,\theta)$ is continuous on $\SS\times\cl\Theta$. Therefore its level sets are closed and hence the level-compactness assumption for $I$ holds. 
% To prove radial monotonicity in $\theta$, we fix $s \in\SS, \theta\in\Theta$ and define $\theta(\lambda)=(1-\lambda)s + \lambda \theta$ for any $\lambda\in(0,1)$. We show that $I(s,\theta(\lambda))$ is strictly monotonically increasing in $\lambda$, which implies the desired radial monotonicity. The rate function~\eqref{ratefunction:Yule:Walker} satisfies
% \[
% \frac{\d}{\d \lambda}I(s,\theta(\lambda)) = \frac{(\theta - s)(\theta(\lambda)-s)}{(\theta(\lambda)-s) + 1 - s^2},
% \]
% whose denominator is strictly positive. It therefore remains to show that $(\theta - s)(\theta(\lambda)-s)>0$. Consider the case where $\theta>s$, then $\theta(\lambda)=(1-\lambda)s + \lambda \theta >(1-\lambda)s + \lambda s = s$ and hence $(\theta - s)(\theta(\lambda)-s)>0$. Similarly, if $\theta<s$, then $\theta(\lambda)=(1-\lambda)s + \lambda \theta<(1-\lambda)s + \lambda s = s$ and hence $(\theta - s)(\theta(\lambda)-s)>0$. Note that in case $\theta = s$ the rate function vanishes and therefore $I(s,\theta)$ satisfies \eqref{eq:suff:radial:mon} and is indeed radially monotonic in $\theta$. 
\end{proof}

%%%%%%%%%%%%%%%%%%%%%%%%%%%%%
\section{Auxiliary probabilistic results}\label{app:auxiliary}

Throughout this section we assume that all random objects are defined on a probability space~$(\Omega,\mc F,\mb P)$ with a fixed probability measure~$\mb P$, and we denote the expectation operator with respect to~$\mb P$ by~$\mb E[\cdot]$.

%%%%%
\begin{lemma} \label{lem:aux:prob}
If the real-valued random variables~$z_T$, $T\in\N$, converge in probability to a random variable~$z_\infty$ and if there exists a random variable~$\bar z$ with~$|z_T|\leq \bar z$ for all~$T\in\N$ and~$\mb E[\bar z]<\infty$, then $\lim_{T\to\infty}\mb E[z_T]=\mb E[z_\infty]$.
\end{lemma}
\begin{proof}
If~$\{z_{T(m)}\}_{m\in\N}$ is a subsequence of~$\{z_T\}_{T\in\N}$, then~\cite[Theorem~2.3.2]{durrett_book} ensures that there exists a further subsequence~$\{z_{T(m_{k})}\}_{k\in\N}$ that converges almost surely to~$z_\infty$. By the dominated convergence theorem, we may thus conclude that~$\lim_{k\to\infty} \mb E[z_{T(m_{k})}] = \mb E[z_\infty]$. As the subsequence~$\{z_{T(m)}\}_{m\in\N}$ was chosen arbitrarily, this finally implies via~\cite[Theorem~2.3.3]{durrett_book} that $\lim_{T\to\infty}\mb E[z_T] = \mb E[z_\infty]$.
\end{proof}

\begin{lemma}\label{lem:auxiliary:lem1}
If the random variables~$z_T$, $T\in\N$, take values in a compact state space~$Z\subseteq \mb R^n$, then there exists a deterministic vector~$z_\infty\in Z$ such that~$\limsup_{T\to\infty} \mb P [ \|z_T - z_\infty\| <\rho]>0$ for all~$\rho>0$.
\end{lemma}

% The proof requires one preliminary step. Let $\{\widehat x_T\}_{T\in\N}$ be a sequence of random variables taking value in the compact set $X$. Then, there exists a function $x_\infty:\Theta\to X$ such that
% \begin{equation}\label{proof:prescr:optimality:step1}
% \limsup_{T\to\infty} \mb P_\theta[ \|\widehat x_T - x_\infty(\theta)\| <\delta]>0 \quad \forall \delta>0 \quad \forall \theta\in \Theta.
% \end{equation}
% The proof of \eqref{proof:prescr:optimality:step1} is relegated to
%  Appendix~\ref{app:auxiliary}, Lemma~\ref{lem:auxiliary:lem1}. 
If the sequence~$\{z_T\}_{T\in\N}$ converges in probability, then~$z_\infty$ may be set to any point in the support of the limiting random variable. We emphasize, however, that Lemma~\ref{lem:auxiliary:lem1} remains valid even if the sequence~$\{z_T\}_{T\in\N}$ fails to converge. One can thus think of~$z_\infty$ as a probabilistic accumulation point of~$\{z_T\}_{T\in\N}$. %Note that if $\{\widehat x_T\}_{T\in\N}$ converges in probability to~$x_\infty(\theta)$, then any selection $x_\infty(\theta)\in \mathsf{supp} \  x_\infty(\theta)$ would satisfy \eqref{proof:prescr:optimality:lem} where we take $\mathsf{supp} \ x_\infty(\theta)$ as the smallest closed set for which $x_\infty(\theta)\in \mathsf{supp} \ x_\infty(\theta)$ occurs with probability one. Essentially, we show that we can prove \eqref{proof:prescr:optimality:lem} without requiring $\{\widehat x_T\}_{T\in\N}$ to converge at all.

\begin{proof}[Proof of Lemma~\ref{lem:auxiliary:lem1}]
As~$Z$ is bounded, we may assume without loss of generality that~$\|z\|\leq 1$ for all~$z\in Z$. We then  construct~$z_\infty$ as follows. First, we define~$z_\infty^0=0$. Next, for each~$k\in\N$, we set~$\rho^k=1/2^{k-1}$, and we recursively use the procedure described below to construct~$z_\infty^k\in Z$ with~$\|z_\infty^k-z_\infty^{k-1}\|\leq \rho^k$ and
\begin{equation}
    \label{eq:iterate-k}
    \limsup_{T\to\infty} \mb P\Big[ \|z_T - z_{\infty}^k\| < \rho^k \Big] >0.
\end{equation}
Before showing how one can construct iterates~$z_\infty^k$ with these properties, we explain how they can be used to prove the the lemma. To this end, note that $\{z_\infty^k\}_{k\in\N}$ represents a Cauchy sequence because 
\[
    \|z_\infty^{k}-z_\infty^{k'}\|\leq \sum_{\ell=k}^{k'-1} \|z_\infty^{\ell+1}-z_\infty^\ell\|\leq \sum_{\ell=k}^{k'-1} \frac{1}{2^\ell}\leq \sum_{\ell=k}^{\infty} \frac{1}{2^\ell} = \frac{1}{2^{k}}=\frac{\rho^k}{2} \quad \forall k<k'
\]
and because~$\rho^k$ converges to~$0$ as~$k$ grows. As~$Z\subseteq \mb R^n$ is compact, the sequence~$\{z_\infty^k\}_{k\in\N}$ thus converges to a point~$z_\infty\in Z$ that satisfies~$\|z^{k}_\infty-z_\infty\|\leq \rho^k/2$ for all~$k\in\N$. Next, select any~$\rho>0$ and an arbitrary~$k\in \N$ with~$\rho^k<2\rho/3$. The triangle inequality then implies that for all~$z\in Z$ with~$\|z-z_\infty^k\|<\rho^k$ we have
\[
    \|z-z_\infty\|\leq \|z-z_\infty^k\|+\|z_\infty^k-z_\infty\|\leq 3\rho^k/2<\rho.
\]
We may thus conclude that
\begin{equation*}
    \limsup_{T\to\infty} \mb P\Big[ \|z_T - z_{\infty}\| < \rho \Big] \ge 
    \limsup_{T\to\infty} \mb P\Big[ \|z_T - z_{\infty}^k\| < \rho^k \Big]>0,
\end{equation*}
where the strict inequality follows from~\eqref{eq:iterate-k}. As~$\rho>0$ was chosen arbitrarily, the claim follows.

It remains to be shown that one can always construct iterates~$z_\infty^k\in Z$ with~$\|z_\infty^k-z_\infty^{k-1}\|\leq \rho^k$ that satisfy~\eqref{eq:iterate-k}. To see this, initialize the iteration counter as~$k=1$, and set~$Z^k=Z$. As~$Z^k$ is bounded, there exist a finite index set~$\mc J^k$ and finitely many points~$\overline z_{j}^k\in Z^k$, $j\in \mc J^k$, such that the open balls~$B_{j}^k=\{z\in Z:\|z-\overline z_{j}^k\|< \rho^k\}$, $j\in\mc J^k$, cover~$Z^k$. Next, select~$j^k\in\mc J^k$ with~$\limsup_{T\to\infty} \mb P[z_T\in B_{j^k}^k]>0$, and set~$z_{\infty}^k=\overline z_{j^k}^k$. Note that~$j^k$ exists because~$\mb P[z_T\in \cup_{j\in\mc J^k}B_{j}^k]=\mb P[z_T\in Z^k]$ for all~$T\in\N$ and because~$\limsup_{T\to\infty}\mb P[z_T\in Z^k]>0$. Finally, define~$Z^{k+1}=B^k_{j^k}$, increment the iteration counter~$k$ and repeat the above procedure. By construction, $z_\infty^{k+1}$ belongs to~$Z$ as well as to the ball of radius~$\rho^k$ around~$z_\infty^k$, and it satisfies~\eqref{eq:iterate-k} for every~$k\in\mb N$. Hence, the sequence~$\{z_\infty^k\}_{k\in\N}$ displays all desired properties.
\end{proof}

%%%%%%%%%%%%%%%%%%%%%%%%%%%%%
\section{Mean-variance portfolio selection}\label{app:portfolio:optimization}

We exemplify the construction of optimal data-driven predictors and prescriptors in the context of a Marko\-witz-type portfolio selection problem with i.i.d.\ Gaussian asset returns. We adopt here all conventions and assumptions of Section~\ref{sec:iid-parametric}. The portfolio selection problem to be studied seeks a long-short portfolio from within the feasible set $X = \{x\in\mathbb R^d:\sum_{i=1}^d x_i=1\}$ that minimizes the mean-variance objective
\begin{equation}\label{eq:cost:mean:var:port}
c(x,\theta) =  \mathbb{E}_\theta[-x^\top \xi] + \rho \, \mathsf{Var}_\theta(x^\top \xi)= - x^\top \theta + \rho x^\top \Sigma x,
\end{equation}
where the vector~$\xi$ of asset returns is governed by a Gaussian distribution with mean $\theta\in\mathbb R^d$ and covariance matrix $\Sigma\in\mathbb R^{d\times d}$. Note that~$\Sigma$ can be estimated to within reasonable accuracy from about one year of return data, whereas~$\theta$ is subject to a blurring phenomenon and is hard to estimate accurately even when ten years of return data are available \cite[Chapter~8]{ref:Luenberger-97}. It is therefore reasonable to assume that~$\Sigma$ is known but~$\theta$ is unknown. In the following, we aim to construct a data-driven predictor-prescriptor pair~$(\widehat c_T, \widehat x_T)$ whose out-of-sample disappointment decays at a
% \begin{equation}\label{eq:portf:security}
%     \limsup_{T\to\infty} \frac{1}{T} \log \mathbb{P}_\theta[c(\widehat x_T,\theta)>\widehat c_T(\widehat x_T)]\leq -r
% \end{equation}
prescribed rate~$r>0$. From Section~\ref{sec:iid-parametric} we know that the optimal (least conservative) predictor-prescriptor pair with this property is given by
\begin{equation*}%\label{eq:portf:DRO}
    \widehat c_T^\star(x) = \max_{\theta\in\mathbb{R}^d} \left\{- x^\top \theta + \rho x^\top \Sigma x \ : \ \frac{1}{2}(\theta-\widehat S_T)^\top \Sigma^{-1}(\theta-\widehat S_T) \leq r \right\} \quad\text{and}\quad \widehat x_T^\star \in \arg\min_{x\in X} \widehat c_T^\star (x),
\end{equation*}
where $\widehat S_T=\frac{1}{T}\sum_{t=1}^T\xi_t$ denotes the empirical average return. More precisely, Theorems~\ref{thm:predictor:equivalence} and~\ref{thm:prescriptor:equivalence} imply that $(\widehat c_T^\star, \widehat x_T^\star)$ constructed as above Pareto-dominates every other conceivable data-driven predictor-prescriptor pair. As optimization problems with a linear objective function and an ellipsoidal feasible set can be solved in closed form, we can re-express the optimal predictor as~$ \widehat c_T^\star(x) = -x^\top \widehat S_T - \sqrt{2r} \| \Sigma^{1/2} x\|_2 + \rho x^\top \Sigma x$, and therefore the optimal prescriptor~$\widehat x_T^\star$ can be computed efficiently by solving a second-order cone program.
%The bound~\eqref{eq:portf:security} can be interpreted as a security guarantee for the investor about the probability that the realized cost $c(\widehat x_T,\theta)$ under investment policy $\widehat x_T$ is lower than some computed upper confidence value $\widehat c_T(\widehat x_T)$. A first guess of data-driven predictor and prescriptor is via SAA, i.e.,
We will compare our optimal predictor against a penalized SAA predictor $\widehat c_T^{\,\mathsf{SAA}}(x) =c(x,\widehat S_T)+\varepsilon$, where the bias parameter~$\varepsilon\geq 0$ has no impact on the associated prescriptor $\widehat x_T^{\,\mathsf{SAA}} \in \arg\min_{x\in X} \widehat c_T^{\,\mathsf{SAA}} (x)$ but can be chosen judiciously to achieve any desired out-of-sample disappointment. We will also compare the optimal predictor against the distributionally robust predictor
% where $\gamma>0$ is a penalty parameter to be selected to ensure the exponential decay~\eqref{eq:portf:security}. In this setting, intuitively one aims to select the smallest penalty parameter $\gamma$ such that \eqref{eq:portf:security} holds with equality. How to find this constant, however, is largely unclear. We compare this method with the optimal DRO predictor and prescriptor suggested in Section~\ref{sec:iid-parametric}, i.e., 
% \begin{equation}\label{eq:portf:DRO}
%     \widehat c_T^{\,\mathsf{DRO}}(x) = \max_{\theta\in\mathbb{R}^d} \left\{- x^\top \theta + \rho x^\top \Sigma x \ : \ \frac{1}{2}(\theta-\widehat S_T)^\top \Sigma^{-1}(\theta-\widehat S_T) \leq r \right\}, \qquad \widehat x_T^{\,\mathsf{DRO}} \in \arg\min_{x\in X} \widehat c_T^{\,\mathsf{DRO}}(x),
% \end{equation}
% where $\widehat S_T=\frac{1}{T}\sum_{t=1}^T\xi_t$ is the empirical average return. As shown in Section~\ref{sec:iid-parametric}, the predictor and prescriptor~\eqref{eq:portf:DRO} are optimal with respect to \eqref{eq:optimal-prescriptor}.
%As a further comparison we introduce the DRO predictor with a Wasserstein ambiguity set, which in the setting considered here according to \cite[Theorem~2]{ref:Nguyen-risk-21}, reduces to 
\begin{equation*}%\label{eq:portf:Wstein}
    \widehat c_T^{\,\mathsf{W}}(x) = \max_{\theta\in\mathbb{R}^d} \left\{- x^\top \theta + \rho x^\top \Sigma x \ : \ \|\theta-\widehat S_T\|_2 \leq \varepsilon \right\}, %\qquad \widehat x_T^{\,\mathsf{W}} \in \arg\min_{x\in X} \widehat c_T^{\,\mathsf{W}}(x),
\end{equation*}
which evaluates the worst-case mean-variance functional across all Gaussian asset return distributions with a 2-Wasserstein distance of at most $\varepsilon\geq 0$ from the nominal distribution~$\mathcal N(\widehat S_T,\Sigma)$. Indeed, the 2-Wasserstein distance between two Gaussian distributions with mean vectors~$\theta$ and~$\widehat S_T$, respectively, and with the same covariance matrix is given by~$\|\theta-\widehat S_T\|_2$ \cite{gelbrich1990wasserstein}. One readily verifies that evaluating the Wasserstein distributionally robust prescriptor~$\widehat x_T^{\,\mathsf{W}} \in \arg\min_{x\in X} \widehat c_T^{\,\mathsf{W}}(x)$ is tantamount to solving a second-order cone program. We emphasize that the distributionally robust predictor-prescriptor pairs described in \cite[Section~5]{ref:vanParys:fromdata-17} are not well-defined in the context considered here because the mean-variance portfolio selection problem fails to be risk-neutral and because the Gaussian return distribution fails to have a compact support.

We conduct several numerical experiments with synthetic asset return data drawn from a normal distribution~$\mathcal{N}(\theta_\star,\Sigma)$, where~$\theta_\star$ and~$\Sigma$ are calibrated to match the historical mean vector and covariance matrix of the ``25 portfolios formed on size and book-to-market'' dataset from the Fama-French online data library.\footnote{See \url{http://mba.tuck.dartmouth.edu/pages/faculty/ken.french/data_library.html} (accessed on 7 September 2022).}
This dataset contains 1{,}153 monthly returns between July 1926 and July 2022 of $d=25$ portfolios of stocks (``assets'') formed on size and on the ratio of book equity to market equity.
Working with synthetic data allows us to test the performance of the proposed data-driven predictors and prescriptors based on datasets of an arbitrary size.
 In what follows we set the risk-aversion parameter to~$\rho=\frac{1}{2}$. Figure~\ref{fig:portf:pareto} visualizes the trade-off between the asymptotic in-sample cost $\lim_{T\to\infty} \mb E_{\theta_\star}[\widehat c_T(\widehat x_T)]$ and the decay rate of the out-of-sample disappointment $\lim_{T\to\infty}-\frac{1}{T}\log \mb P_{\theta_\star} [ c(\widehat x_T, {\theta_\star}) > \widehat c_T(\widehat x_T) ]$ for the optimal, the penalized SAA, the Wasserstein distributionally robust predictor-prescriptor pairs
%{\color{red} 
% and the method of~\cite{ref:vanParys:fromdata-17}}
as a function of~$r$ and~$\varepsilon$. %Not surprisingly, higher decay rates of the out-of-sample probabilities result in higher expected in-sample costs.

While computing the optimal prescriptor~$\widehat x_T^\star$ for a single training dataset is essentially instant, numerically estimating the out-of-sample disappointment of~$(\widehat c_T^\star, \widehat x_T^\star)$ for a fixed sample size~$T$ requires on the order of $e^{rT}$ independent training datasets. Moreover, the sample size~$T$ needed to approximate the asymptotic decay rate of the out-of-sample disappointment is significantly higher in this portfolio selection problem with $25$ random asset returns than in the newsvendor problem of Section~\ref{sec:newsvendor} with a single random demand ($T=20{,}000$ versus $T=200)$. Evaluating the decay rate of the out-of-sample disappointment therefore becomes cumbersome as $r$ grows. For this reason, Figure~\ref{fig:portf:pareto} focuses only on a relatively narrow range of small decay rates.  
Nevertheless, Figure~\ref{fig:portf:pareto} does indeed corroborate the theoretical Pareto dominance property of $(\widehat c^\star, \widehat x^\star)$ in this range, which is guaranteed by Theorem~\ref{thm:prescriptor:equivalence}. We emphasize that practitioners do not have to compute the Pareto curve of Figure~\ref{fig:portf:pareto}. They only need to solve a single optimization problem to find the optimal decision corresponding to their dataset.
% {\color{red}Moreover, Figure~\ref{fig:portf:pareto} shows that the presented method for the considered mean-variance portfolio selection problem is inherently different compared to \cite{ref:vanParys:fromdata-17}.}

%Based on the observation that serial dependence in monthly stock returns are negligible we model these return rates as i.i.d.~random variables from a normal distribution with mean $\theta\in\mathbb{R}^d$ and variance $\Sigma\in\mathbb{R}^{d\times d}$. For the sake of this numerical experiment, based on this real world data, we compute the corresponding least-squares estimates for mean and variance and consider those as ground truth $\theta_\star$ and $\Sigma_\star$. We then generate synthetic i.i.d.~return from $\mathcal{N}(\theta_\star,\Sigma_\star)$ as our training data. This approach allows us to simulate the high sample numbers required to visualize the asymptotic regime considered in this work. 
 
%The red line shows the penalized SAA method~\eqref{eq:portf:SAA} for a varying parameter $\gamma$. The green line visualizes the Wasserstein DRO approach~\eqref{eq:portf:Wstein} for different choices of radii $r_\mathsf{W}$ and the blue line corresponds to \eqref{eq:portf:DRO} for varying $r$. 
%The plots in Figure~\ref{fig:portf:pareto} corroborate the Pareto optimality of the DRO method~\eqref{eq:portf:DRO} established in Theorem~\ref{thm:optimality_prescriptor}. 

%%%%%
\begin{figure}[t] 
\centering
{% This file was created by matlab2tikz.
% Minimal pgfplots version: 1.3
%  http://www.mathworks.com/matlabcentral/fileexchange/22022-matlab2tikz
\begin{tikzpicture}

\begin{axis}[%
width=2.2in,
height=1.8in,
at={(1.011111in,0.641667in)},
scale only axis,
unbounded coords=jump,
xmin=0,
xmax=0.00055,
ymin=-0.041,
ymax=-0.0305,
ytick = {-0.04, -0.038,-0.036,-0.034,-0.032},
xlabel={\footnotesize{$\lim\limits_{T\to\infty}-\frac{1}{T}\log \mb P_{\theta_\star} [ c(\widehat x_T, {\theta_\star}) > \widehat c_T(\widehat x_T) ]$}},
x label style={xshift=-3.0em,yshift=-0.25em},
ylabel={\footnotesize{$\lim\limits_{T\to\infty} \mb E_{\theta_\star}[\widehat c_T(\widehat x_T)]$}},
y label style={yshift=0.25em},
title style={font=\bfseries},
legend to name=named:new,
legend style={legend cell align=left,align=left,draw=white!15!black,legend columns=3}
]
\addplot [color=red,solid,line width = 1.2pt, smooth]
  table[row sep=crcr]{%
  1.62126926341671e-05	-0.0409561686039733\\
1.62196081040843e-05	-0.0409557344068034\\
1.63488727529712e-05	-0.0409476279283147\\
1.65016506496156e-05	-0.0409380702814846\\
1.6515562725387e-05	-0.0409372012115284\\
1.66618737055642e-05	-0.0409280740272405\\
1.68366135213536e-05	-0.0409172036689653\\
1.68366135213536e-05	-0.0409172036689653\\
1.68459501536575e-05	-0.0409166237735459\\
1.71644374840744e-05	-0.0408968986571825\\
1.74329767143356e-05	-0.0408803512849132\\
1.74329767143356e-05	-0.0408803512849132\\
1.78313663012432e-05	-0.0408559436481864\\
1.90093650461948e-05	-0.0407847468240182\\
1.97683136715375e-05	-0.0407396360819442\\
2.01558932519908e-05	-0.0407168248737946\\
2.13829796808967e-05	-0.0406455953577706\\
2.21479656267872e-05	-0.0406019399465622\\
2.25859246312813e-05	-0.0405772020687374\\
2.38910385150152e-05	-0.0405045662055056\\
2.38910385150152e-05	-0.0405045662055056\\
2.48187491734004e-05	-0.0404539026224449\\
2.76489761698357e-05	-0.0403041388681379\\
3.05322979524101e-05	-0.0401586611412716\\
3.18320415559819e-05	-0.0400953103445331\\
3.38407309636193e-05	-0.040000008812546\\
3.38407309636193e-05	-0.040000008812546\\
4.14941852066708e-05	-0.0396636515446596\\
5.29407409397022e-05	-0.0392281695964092\\
5.58591837126773e-05	-0.0391280117108424\\
6.61565893081125e-05	-0.0388033482501249\\
8.07644163056138e-05	-0.0384031527290713\\
8.41811405311739e-05	-0.0383173515255692\\
9.66776737422971e-05	-0.0380220327526961\\
0.0001146977883148	-0.0376320709189825\\
0.000119189757420979	-0.0375392583753529\\
0.000133537720817926	-0.0372504596669901\\
0.000154943651623675	-0.0368346213936643\\
0.000160694414167858	-0.0367253309030392\\
0.000179256010991154	-0.0363798233098491\\
0.000203336469759456	-0.0359521879869786\\
0.000207704728331394	-0.0358777146410567\\
0.000264583641191468	-0.0350169525959605\\
0.000319846482760807	-0.0343380832837719\\
0.000387200165159138	-0.033389865256697\\
0.000425859659570812	-0.0326167806297877\\
0.000460517018598809	-0.0318133123713712\\
0.000515447633032215	-0.0309781091154828\\
};
\addlegendentry{penalized SAA\quad};

\addplot [color=blue,solid,line width = 1.2pt]
  table[row sep=crcr]{%
1.75867392843598e-05	-0.0408255794339219\\
2.49338744915172e-05	-0.0404184420635016\\
2.92475345041455e-05	-0.0402011573527395\\
3.28922194428779e-05	-0.0400289515316029\\
3.62628617239089e-05	-0.0398783157755419\\
3.62628617239089e-05	-0.0398783157755419\\
5.4196036996256e-05	-0.0391951579077383\\
6.98711359453867e-05	-0.0387243051419821\\
8.46409760686576e-05	-0.0383551593061337\\
9.89044523709774e-05	-0.0380460176106986\\
0.000112880586367566	-0.037773811146366\\
0.000127105606177939	-0.0375176460964654\\
0.000127105606177939	-0.0375176460964654\\
0.000202277719902633	-0.0363305666538489\\
0.000279078242096598	-0.0353695136972974\\
0.000345387763949107	-0.0347329654583103\\
0.000418152125579449	-0.0338907191954933\\
0.00044613291497622	-0.0334467234828731\\
0.000480790274004217	-0.0328381261466948\\
0.000515447633032215	-0.0323274581500017\\
%   1.75867392843598e-05	-0.0408678104361822\\
% 2.49338744915172e-05	-0.0404093550208976\\
% 2.92475345041455e-05	-0.0401818860824685\\
% 3.28922194428779e-05	-0.0400068382951591\\
% 3.62628617239089e-05	-0.0398592723203399\\
% 3.62628617239089e-05	-0.0398592723203399\\
% 5.4196036996256e-05	-0.0392160825529685\\
% 6.98711359453867e-05	-0.0387564034126581\\
% 8.46409760686576e-05	-0.038380013069489\\
% 9.89044523709774e-05	-0.0380541293668348\\
% 0.000112880586367566	-0.0377630949708017\\
% 0.000127105606177939	-0.0374979808955698\\
% 0.000127105606177939	-0.0374979808955698\\
% 0.000202277719902633	-0.0363283118745076\\
% 0.000279078242096598	-0.0354200419869775\\
% 0.000345387763949107	-0.0346559711313428\\
% 0.000418152125579449	-0.0339867516118001\\
% 0.00044613291497622	-0.0333860470556136\\
% 0.000480790274004217	-0.0328378335799254\\
% 0.000515447633032215	-0.0323315011437732\\
% old
% 1.75595006854081e-05	-0.040868497139363\\
% 1.78659003901421e-05	-0.0408426363183124\\
% 1.79875549513742e-05	-0.0408663761076867\\
% 2.52256130854699e-05	-0.0404079228933365\\
% 2.93718363068916e-05	-0.0401804550873031\\
% 3.29324570076189e-05	-0.0400054081898009\\
% 3.61854743071293e-05	-0.0398578429775723\\
% 3.61854743071293e-05	-0.0398578429775723\\
% 6.14474395198057e-05	-0.0389710152905813\\
% 8.3792242231585e-05	    -0.0383785920128025\\
% 0.000105515660267346	-0.0379034980461972\\
% 0.000126161588205483	-0.0374965653482095\\
% 0.000126161588205483	-0.0374965653482095\\
% 0.000253360282279233	-0.0357123598750729\\
% 0.000397878870174041	-0.0344390467677642\\
};
\addlegendentry{optimal \quad};

\addplot [color=green,solid,line width = 1.2pt, smooth]
  table[row sep=crcr]{%
1.74282530487676e-05	-0.0407972233251059\\
2.07606229781522e-05	-0.0406106563102659\\
2.39259894507038e-05	-0.0404411137112566\\
2.76953556734221e-05	-0.0402484546763218\\
3.07630320951437e-05	-0.0400986809280225\\
3.36280273475779e-05	-0.0399641983394544\\
3.36280273475779e-05	-0.0399641983394544\\
3.36280273475779e-05	-0.0399641983394544\\
4.82390470517295e-05	-0.0393507588251983\\
4.89570894345972e-05	-0.0393234246385288\\
6.10785598034279e-05	-0.038895818576493\\
6.93280531674218e-05	-0.0386369872208215\\
7.2721522844343e-05	-0.0385370298458351\\
8.41183546228429e-05	-0.0382253122081461\\
8.75500738377944e-05	-0.0381378736439398\\
9.45017019248198e-05	-0.0379686116110982\\
0.000104521754157735	-0.0377407597411032\\
0.000104521754157735	-0.0377407597411032\\
0.000104521754157735	-0.0377407597411032\\
0.000104521754157735	-0.0377407597411032\\
0.000159465851131333	-0.0367033156610928\\
0.000201240924978675	-0.0360380659599477\\
0.000206550427329685	-0.0359580964640256\\
0.000210995389259872	-0.0358918671768251\\
0.000263276377186252	-0.0351603120977471\\
0.000300244378372006	-0.0346891388847453\\
0.000315998430704001	-0.0344950859067899\\
0.000342160837892228	-0.0341719521398363\\
0.000343748272807957	-0.034152023847667\\
0.000363221511146043	-0.0339011979268965\\
0.000405586404165404	-0.0332825379959788\\
0.000418152125579449	-0.0330697639776537\\
0.000418152125579449	-0.0330697639776537\\
0.000425859659570812	-0.0329308828306194\\
0.000425859659570812	-0.0329308828306194\\
0.000460517018598809	-0.0322181533384574\\
0.000460517018598809	-0.0322181533384574\\
0.000480790274004217	-0.0317317439773364\\
0.000480790274004217	-0.0317317439773364\\
0.000480790274004217	-0.0317317439773364\\
0.000480790274004217	-0.0317317439773364\\
0.000480790274004217	-0.0317317439773364\\
0.000480790274004217	-0.0317317439773364\\
};
\addlegendentry{Wasserstein DRO};

% \addplot [color=cyan,solid,line width = 1.2pt, smooth]
%   table[row sep=crcr]{%
% 1.61758262730498e-05	-0.0409732280435874\\
% 1.93714372811033e-05	-0.0407709521150647\\
% 2.2864242841898e-05	-0.0405686762127184\\
% 2.62293269508014e-05	-0.0403664003425406\\
% 3.05138851368825e-05	-0.0401641244998802\\
% 3.05138851368825e-05	-0.0401641244998802\\
% 6.51344388632547e-05	-0.0389382109207028\\
% 0.00011109635949883	-0.037712298399238\\
% 0.000176338030231875	-0.0364863869318949\\
% 0.000269084948774354	-0.0352604765293305\\
% 0.000391202300542815	-0.0340345671847062\\
% inf	-0.0328086588941359\\
% inf	-0.0328086588941359\\
% inf	-0.0261219050609899\\
% inf	-0.0194351827211402\\
% inf	-0.012748491884385\\
% inf	-0.00606183254411808\\
% inf	0.00062479529330567\\
% inf	0.00731139163531834\\
% inf	0.0139979564722092\\
% inf	0.02068448981478\\
% inf	0.027370991656246\\
% inf	0.0340574620047471\\
% inf	0.0407439008468261\\
% };
% \addlegendentry{Method~\cite{ref:vanParys:fromdata-17}};

\end{axis}
\end{tikzpicture}%
}
 \newline
\ref{named:new}
\caption[]{Asymptotic in-sample cost versus decay rate of out-of-sample disappointment. All probabilities and expectations involving random training data are evaluated empirically using~$10^4$ independent training~sets.}
\label{fig:portf:pareto}
\end{figure}

% consider a data set of daily returns from 1970 to 2011 (14'810 data points) from Kenneth
% French's Web site with $d=25$ value-weighted portfolios of stocks sorted on size and investment. 
% Given the data set we estimate the parameters of the AR process \eqref{AR:process:general:general} comprising the AR system matrix $A\in\mathbb{R}^{d\times d}$, the covariance matrix $\Sigma\in\mathbb{R}^{d\times d}$ and the drift term $\theta_\star \in\mathbb{R}^d$, via least squares estimation \cite[Section~3.2]{ref:Luetkepohl-05}.
% For our simulation experiment we consider the risk factor $\rho=0.5$.

% We assume the idiosyncratic risk factor to be $(\theta_\star)_i = i\cdot 3\%$ specific to asset $i$. The AR process is described by the asymptotically stable matrix $A\in\mathbb{R}^{d\times d}$ and the covariance matrix $\Sigma$. For our simulation experiment these two matrices are chosen as follows: We first randomly sample a diagonal matrix $A' = \mathrm{diag}(a)$, where $a\in\mathbb{R}^d$ is a random vector and $\Sigma' = \mathrm{diag}(\sigma)$, where $\sigma\in\mathbb{R}_+^d$ is randomly chosen. We then sample random matrices $A''\in\mathbb{R}^{d\times d}$ and $\Sigma''\succ 0$ and set $A = (1-\lambda)A' + \lambda A''$ and $\Sigma = (1-\lambda)\Sigma' + \lambda \Sigma''$, where $\lambda\in[0,1]$ is a small parameter. Clearly $\Sigma\succ 0$ and in case $A$ is not asymptotically stable we scale it with a scalar to ensure asymptotic stability. 

The previous experiment focused on asymptotic performance. In contrast,
Figure~\ref{fig:portfolio:iid:setting} depicts the out-of-sample disappointment of different predictor-prescriptor pairs and their decay rates for finite~$T$. As expected, the out-of-sample disappointment of the SAA predictor-prescriptor pair with $\varepsilon=0$ saturates at a strictly positive level as~$T$ tends to infinity. Although the out-of-sample disappointment decays exponentially for~$\varepsilon>0$, the decay rate appears to be quite sensitive to the particular choice of~$\varepsilon$. Calibrating~$\varepsilon$ to achieve a desired decay rate~$r$ seems therefore quite challenging. The same remark can be made concerning the Wasserstein distributionally robust predictor-prescriptor pairs.
In contrast, the optimal predictor-prescriptor pair attains the desired decay rate without any calibration; see Figure~\ref{fig:port:iid:decay}. %We note that simulating the decay rate of the out-of-sample disappointment probabilities for prescriptors that decay at a large rate becomes numerically expensive as it requires a vast number of independent experiments.

% performance of the sophisticated investor, i.e., a decision maker that uses a prescriptor induced by the rate function \eqref{eq:ex:AR:rate:function} and a na\"ive investor using a prescriptor via the i.i.d.~rate function \eqref{eq:ex:iid:rate:function}. The trade-off between in-sample cost and the out-of-sample disappointment probability for the data-driven prescriptor generated from the AR model is depicted in Figure~\ref{fig:portfolio:new} (dashed red line).

Figure~\ref{fig:portf:out-of-sample} reports the expected out-of-sample cost~$\mathbb{E}_{\theta_\star}[c(\widehat x_T,\theta_\star)]$ of different prescriptors~$\widehat x_T$ as a function of~$T$. Even though our theory offers only {\em indirect} statistical guarantees on the out-of-sample cost of the optimal prescriptor~$\widehat x_T^\star$ ({\em e.g.}, the out-of-sample cost $c(\widehat x_T^\star,\theta_\star)$ falls below the in-sample cost $\widehat c_T^\star(\widehat x^\star_T)$ with high probability $\approx 1-e^{-rT}$), the expected out-of-sample cost~$\mathbb{E}_{\theta_\star}[c(\widehat x_T,\theta_\star)]$ is often the actual quantity of interest in applications. Note that unlike Figures~\ref{fig:portf:pareto} and~\ref{fig:portfolio:iid:setting}, Figure~\ref{fig:portf:out-of-sample} focuses on the small data regime. We observe that all distributionally robust prescriptors outperform the SAA prescriptor for small sample sizes~$T$, which is consistent with the findings in~\cite{ref:Peyman-18}. As~$T$ grows, however, the SAA prescriptor eventually displays the lowest out-of-sample cost because it is asymptotically consistent---unlike the DRO prescriptors with an ambiguity set of a fixed positive radius. In all experiments of Figure~\ref{fig:portf:out-of-sample} the radius~$\varepsilon$ of the Wasserstein ambiguity set is calibrated to ensure that the out-of-sample disappointment of the corresponding prescriptor decays at the prescribed rate~$r$. Figure~\ref{fig:portf:out-of-sample} reveals that the optimal and the Wasserstein distributionally robust prescriptors display a similar out-of-sample cost. Yet, in all experiments the out-of-sample disappointment of the optimal prescriptor decays faster.

\begin{figure}[t] 
\centering
\subfloat[Out-of-sample disappointment
]{% This file was created by matlab2tikz.
% Minimal pgfplots version: 1.3
%  http://www.mathworks.com/matlabcentral/fileexchange/22022-matlab2tikz
\definecolor{mycolor1}{rgb}{0.00000,0.44700,0.74100}%
\definecolor{mycolor2}{rgb}{0.85000,0.32500,0.09800}%
\begin{tikzpicture}

\begin{axis}[%
width=1.75in,
height=1.75in,
at={(1.011111in,0.641667in)},
scale only axis,
xmin=0,
xmax=8000,
ymin=0,
ymin=0,
ymax=1,
xtick = {0,2000,4000,6000,8000},
xticklabels = {0,2{,}000,4{,}000,6{,}000,8{,}000},
xlabel={\footnotesize{$T$}},
ylabel={\footnotesize{$\mb P_{\theta_\star} [ c(\widehat x_T, {\theta_\star}) > \widehat c_T(\widehat x_T) ]$}},
y label style={yshift=-0.25em},
scaled ticks=false, tick label style={/pgf/number format/fixed},
legend to name=named,
legend style={legend cell align=left,align=left,draw=white!15!black,legend columns=2}
]

% \addplot [color=blue,only marks,mark=*,
% mark options={scale=1.2, fill=blue},
% forget plot]
%   table[row sep=crcr]{%
% 0.0036	0.376\\
% };

% \addplot [color=red,only marks,mark=*,
% mark options={scale=1.2, fill=red},
% forget plot]
%   table[row sep=crcr]{%
% 0.0023	0.361\\
% };

% % trick for legend
% %%%%%%%%%%%%%%%%%%%%%%%%%
% \addplot [color=blue,style=solid,line width = 1.5pt]
%   table[row sep=crcr]{%
% 0.02	0.330505369033144\\
% 0.03	0.338864790100878\\
% };
% \addlegendentry{Out-of-sample cost $\mb E_{\theta_\star}[c(\widehat x_T,\theta_\star)]$ from i.i.d.~ambiguity set \eqref{eq:ex:iid:rate:function}};

% \addplot [color=red,style=solid,line width = 1.5pt]
%   table[row sep=crcr]{%
% 0.02	0.329037961337489\\
% 0.03	0.337708399250403\\
% };
% \addlegendentry{Out-of-sample cost $\mb E_{\theta_\star}[c(\widehat x_T,\theta_\star)]$ from optimal~ambiguity set \eqref{eq:ex:AR:rate:function}};

% SAA
\addplot [color=red,solid,line width = 1.5pt,smooth]
  table[row sep=crcr]{%
5	1\\
32	1\\
59	1\\
85	1\\
112	1\\
139	1\\
166	1\\
193	1\\
220	1\\
246	1\\
273	0.9999875\\
300	0.9999875\\
350	1\\
800	0.998425\\
1250	0.9914625\\
1700	0.9802125\\
2150	0.9649375\\
2600	0.951675\\
3050	0.935925\\
3500	0.922275\\
3950	0.909\\
4400	0.896875\\
4850	0.885875\\
5300	0.8754375\\
5750	0.866125\\
6200	0.856525\\
6650	0.84865\\
7100	0.8404\\
7550	0.8329125\\
8000	0.826525\\
};
\addlegendentry{SAA method};

\addplot [color=blue,solid,line width = 1.5pt,smooth]
  table[row sep=crcr]{%
5	1\\
32	1\\
59	1\\
85	1\\
112	1\\
139	1\\
166	1\\
193	0.9999875\\
220	0.9999625\\
246	0.9997875\\
273	0.9998375\\
300	0.99965\\
350	0.999025\\
800	0.9313\\
1250	0.738425\\
1700	0.5247375\\
2150	0.3524875\\
2600	0.227325\\
3050	0.1468875\\
3500	0.0924125\\
3950	0.05815\\
4400	0.0362375\\
4850	0.0222375\\
5300	0.0139375\\
5750	0.0086125\\
6200	0.0053875\\
6650	0.003325\\
7100	0.0021375\\
7550	0.0013625\\
8000	0.000825\\
};
\addlegendentry{optimal DRO method};

\addplot [color=red,dotted,line width = 1.5pt,smooth]
  table[row sep=crcr]{%
5	1\\
32	1\\
59	1\\
85	1\\
112	1\\
139	0.9999875\\
166	0.999925\\
193	0.9997875\\
220	0.9994\\
246	0.9985\\
273	0.997075\\
300	0.9949875\\
350	0.988275\\
800	0.7038125\\
1250	0.34455\\
1700	0.14635\\
2150	0.058375\\
2600	0.021875\\
3050	0.0081625\\
3500	0.0030625\\
3950	0.0010125\\
4400	0.0003875\\
4850	0.00015\\
5300	5e-05\\
5750	1.25e-05\\
6200	0\\
6650	0\\
7100	0\\
7550	0\\
8000	0\\
};
\addlegendentry{SAA - large gamma};

\addplot [color=red,dashed,line width = 1.5pt,smooth]
  table[row sep=crcr]{%
5	1\\
32	1\\
59	1\\
85	1\\
112	1\\
139	1\\
166	1\\
193	0.9999875\\
220	0.999975\\
246	0.999875\\
273	0.99985\\
300	0.9997\\
350	0.9991\\
800	0.9436625\\
1250	0.795575\\
1700	0.6234875\\
2150	0.474675\\
2600	0.351875\\
3050	0.25975\\
3500	0.1914\\
3950	0.140225\\
4400	0.1019875\\
4850	0.0748875\\
5300	0.05485\\
5750	0.039725\\
6200	0.029775\\
6650	0.0215625\\
7100	0.01615\\
7550	0.011425\\
8000	0.00815\\
};
\addlegendentry{SAA - small gamma};

\addplot [color=green,dotted,line width = 1.5pt,smooth]
  table[row sep=crcr]{%
5	1\\
32	0.9971375\\
59	0.942325\\
85	0.8151125\\
112	0.663925\\
139	0.5336625\\
166	0.4351875\\
193	0.36115\\
220	0.305625\\
246	0.2647875\\
273	0.2296625\\
300	0.2031875\\
350	0.1691125\\
800	0.0599125\\
1250	0.0274125\\
1700	0.0136875\\
2150	0.007225\\
2600	0.0040625\\
3050	0.0020875\\
3500	0.00105\\
3950	0.0004875\\
4400	0.0002875\\
4850	0.00015\\
5300	7.5e-05\\
5750	2.5e-05\\
6200	3.75e-05\\
6650	1.25e-05\\
7100	0\\
7550	0\\
8000	1.25e-05\\
};
\addlegendentry{Wasserstein~\eqref{eq:portf:Wstein} with $r_\mathsf{W}^{(3)}$};

\addplot [color=green,dashed,line width = 1.5pt,smooth]
  table[row sep=crcr]{%
5	1\\
32	1\\
59	1\\
85	1\\
112	1\\
139	1\\
166	1\\
193	0.9999875\\
220	0.9998875\\
246	0.999775\\
273	0.9996125\\
300	0.9992875\\
350	0.9982375\\
800	0.923\\
1250	0.752\\
1700	0.5741125\\
2150	0.4287125\\
2600	0.3138\\
3050	0.23105\\
3500	0.1689375\\
3950	0.12335\\
4400	0.08945\\
4850	0.0656625\\
5300	0.048275\\
5750	0.03545\\
6200	0.0261\\
6650	0.01915\\
7100	0.0138875\\
7550	0.0104125\\
8000	0.0076125\\
};
\addlegendentry{Wasserstein~\eqref{eq:portf:Wstein} with $r_\mathsf{W}^{(1)}$};

\addplot [color=green,solid,line width = 1.5pt,smooth]
  table[row sep=crcr]{%
5	1\\
32	1\\
59	1\\
85	0.9999\\
112	0.9993625\\
139	0.99745\\
166	0.9930125\\
193	0.984075\\
220	0.9685375\\
246	0.9480125\\
273	0.919725\\
300	0.886\\
350	0.807675\\
800	0.2036625\\
1250	0.041725\\
1700	0.0083875\\
2150	0.0018375\\
2600	0.0003\\
3050	5e-05\\
3500	1.25e-05\\
3950	0\\
4400	0\\
4850	0\\
5300	0\\
5750	0\\
6200	0\\
6650	0\\
7100	0\\
7550	0\\
8000	0\\
};
\addlegendentry{Wasserstein~\eqref{eq:portf:Wstein} with $r_\mathsf{W}^{(2)}$};

% \addplot [color=black,dashed,line width = 1.0pt,smooth]
%   table[row sep=crcr]{%
% 0	0.1\\
% 15000	0.1\\
% };
% \addlegendentry{desired decay rate};

\end{axis}
\end{tikzpicture}%
%%% Local Variables:
%%% mode: latex
%%% TeX-master: "main"
%%% End:
 \label{fig:port:iid:prob} }   
\hspace{1mm}
\subfloat[Decay rate of out-of-sample disappointment]{% This file was created by matlab2tikz.
% Minimal pgfplots version: 1.3
%  http://www.mathworks.com/matlabcentral/fileexchange/22022-matlab2tikz
\definecolor{mycolor1}{rgb}{0.00000,0.44700,0.74100}%
\definecolor{mycolor2}{rgb}{0.85000,0.32500,0.09800}%
\begin{tikzpicture}

\begin{axis}[%
width=1.75in,
height=1.75in,
at={(1.011111in,0.641667in)},
scale only axis,
xmin=0,
xmax=8000,
ymin=0,
ymax=0.4,
xtick = {0,2000,4000,6000,8000},
xticklabels = {0,2{,}000,4{,}000,6{,}000,8{,}000},
ytick = {0,0.1,0.2,0.3,0.4},
yticklabels = {0$\%$,0.1$\%$,0.2$\%$,0.3$\%$,0.4$\%$},
xlabel={\footnotesize{$T$}},
ylabel={\footnotesize{$ -\frac{1}{T} \log \mb P_{\theta_\star} [ c(\widehat x_T, {\theta_\star}) > \widehat c_T(\widehat x_T) ]$}},
y label style={yshift=-0.25em},
scaled ticks=false, tick label style={/pgf/number format/fixed},
legend to name=named,
legend style={legend cell align=left,align=left,draw=white!15!black,legend columns=2}
]

% \addplot [color=blue,only marks,mark=*,
% mark options={scale=1.2, fill=blue},
% forget plot]
%   table[row sep=crcr]{%
% 0.0036	0.376\\
% };

% \addplot [color=red,only marks,mark=*,
% mark options={scale=1.2, fill=red},
% forget plot]
%   table[row sep=crcr]{%
% 0.0023	0.361\\
% };

% % trick for legend
% %%%%%%%%%%%%%%%%%%%%%%%%%
% \addplot [color=blue,style=solid,line width = 1.5pt]
%   table[row sep=crcr]{%
% 0.02	0.330505369033144\\
% 0.03	0.338864790100878\\
% };
% \addlegendentry{Out-of-sample cost $\mb E_{\theta_\star}[c(\widehat x_T,\theta_\star)]$ from i.i.d.~ambiguity set \eqref{eq:ex:iid:rate:function}};

% \addplot [color=red,style=solid,line width = 1.5pt]
%   table[row sep=crcr]{%
% 0.02	0.329037961337489\\
% 0.03	0.337708399250403\\
% };
% \addlegendentry{Out-of-sample cost $\mb E_{\theta_\star}[c(\widehat x_T,\theta_\star)]$ from optimal~ambiguity set \eqref{eq:ex:AR:rate:function}};

% SAA
\addplot [color=red,solid,line width = 1.5pt,smooth]
  table[row sep=crcr]{%
5	0\\
32	-0\\
59	-0\\
85	-0\\
112	-0\\
139	-0\\
166	-0\\
193	-0\\
220	-0\\
246	-0\\
273	4.5787831961985e-06\\
300	4.16669270854064e-06\\
350	-0\\
800	1.97030202046054e-04\\
1250	6.85932257611103e-04\\
1700	1.17564082873602e-03\\
2150	1.66009053860426e-03\\
2600	1.90506496273266e-03\\
3050	2.17114537442981e-03\\
3500	2.31176672158618e-03\\
3950	2.41544771657363e-03\\
4400	2.4736086371398e-03\\
4850	2.49854478044777e-03\\
5300	2.51002863363909e-03\\
5750	2.49958328755812e-03\\
6200	2.4979318239654e-03\\
6650	2.46779590036987e-03\\
7100	2.44897619734214e-03\\
7550	2.42154548807153e-03\\
8000	2.38156392680804e-03\\
};
\addlegendentry{SAA ($\varepsilon=0$)};

\addplot [color=blue,solid,line width = 1.5pt,smooth]
  table[row sep=crcr]{%
% 100	0\\
% 1756	0.0484961313957988\\
% 3411	0.0738111542835823\\
% 5067	0.083579768936739\\
% 6722	0.0885201220436895\\
% 8378	0.0926643756151809\\
% 10033	0.0908504952872706\\
% 11689	0.0906547586029265\\
5	0\\
32	-0\\
59	-0\\
85	-0\\
112	-0\\
139	-0\\
166	-0\\
193	6.47672441742068e-06\\
220	1.70457741557283e-05\\
246	8.6391293221149e-05\\
273	5.9528646357314e-05\\
300	1.1668708809846e-04\\
350	2.78707320479763e-04\\
800	8.89672743181561e-03\\
1250	0.0242588591502838\\
1700	0.0379327730254233\\
2150	0.0484995403884408\\
2600	0.0569759448988523\\
3050	0.0628881406990845\\
3500	0.0680426579456834\\
3950	0.0720184658232269\\
4400	0.0754013814712355\\
4850	0.0784737159997976\\
5300	0.0806258911275811\\
5750	0.0826876633503019\\
6200	0.0842528036056818\\
6650	0.0858088060883356\\
7100	0.0865932162979346\\
7550	0.087396477237442\\
8000	0.0887515896453699\\
};
\addlegendentry{optimal};

\addplot [color=red,dotted,line width = 1.5pt,smooth]
  table[row sep=crcr]{%
5	0\\
32	-0\\
59	-0\\
85	-0\\
112	-0\\
139	8.99286196087908e-06\\
166	4.51824172534218e-05\\
193	1.10115327110895e-04\\
220	2.72809123651118e-04\\
246	6.10213872466297e-04\\
273	1.07299859801353e-03\\
300	1.67503490550914e-03\\
350	3.36979425220327e-03\\
800	0.0439054116906116\\
1250	0.0852412848943136\\
1700	0.113044368579073\\
2150	0.132133375022934\\
2600	0.147015801797633\\
3050	0.157646058504249\\
3500	0.165386391517465\\
3950	0.174565386303382\\
4400	0.178540788844793\\
4850	0.181543819873567\\
5300	0.186858255708229\\
5750	0.196344033280974\\
};
\addlegendentry{SAA ($\varepsilon=0.02$) \quad};

\addplot [color=red,dashed,line width = 1.5pt,smooth]
  table[row sep=crcr]{%
5	0\\
32	-0\\
59	-0\\
85	-0\\
112	-0\\
139	-0\\
166	-0\\
193	6.47672441742068e-06\\
220	1.13637784114824e-05\\
246	5.08161842077823e-05\\
273	5.4949176236304e-05\\
300	1.00015003000664e-04\\
350	2.57258640904044e-04\\
800	7.24833724362669e-03\\
1250	0.0182952124257254\\
1700	0.0277897977752052\\
2150	0.0346569730071004\\
2600	0.0401722800033916\\
3050	0.0441978901312394\\
3500	0.0472397085703795\\
3950	0.0497343545212085\\
4400	0.0518842050506186\\
4850	0.0534385214757133\\
5300	0.0547764734388849\\
5750	0.0561004272445106\\
6200	0.0566788090925931\\
6650	0.0576962343487302\\
7100	0.0581103553424433\\
7550	0.0592311436053228\\
8000	0.0601217168966171\\
};
\addlegendentry{SAA ($\varepsilon=0.01$)\quad};

\addplot [color=green,dotted,line width = 1.5pt,smooth]
  table[row sep=crcr]{%
% 5	0\\
% 32	8.95813996343987e-03\\
% 59	0.100686531004304\\
% 85	0.24050486877372\\
% 112	0.365701864011132\\
139	0.451792562808812\\
166	0.501191750770706\\
193	0.527700463463876\\
220	0.538816553978637\\
246	0.540173846126651\\
273	0.538880746899738\\
300	0.531208693639338\\
350	0.507768887157902\\
800	0.351858764312048\\
1250	0.28774049321905\\
1700	0.252427780670907\\
2150	0.229312002092262\\
2600	0.211767566589472\\
3050	0.202353708237353\\
3500	0.195970431851792\\
3950	0.193068867532313\\
4400	0.185324720402883\\
4850	0.181543819873567\\
5300	0.179207970649584\\
5750	0.184289299706019\\
6200	0.164373703628837\\
6650	0.169771156596331\\
%7100	inf\\
%7550	inf\\
% 8000	0.1411222739207\\
};
\addlegendentry{Wasserstein ($\varepsilon=5\cdot 10^{-5}$) \qquad};

\addplot [color=green,solid,line width = 1.5pt,smooth]
  table[row sep=crcr]{%
5	0\\
32	-0\\
59	-0\\
85	1.17652941568644e-04\\
112	5.69377937078213e-04\\
139	1.83687538684697e-03\\
166	4.22411258998581e-03\\
193	8.31770223950669e-03\\
220	0.0145309441886733\\
246	0.0217022728292832\\
273	0.0306522222219552\\
300	0.0403461094590187\\
350	0.0610272940283093\\
800	0.198911383369133\\
1250	0.254132384754739\\
1700	0.281236045687314\\
2150	0.292992991947781\\
2600	0.311989541665695\\
3050	0.324704509919217\\
3500	0.322565197533029\\
};
\addlegendentry{Wasserstein ($\varepsilon=5\cdot 10^{-6}$)};

\addplot [color=green,dashed,line width = 1.5pt,smooth]
  table[row sep=crcr]{%
5	0\\
32	-0\\
59	-0\\
85	-0\\
112	-0\\
139	-0\\
166	-0\\
193	6.47672441742068e-06\\
220	5.11392402725497e-05\\
246	9.14737058120108e-05\\
273	1.41968900192603e-04\\
300	2.37584649585943e-04\\
350	5.04015723015577e-04\\
800	0.0100157555599106\\
1250	0.0228015164025838\\
1700	0.0326429358138433\\
2150	0.0393938952504185\\
2600	0.0445769014878156\\
3050	0.048036758743143\\
3500	0.0508064701412992\\
3950	0.0529804920508043\\
4400	0.0548653515688047\\
4850	0.0561490163443723\\
5300	0.0571856877484445\\
5750	0.0580805569739826\\
6200	0.0588035478169918\\
6650	0.0594804896745185\\
7100	0.0602361425917801\\
7550	0.0604602420064792\\
8000	0.0609745455743265\\
};
\addlegendentry{Wasserstein ($\varepsilon=5\cdot 10^{-7}$)$\quad$};

\addplot [color=black,dashed,line width = 1.0pt,smooth]
  table[row sep=crcr]{%
0	0.1\\
15000	0.1\\
};
\addlegendentry{desired decay rate $r$};

% Finite resolution (\log(10000)*100 = 1151.29254)
\addplot [name path=statb, draw=black!40, samples=100, domain=2000:8000] {1151.29254/x} ;
\addplot [name path=plotb, draw=black!20, domain=0:8000, samples=200]{0.5};
\addplot[black!20, pattern=north east lines wide] fill between[ 
of = statb and plotb, 
];

\end{axis}
\end{tikzpicture}%
%%% Local Variables:
%%% mode: latex
%%% TeX-master: "main"
%%% End:
\label{fig:port:iid:decay}} \newline
\vspace{2mm}

\ref{named}
\caption[]{Out-of-sample disappointment of different predictor-prescriptor pairs. %The SAA penalization parameters are $\gamma_1<\gamma_2$ and the Wasserstein radii satisfy $r_\mathsf{W}^{(1)}<r_\mathsf{W}^{(2)}<r_\mathsf{W}^{(3)}$. 
  All probabilities involving random training data are evaluated empirically using~$10^5$ independent training sets.
  The striped area in the right panel indicates the region where the decay rate $-\frac{1}{T} \log \mb P_{\theta_\star} [ c(\widehat x_T, {\theta_\star}) > \widehat c_T(\widehat x_T)]$ of the out-of-sample disappointment cannot be determined accurately because we expect to observe less than one disappointment event among all independent training sets under consideration.
}
\label{fig:portfolio:iid:setting}
\end{figure}
%%%%%%%%%%%%

% %%%%%
% \begin{figure}[t] 
% \centering
% {\input{iid_out_of_sample.tex}}
%  \newline
% \ref{named}
% \caption[]{Visualization of the out-of-sample cost, where the tubes show the 10$\%$ and 90$\%$ quantiles (shaded areas) and the means (solid lines) evaluated on $10^3$ independent training sets. The radius parameter of the DRO method is chosen as $r=0.1$.}
% \label{fig:portf:out-of-sample}
% \end{figure}

%%%%%%%% 
\begin{figure}[t] 
\centering
\includegraphics[width=1.0\textwidth]{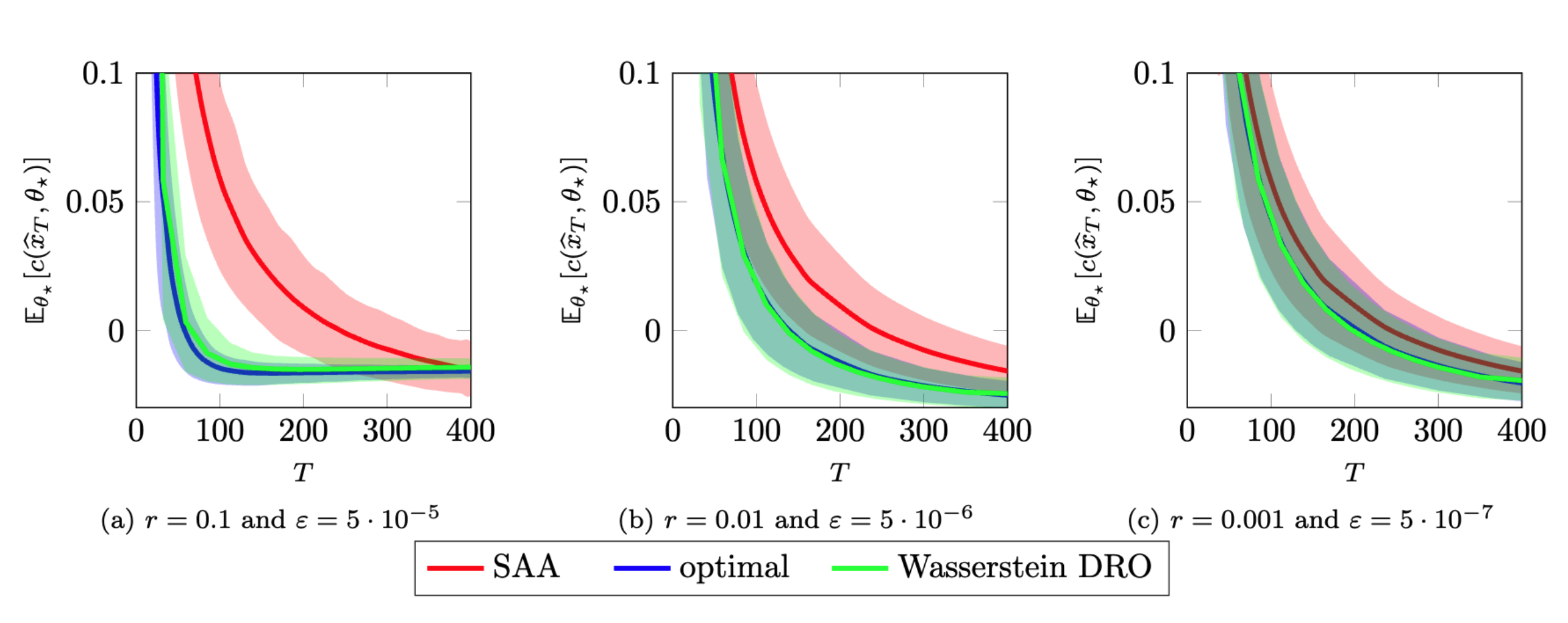}
%\subfloat[$r=0.1$ and $\varepsilon=5\cdot 10^{-5}$
%]{\input{iid_out_of_sample} \label{fig:oos:iid:r:0:1} }  
%\hspace{1mm}
%\subfloat[$r=0.01$ and $\varepsilon=5\cdot 10^{-6}$]{\input{iid_out_of_sample_r_0_01}\label{fig:oos:iid:r:0:01}} 
%\hspace{1mm}
%\subfloat[$r=0.001$ and $\varepsilon=5\cdot 10^{-7}$]{\input{iid_out_of_sample_r_0_001}\label{fig:oos:iid:r:0:001}} \newline 
%\vspace{2mm}
%\ref{named}
\caption[]{Visualization of the out-of-sample cost, where the tubes show the 10$\%$ and 90$\%$ quantiles (shaded areas) and the means (solid lines) evaluated on $10^3$ independent training sets.}
\label{fig:portf:out-of-sample}
\end{figure}

\section{Relation to classical efficiency concepts}\label{sec:classical:efficiency}
%%%%%%%%%%%%%%%%%%%%%%%%%%%%%%%%%%%%%%%%%%%%%%%%%%%%%%%%%%%%
The study of the fundamental performance limitations and the efficiency properties of various estimators has of course a long and distinguished history in statistics. In this appendix we highlight several connections between the Pareto dominance properties of data-driven predictors studied in Sections~\ref{sec:optimal:dd:prescriptors} and~\ref{sec:equivalence} and some classical efficiency concepts.

Any data-driven predictor~$\widehat c_T(x)$ can be regarded as an estimator for the cost~$c(x, \theta)$ of  a fixed decision~$x\in X$ under the probability measure~$\mathbb P_\theta$. If one is only interested in cost prediction, then the symmetric error probability
\(
  \mb P_\theta[(\widehat c_T( x)-c( x, \theta))^2> \varepsilon_T^2]
\)
for some prescribed error tolerance~$\varepsilon_T>0$ represents a more appropriate performance measure than the asymmetric out-of-sample disappointment introduced in Definition~\ref{def:oos-disappointment}. There is indeed a vast literature on quantifying the statistical efficiency of estimators based on how fast this error probability decays to zero as the sample size~$T$ grows. Estimators for which this decay is in some precise sense as fast as possible are designated as efficient. %At the same time estimators which are dominated in terms of their decay rate by a better estimator use the available information inefficiently.
There are two classical notions of efficiency that correspond to different asymptotic regimes of the error tolerance~$\varepsilon_T$.

The Cram\'er-Rao inequality guarantees that the variance of $T^{\frac{1}{2}}(\widehat c_T(x)-c( x, \theta))$ is bounded below by the inverse Fisher information whenever~$\widehat c_T(x)$ represents an unbiased estimator for~$c(x,\theta)$ and some standard regularity conditions are met. %That is, $\E{\theta}{(\widehat c_T(x)-c( x, \theta))^2}\geq \frac 1T i(\theta)$ for any unbiased estimator $\widehat c_T$ and where $i(\theta)$ is denoted as the Fischer information. 
Estimators that attain this bound asymptotically are termed \textit{relatively Pitman efficient} \cite{kester1985some}. For such estimators the error probability 
\(
\mb P_\theta[(\widehat c_T( x)-c( x, \theta))^2> \varepsilon_T^2]
\)
can be guaranteed to remain uniformly small if the error tolerances decay as~$\varepsilon_T=\mc O(T^{-\frac 12})$.

When focusing on constant error tolerances $\varepsilon_T=\varepsilon>0$, on the other hand, then the error rate
\begin{equation*}
  e(\varepsilon, \theta, \widehat c(x)) = \limsup_{T\to\infty}  \frac{1}{T} \log \mb P_\theta \left[ (\widehat c_T(x)-c(x, \theta))^2>\varepsilon^2 \right]
\end{equation*}
may be used as an appropriate yardstick for comparing estimators \cite{basu1956concept,bahadur1960stochastic}. %; clearly the smaller this error rate the more preferable the estimator.
% Observe that the error rate of an estimator is in general a function of both the considered error size $\varepsilon>0$ as well as the unknown model $\theta\in \Theta$. 
Bahadur proved under standard regularity conditions that the error rate $e(\varepsilon, \theta, \widehat c(x))$ of any consistent estimator $\widehat c(x)$ is bounded below by a function $b(\varepsilon, \theta)$ and thus established a constant error counterpart to the Cram\'er-Rao bound \cite{bahadur1960stochastic}.
%As $e(\varepsilon, \theta, \widehat c(x))$ and $b(\varepsilon, \theta)$ are difficult to evaluate and 
As small error tolerances are particularly important, it is sometimes reasonable to measure the quality of an estimator by its error rate in the limit when $\varepsilon$ tends to~$0$. Accordingly, an estimator is called \textit{locally Bahadur efficient}~if
\[
  \lim_{\varepsilon\to 0} \frac{e(\varepsilon, \theta, \widehat c(x))}{b(\varepsilon, \theta)}=1 \quad \forall \theta\in \Theta.
\]
Such estimators attain Bahadur's lower bound for small values of $\varepsilon$.
% Naturally, this begs the question whether there exist estimators which are optimal in a stronger sense. 
Similarly, an estimator is called \textit{globally Bahadur efficient} if $e(\varepsilon, \theta, \widehat c(x))=b(\varepsilon, \theta)$ for all~$\varepsilon>0$ and~$\theta\in\Theta$. As~$e(\varepsilon, \theta, \widehat c(x))$ is never smaller than~$b(\varepsilon, \theta)$, such an estimator constitutes a Pareto dominant solution of the multi-objective optimization problem
\begin{equation*}
  \displaystyle\mathop{\text{minimize}}_{\widehat c}{} \left\{ e(\varepsilon, \theta, \widehat c(x))  \right\}_{\varepsilon>0,\, \theta\in\Theta},
\end{equation*}
which is reminiscent of~\eqref{eq:optimal-predictor}. A globally efficient estimator, should it exist, enjoys an optimal error rate for constant errors of any size $\varepsilon>0$.
As most multi-objective optimization problems admit no Pareto dominant solutions, the existence of Bahadur efficient estimators can not be taken for granted. They are in fact only known to exist if the ambiguity set~$\mathcal P$ constitutes an exponential family, and there is strong evidence suggesting that they do not exist for more general ambiguity sets \cite{kester1985some}. These findings are in line with the strong optimality results presented in Section~\ref{sec:equivalence}, which also require~$\mathcal P$ to represent an exponential family.

\end{document}